\documentclass[a4paper,11pt]{amsart}

\usepackage[a4paper,hmargin={2cm,2cm},vmargin={2.5cm,2.5cm}]{geometry}

\usepackage{txfonts}
\usepackage[square]{natbib} 
\usepackage[leqno]{amsmath}
\usepackage{amssymb,amsthm}
\usepackage[mathscr]{euscript}
\usepackage{bbm}
\usepackage{dsfont}
\usepackage[all,arc]{xy}
\DeclareMathAlphabet{\mathpzc}{OT1}{pzc}{m}{it}

\theoremstyle{plain}
\newtheorem{theorem}{Theorem}[section]
\newtheorem{lemma}[theorem]{Lemma}
\newtheorem{proposition}[theorem]{Proposition}
\newtheorem{corollary}[theorem]{Corollary}

\theoremstyle{definition}

\newtheorem{examples}[theorem]{Examples}
\newtheorem{example}[theorem]{Example}

\theoremstyle{remark}
\newtheorem{remark}[theorem]{Remark}

\newenvironment{eqcond}{\begin{enumerate}}{\end{enumerate}}

\newenvironment{sideremark}{\begin{remark}}{\end{remark}} 

\newcommand{\Rw}{\Rightarrow}

\newcommand{\Lw}{\Leftarrow}

\newcommand{\hrw}{\hookrightarrow}

\newcommand{\fra}{\mathfrak{a}}

\newcommand{\frf}{\mathfrak{f}}
\newcommand{\frg}{\mathfrak{g}}
\newcommand{\frj}{\mathfrak{j}}
\newcommand{\frp}{\mathfrak{p}}
\newcommand{\frq}{\mathfrak{q}}
\newcommand{\fru}{\mathfrak{u}}
\newcommand{\frw}{\mathfrak{w}}

\newcommand{\frx}{\mathfrak{x}}
\newcommand{\fry}{\mathfrak{y}}
\newcommand{\frz}{\mathfrak{z}}

\newcommand{\calA}{\mathcal{A}}
\newcommand{\calB}{\mathcal{B}}
\newcommand{\calC}{\mathcal{C}}

\newcommand{\calO}{\mathcal{O}}

\newcommand{\frW}{\mathfrak{W}}
\newcommand{\frX}{\mathfrak{X}}
\newcommand{\frY}{\mathfrak{Y}}

\DeclareMathOperator{\ev}{ev}
\DeclareMathOperator{\cl}{cl}
\DeclareMathOperator{\interior}{int}
\DeclareMathOperator{\can}{can}
\DeclareMathOperator{\colim}{colim}
\DeclareMathOperator{\yoneda}{\mathpzc{y}}
\DeclareMathOperator{\yonedaT}{\mathpzc{Y}}
\DeclareMathOperator{\yonmult}{\mathpzc{m}}
\DeclareMathOperator{\trans}{\mathpzc{t}}
\DeclareMathOperator{\Sup}{Sup}
\DeclareMathOperator{\upc}{\uparrow\!}
\DeclareMathOperator{\downc}{\downarrow\!}
\DeclareMathOperator{\kar}{kar}
\DeclareMathOperator{\Fix}{Fix}
\DeclareMathOperator{\Pt}{Pt}
\DeclareMathOperator{\Spl}{Spl}
\DeclareMathOperator{\map}{map}

\newcommand{\mate}[1]{\,^\ulcorner\! #1^\urcorner}

\newcommand{\fspstr}[2]{\llbracket #1,#2\rrbracket}

\newcommand{\catfont}[1]{\mathsf{#1}}

\newcommand{\V}{\catfont{V}}

\newcommand{\two}{\catfont{2}}

\newcommand{\SET}{\catfont{Set}}
\newcommand{\NREL}{\catfont{NRel}}
\newcommand{\REL}{\catfont{Rel}}
\newcommand{\UREL}{\catfont{URel}}
\newcommand{\MOD}{\catfont{Mod}}
\newcommand{\TOP}{\catfont{Top}}
\newcommand{\CDTOP}{\catfont{CDTop}}
\newcommand{\TATOP}{\catfont{TATop}}
\newcommand{\SOB}{\catfont{Sob}}
\newcommand{\AP}{\catfont{App}}
\newcommand{\CDAP}{\catfont{CDApp}}
\newcommand{\TAAP}{\catfont{TAApp}}
\newcommand{\ASOB}{\catfont{ASob}}
\newcommand{\FRM}{\catfont{Frm}}
\newcommand{\SLAT}{\catfont{SLat}}

\newcommand{\SUP}{\catfont{Sup}}
\newcommand{\ORD}{\catfont{Ord}}
\newcommand{\TAL}{\catfont{TAL}}
\newcommand{\CCD}{\catfont{CCD}}
\newcommand{\MET}{\catfont{Met}}
\newcommand{\CDMET}{\catfont{CDMet}}

\newcommand{\ORDCH}{\catfont{OrdCompHaus}}
\newcommand{\METCH}{\catfont{MetCompHaus}}
\newcommand{\COCTS}{\catfont{Cocts}}
\newcommand{\ContMet}{\catfont{ContMet}}

\newcommand{\Mat}[1]{#1\text{-}\catfont{Rel}}
\newcommand{\Mod}[1]{#1\text{-}\catfont{Mod}}

\newcommand{\Cat}[1]{#1\text{-}\catfont{Cat}}
\newcommand{\Cocts}[1]{#1\text{-}\catfont{Cocts}}
\newcommand{\Inf}[1]{#1\text{-}\catfont{Inf}}
\newcommand{\DTop}[1]{#1\text{-}\catfont{DTop}}
\newcommand{\DApp}[1]{#1\text{-}\catfont{DApp}}
\newcommand{\ATop}[1]{#1\text{-}\catfont{ATop}}
\newcommand{\AApp}[1]{#1\text{-}\catfont{AApp}}
\newcommand{\Sob}[1]{#1\text{-}\catfont{Sob}}

\newcommand{\relto}{{\longrightarrow\hspace*{-2.8ex}{\mapstochar}\hspace*{2.6ex}}}
\newcommand{\modto}{{\longrightarrow\hspace*{-2.8ex}{\circ}\hspace*{1.2ex}}}
\newcommand{\kto}{\relbar\joinrel\rightharpoonup}
\newcommand{\krelto}{\,{\kto\hspace*{-2.5ex}{\mapstochar}\hspace*{2.6ex}}}
\newcommand{\kmodto}{\,{\kto\hspace*{-2.8ex}{\circ}\hspace*{1.3ex}}}
\newcommand{\xrelto}[1]{{\xrightarrow{#1}\hspace*{-2.8ex}{\mapstochar}\hspace*{2.8ex}}}

\newcommand{\kleisli}{\circ}

\newcommand{\blackright}{\mbox{ $-\!\mbox{\footnotesize $\bullet$}$ }}
\newcommand{\blackleft}{\mbox{ $\mbox{\footnotesize $\bullet$}\!-$ }}

\newcommand{\whiteleft}{\mbox{ $\mbox{\footnotesize $\circ$}\!-$ }}

\newcommand{\monadfont}[1]{\mathbbm{#1}}

\newcommand{\mT}{\monadfont{T}}

\newcommand{\mU}{\monadfont{U}}
\newcommand{\mPf}{\monadfont{B}}

\newcommand{\mM}{\monadfont{M}}

\newcommand{\mPsh}{\monadfont{P}}
\newcommand{\mPhi}{\monadfont{I}^\Phi}
\newcommand{\mPhidiscrete}{\monadfont{J}^\Phi}

\newcommand{\monad}{(T,e,m)}
\newcommand{\monadb}{(T',e',m')}

\newcommand{\umonad}{(U,e,m)}

\newcommand{\Phimonad}{(\Phi,\yoneda^{\Phi},\yonmult^\Phi)}

\newcommand{\doo}[1]{\overset{\centerdot}{#1}}
\newcommand{\eps}{\varepsilon}
\newcommand{\op}{\mathrm{op}}
\newcommand{\sep}{\mathrm{sep}}

\newcommand{\true}{\mathsf{true}}
\newcommand{\false}{\mathsf{false}}

\newcommand{\field}[1]{\mathds{#1}}

\newcommand{\N}{\field{N}}

\newcommand{\R}{\field{R}}

\usepackage[hypertexnames=false]{hyperref} 

\begin{document}

\title{Duality for distributive space}

\author{Dirk Hofmann}
\thanks{Partial financial assistance by by Centro de Investiga\c{c}\~ao e Desenvolvimento em Matem\'atica e Aplica\c{c}\~oes da Universidade de Aveiro/FCT and the project MONDRIAN (under the contract PTDC/EIA-CCO/108302/2008) is gratefully acknowledged.}
\address{Departamento de Matem\'{a}tica\\ Universidade de Aveiro\\3810-193 Aveiro\\ Portugal}
\email{dirk@ua.pt}
\date{} 

\begin{abstract}
The main source of inspiration for the present paper is the work of R.~Rosebrugh and R.J.~Wood on constructive complete distributive lattices where the authors employ elegantly the concepts of adjunction and module in their study of ordered sets. Both notions (suitably adapted) are available in topology too, which permits us to investigate topological, metric and other kinds of spaces in a similar spirit. Therefore, relative to a choice $\Phi$ of modules, we consider spaces which admit all colimits with weight in $\Phi$, as well as (suitably defined) $\Phi$-distributive and $\Phi$-algebraic spaces. We show that the category of $\Phi$-distributive spaces and $\Phi$-colimit preserving maps is dually equivalent to the idempotent splitting completion of a category of spaces and convergence relations between them. We explain the connection of these results to the traditional duality of spaces with frames, and conclude further duality theorems. Finally, we study properties and structures of the resulting categories, in particular monoidal (closed) structures.

\end{abstract}

\maketitle

\setcounter{tocdepth}{1}
\tableofcontents

\section*{Introduction}

The work presented in this paper grew out of a simple observation regarding the well-known adjunction
\[
\xymatrix@C=3em{\ORD\ar@<1.35ex>[r] & \CCD^\mathrm{op}\ar@<1ex>[l]_\bot}
\]
between the category $\ORD$ of ordered sets and monotone maps and the dual of the category $\CCD$ of (constructively) completely distributive lattices and left and right-adjoint monotone maps. The functor $\ORD\to\CCD^\mathrm{op}$ can be constructed by either sending an ordered set $X$ to the set $\mathrm{Down}(X)\cong\ORD(X^\mathrm{op},2)$ of all down-sets of $X$ or to the set $\mathrm{Up}(X)\cong\ORD(X,2)$ of all up-sets of $X$. The dual adjunction between $\TOP$ and $\FRM$ can be seen as an extension of $\ORD\rightleftarrows\CCD^\op$ to topological spaces; however, this is only really true for the second construction. The first one does not even seem to make sense for topological spaces since it is not clear what $X^\op$ means now. But our recent study of ``spaces as categories'' required such a notion anyway, and since \citep{CH_Compl} we have a candidate which so far proved to be useful. Therefore we ask here (Sections \ref{SectionDuality} and \ref{SectionFramesDist}) about the construction $X\mapsto\TOP(X^\op,\two)$, and the answer leads to a scenario which seems to be even closer to the $\ORD$-case than the ``usual'' dual adjunction with frames.

As it is well-known, the adjunction between $\ORD$ and $\CCD$ restricts to a dual equivalence between $\ORD$ and the full subcategory $\TAL$ of $\CCD$ defined by the totally algebraic lattices. This equivalence is actually the restriction of a larger one, in \citep{RW_CCD4} R.~Rosebrugh and R.J.~Wood showed that the category $\CCD_{\sup}$ of constructive complete distributive lattices and suprema preserving maps is equivalent to the idempotent splitting completion of the category $\REL$ of sets and relations. This theorem turned out to be very powerful since it synthesises many facts about complete distributive lattices, implies various known duality theorems in lattice theory (for example, $\ORD^\op\cong\TAL$ as well as $\SET^\op\cong\mathsf{CABool}$ follow easily), and allows to transfer nice properties and structures from $\REL$ to $\CCD_{\sup}$. Later on, in \citeyear{RW_Split}, \citeauthor{RW_Split} observed that this theorem is not really about lattices but rather a special case of a much more general result about ``a mere monad $D$ on a mere category $\mathsf{C}$ where idempotents split''. More precise, they show that the idempotent splitting completion of the Kleisli category of $D$ is equivalent to the category of split Eilenberg-Moore algebras for $D$ (see Section \ref{SectionRelative}). The equivalence above appears now for both the power-set monad on $\SET$ and the down-set monad on $\ORD$, and further interesting results one obtains by considering submonads of the down-set monad on $\ORD$. More importantly for us, this result paves the road towards similar results for topological, metric and approach spaces. In fact, we argue here that many applications of \citep{RW_Split} can be found in topology since many interesting classes of spaces can be described as algebras for certain monads. For instance, compact Hausdorff spaces are the algebras for the ultrafilter monad on $\SET$, continuous lattices are the algebras for the filter monad on $\SET$, $\ORD$ and $\TOP$, and stably compact spaces are the algebras for the prime filter monad on $\ORD$ and $\TOP$. 

One might wonder at this point what kind of monads on, say, metric spaces correspond to the filter monad. This brings us to our second concern which is the search for a metric counterparts of domain-theoretic notions. This and related questions came into life thanks to the observation (due to Hausdorff, but see \citep{Law_MetLogClo}) that a metric $d:X\times X\to[0,\infty]$ can be seen as generalised order relation where one trades the Boolean algebra $\two=\{\false,\true\}$ for the quantale $[0,\infty]$. In fact, many order theoretic notions can be appropriately translated into the metric context, for instance
\begin{itemize}
\item a non-empty (up-closed) subset of $X$ can be identified with a (monotone) map $\varphi:X\to\two$ satisfying $\exists x\in X\,.\,\varphi(x)$; in a metric space we would now talk about a (contraction) map $\varphi:X\to[0,\infty]$ with $\inf_{x\in\varphi}\varphi(x)=0$;
\item a subset $\varphi:X\to\two$ is directed if it is non-empty and, for all $x,y\in X$,
\[
\varphi(x)\,\&\,\varphi(y)\,\Rw\,\exists z\in X\,.\,(x\le z \,\&\, y\le z \,\&\, \varphi(z));
\]
which in the metric world could be written as 
\[
\varphi(x)\,+\,\varphi(y)\,\geqslant\,\inf_{z\in X}(d(x,z)+d(y,z)+\varphi(z)).
\]
\end{itemize}
Hence, the notion of order ideal and eventually the order theoretic definition of continuous lattice can be brought into the realm of metric spaces. These analogies led indeed to a many interesting results, see for instance \citep{Was_DomGirQant,KW_LimColimCoind} and \citep{Wag_PhD}. But continuous lattices live at the border between order, topology and algebra; they are also known to be precisely the injective topological T$_0$-spaces and the algebras for the filter monad. Therefore we take here injectivity as primitive notion and define ``continuous metric space'' as an injective space. Of course, space cannot mean topological space here, we have to consider a $[0,\infty]$-variant of the definition of topological space. Fortunately, such a notion was already introduced in the 80's under the name approach space by \citeauthor{Low_Ap} (\citeyear{Low_Ap}), and these spaces are extensively described in his \citeyear{Low_ApBook}-book. We also remark that the use of approach spaces in quantitative domain theory was already advocated in \citep{Win_Solve,Win_ScottApp}. Since \citep{Hof_Cocompl} we know that injective T$_0$-approach spaces can be described as cocomplete T$_0$-approach spaces, and that together with colimit preserving maps they form a monadic category over $\SET$ and $\MET$. The latter result provides us with a monad which takes the role of the filter monad in this quantitative setting. In Section \ref{SectionContMetSp} we have a closer look at the algebras for this monad, showing in particular that they define a Cartesian closed subcategory of $\AP$. In Section \ref{SectionRelative} we apply the techniques of \citep{RW_CCD4,RW_Split} to submonads of the (approach) filter monad. Finally, in Section \ref{SectionExamples} we discuss examples.

The work we present here was developed in the context of $(\mT,\V)$-categories where $\mT$ and $\V$ are part of a strict topological theory as described in \citep{Hof_TopTh}. However, we decided to stay here in the more familiar context of topological, metric and approach spaces since we feel that the huge amount of special notations needed in the general case makes the actual results less accessible. We stress that most of our results can be derived for $(\mT,\V)$-categories in general, just a few are indeed only valid for metric or approach spaces. We will indicate whenever there are such restrictions. In Section \ref{SectTopAppCat} we recall the convergence-relational approach to topological and approach spaces, which is the context where ``spaces look like categories''. Section \ref{SectComplOrdSet} presents basic facts about ordered sets in the language of modules and adjunction, and Section \ref{SectMetSp} recalls Lawvere's \citep{Law_MetLogClo} view on metric spaces as enriched categories. In Section \ref{SectionDualSpace} we define the notion of dual spaces. Our approach differs here slightly from previous work \citep{CH_Compl}. In Section \ref{SectCocomplSp} we recall the main results on cocomplete spaces of \citep{Hof_Cocompl,CH_CocomplII}.


Finally, some warnings:
\begin{enumerate}
\item The underlying order of a topological space $X$ we define as
\[
 x\le y\hspace{1em}\text{whenever }\doo{x}\to y,
\]
which is the \emph{dual} of the specialisation order. We do so because we wish to think of the underlying order as the ``point shadow'' of the convergence relation.
\item We consider here the Sierpi\'nski space $\two=\{0,1\}$ with $\{1\}$ closed. This is compatible with the point above since the underlying order gives $0\le 1$, but note that $\varphi:X\to\two$ is the characteristic map of a closed subset.
\item In general we try to avoid imposing separation axioms: our topological spaces need not be T$_0$, our ordered sets need not be anti-symmetric, and so on. This is usually harmless but creates some ``pseudo-issues'' since many notions are only unique up to equivalence. 
\end{enumerate}

\section{Topological and approach spaces as categories}\label{SectTopAppCat}

First we to recall how a topological space can be viewed as a category. The principal idea is to think of the convergence $\mathfrak{x}\to x$ of an ultrafilter $\mathfrak{x}$ on $X$ to a point $x$ in $X$ as a morphism in $X$, so that the convergence relation
\begin{equation*}
UX\times X\to\mathsf{2}
\end{equation*}
becomes the ``hom-functor'' of $X$. Such a relation is the convergence relation of a (unique) topology on $X$ if and only if (see \citep{Bar_RelAlg})
\begin{align}\label{AxiomsTop}
e_X(x)\to x &&\text{and}&& (\mathfrak{X}\to\mathfrak{x}\;\&\;\mathfrak{x}\to x)\,\Rw\, m_X(\mathfrak{X})\to x,
\end{align}
for all $x\in X$, $\mathfrak{x}\in UX$ and $\mathfrak{X}\in UUX$, where $e_X(x)=\doo{x}$ the principal ultrafilter generated by $x\in X$ and
\begin{align*}
m_X(\frX)=\{A\subseteq X\mid A^\#\in\frX\} &&(A^\#=\{\frx\in UX\mid A\in\frx\}).
\end{align*}
The first arrow of \eqref{AxiomsTop} one might see as an identity on $x$, and the second condition of \eqref{AxiomsTop} one might interpret as the existence of a ``composite'' of ``composable pairs of arrows''. Furthermore, a function $f:X\to Y$ between topological spaces is continuous whenever $\mathfrak{x}\to x$ in $X$ implies $f(\mathfrak{x})\to f(x)$ in $Y$, that is, $f$ associates to each object in $X$ an object in $Y$ and to each arrow in $X$ an arrow in $Y$ between the corresponding (ultrafilter of) objects in $Y$. As usual, $\TOP$ denotes the category of topological spaces and continuous maps.

Note that the second condition of \eqref{AxiomsTop} talks about the convergence of an ultrafilter of ultrafilters $\frX$ to an ultrafilter $\frx$, which comes from applying the ultrafilter functor $U$ to the \emph{relation} $a:UX\relto X$. In general, for a relation $r:X\relto Y$ from $X$ to $Y$ and ultrafilters $\frx\in UX$ and $\fry\in UY$ one puts
\[
\frx\,(Ur)\,\fry \;\;\; \text{ if }\;\;\; \forall A\in\frx,B\in\fry\,\exists x\in A,y\in B\,.\,x\,r\, y,
\]
and obtains this way an extension of the $\SET$-functor $U$ to a functor $U:\REL\to\REL$ which, moreover, satisfies $U(r^\circ)=(Ur)^\circ$ (where $r^\circ:Y\relto X$ is defined as $y\,r^\circ\, x$ whenever $x\,r\,y$) and $Ur\subseteq Us$ whenever $r\subseteq s$. Furthermore, the multiplication $m$ is still a natural transformation $m:UU\to U$, but $e:1\to U$ satisfies only $e_Y\cdot r\subseteq Ur\cdot e_X$ for any relation $r:X\relto Y$.

To describe approach spaces, it is only necessary to trade relation for \emph{numerical relation}: $r:X\relto Y$ stands now for $r:X\times Y\to[0,\infty]$. We sketch here very briefly this construction which can be found in \citep{CH_TopFeat}, and for questions concerning approach spaces in general we refer to \citep{Low_ApBook}. Given also $s:Y\relto Z$, one can calculate the composite $s\cdot r:X\relto Z$ by the formula
\begin{align}\label{CompNREL}
s\cdot r(x,z)&=\inf_{y\in Y}(r(x,y)+s(y,z)).
\end{align}
Each relation becomes a numerical relation by interpreting $\true$ as $0$ and $\false$ as $\infty$, and with this interpretation the identity function is also the identity numerical relation. Taking into account the opposite of the pointwise order on the set of all numerical relations from $X$ to $Y$, one obtains the ordered category $\NREL$ of sets and numerical relations. The ``turning around'' of the natural order of $[0,\infty]$ has its roots in the translation of ``$\false\le\true$'' in $\two$ to ``$\infty\geqslant 0$'' in $[0,\infty]$. Due to this switch ``$\exists$'' becomes ``$\inf$'' in \eqref{CompNREL}, but note also that ``$\&$'' is replaced by ``$+$''. Implication $x\Rw-:\two\to\two$ is right adjoint to $x\,\&\,-:\two\to\two$ for $x\in\two$; similarly, for $x\in[0,\infty]$, the map ``addition with $x$'' $x+-:[0,\infty]\to[0,\infty]$ has a right adjoint, namely $\hom(x,-):[0,\infty]\to[0,\infty],\;y\mapsto\max\{y-x,0\}$.

As above, the ultrafilter functor $U$ extends to $U:\NREL\to\NREL$ (with the same properties as in the topological case) via
\[
Ur(\frx,\fry)=\sup_{A\in\frx,B\in\fry}\inf_{x\in A,y\in B}r(x,y)
\]
for a numerical relation $r:X\times Y\to[0,\infty]$. We remark that a different but equivalent formula defining the extension of $U$ to $\NREL$ was used in \citep{CH_TopFeat}, the one above is taken from \citep{CT_MultiCat}.
\begin{sideremark}\label{RemarkXi}
Thinking of a relation $r:X\relto Y$ as a subset $R\subseteq X\times Y$, it is not hard to see that
\[
 \frx\,(Ur)\,\fry\iff \exists\frw\in U(X\times Y)\,.\,U\pi_1(\frw)=\frx\,\&\,U\pi_2(\frw)=\fry
\]
for all $\frx\in UX$ and $\fry\in UY$. Similarly, for a numerical relation $r:X\relto Y$ one has
\[
Ur(\frx,\fry)=\inf\{\xi\cdot Ur(\frw)\;\Bigl\lvert\;\frw\in U(X\times Y), T\pi_1(\frw)=\frx,T\pi_2(\frw)=\fry\},
\]
where $\xi:U[0,\infty]\to[0,\infty],\,\fru\mapsto\sup_{A\in\fru}\inf A$. The notation here is a bit ambiguous since $Ur$ appears on both sides, but on the ride hand side it stands for the functions $Ur:U(X\times Y)\to U[0,\infty]$. We use the occasion to mention that $\xi:U[0,\infty]\to[0,\infty]$ is actually a $\mU$-algebra structure on $[0,\infty]$, that is, a compact Hausdorff topology. Furthermore, $[0,\infty]$ is a monoid in in the category of compact Hausdorff spaces and continuous maps in two different ways since both $+:[0,\infty]\times[0,\infty]\to[0,\infty]$ and $\max:[0,\infty]\times[0,\infty]\to[0,\infty]$ are continuous. It is useful to observe that continuity of $+$ and $\max$ mean precisely that the diagrams
\begin{align*}
\xymatrix{U([0,\infty]\times[0,\infty])\ar[rr]^-{U(+)}\ar[d]_{\langle\xi\cdot U\pi_1,\xi\cdot U\pi_2\rangle} && U[0,\infty]\ar[d]^\xi\\
[0,\infty]\times[0,\infty]\ar[rr]_-{+} && [0,\infty]}
&&&&
\xymatrix{U([0,\infty]\times[0,\infty])\ar[rr]^-{U(\max)}\ar[d]_{\langle\xi\cdot U\pi_1,\xi\cdot U\pi_2\rangle} && U[0,\infty]\ar[d]^\xi\\
[0,\infty]\times[0,\infty]\ar[rr]_-{\max} && [0,\infty]}
\end{align*}
commute. Note also that $\xi$ is compatible with the map $\hom:[0,\infty]\times[0,\infty]\to[0,\infty],\,(x,y)\mapsto\hom(x,y)=\max\{y-x,0\}$ in the sense that $\xi\cdot U(\hom)\geqslant\hom\cdot\langle\xi\cdot U\pi_{_1},\xi\cdot U\pi_{_2}\rangle$.
\[
\xymatrix{
U([0,\infty]\times[0,\infty])\ar[rr]^-{U(\hom)}\ar[d]_{\langle\xi\cdot U\pi_{_1},\xi\cdot U\pi_{_2}\rangle}\ar@{}[drr]|{\leqslant} && U[0,\infty]\ar[d]^\xi\\ [0,\infty]\times[0,\infty]\ar[rr]_-{\hom} && [0,\infty]}
\]
\end{sideremark}

An \emph{approach space} can be described as a pair $(X,a)$ consisting of a set $X$ and a numerical relation $a:UX\relto X$ satisfying
\begin{align}\label{AxiomsApp}
0\geqslant a(\doo{x},x) &&\text{and}&& Ua(\frX,\frx)+a(\frx,x)\geqslant a(m_X(\frX),x),
\end{align}
and a mapping $f:X\to Y$ between approach spaces $X=(X,a)$ and $Y=(Y,b)$ is a \emph{contraction} whenever $a(\frx,x)\geqslant b(Uf(\frx),f(x))$ for all $\frx\in UX$ and $x\in X$. Approach spaces and contraction maps are the main ingredients of the category $\AP$.

There is a canonical forgetful functor $\AP\to\TOP$ sending an approach space $(X,a)$ to the topological space with the same underlying set $X$ and with the convergence relation
\[
\frx\to x \text{ whenever } a(\frx,x)=0.
\]
This functor has a left adjoint $\TOP\to\AP$ which one obtains by interpreting the convergence relation of a topological space as a numerical relation. 
\begin{sideremark}
The left adjoint functor $\TOP\to\AP$ has a further left adjoint which can be obtained by first sending an approach space $(X,a)$ to the pseudotopological space $X$ with convergence
\[
\frx\to x\text{ whenever } a(\frx,x)<\infty,
\]
and then taking its topological reflection.
\end{sideremark}

The pointfree calculus of (numerical) relations allows for a simultaneous treatment of topological and approach spaces emphasising their common nature. For instance, both axioms \eqref{AxiomsTop} and \eqref{AxiomsApp} read as
\begin{align}\label{AxiomsLaxAlg}
\xymatrix{X\ar[r]^{e_X}\ar[dr]_{1_X}^\sqsubseteq & UX\ar[d]|-{\object@{|}}^a\\ & X}
&&
\xymatrix{UUX\ar[r]^{m_X}\ar[d]|-{\object@{|}}_{Ua} & UX\ar[d]|-{\object@{|}}^a\\
 UX\ar@{}[ur]|{\sqsubseteq}\ar[r]|-{\object@{|}}_a & X}\\
1_X\sqsubseteq a\cdot e_X && a\cdot Ua\sqsubseteq a\cdot m_X\notag
\end{align}
where $\sqsubseteq$ stands either for $\subseteq$ or $\geqslant$. Since $f:X\to Y$ is continuous respectively contractive if and only if
\[
\xymatrix{UX\ar[d]|-{\object@{|}}_a\ar[r]^{Uf} & UY\ar[d]|-{\object@{|}}^b\\ X\ar[r]_f\ar@{}[ur]|{\sqsubseteq} & Y,}
\]
we can think of $\TOP$ and $\AP$ as categories of lax Eilenberg--Moore algebras. 
Using the fact that $m_X\dashv m_X^\circ$ and $e_X\dashv e_X^\circ$ in $\REL$\footnote{Since $\REL$ is an ordered category (there is an order relation on hom-sets compatible with composition), it makes sense to talk about adjunction. One easily sees that a relation $r:X\relto Y$ is a function if and only if $1_X\le r^\circ\cdot r$ and $r\cdot r^\circ\le 1_Y$, i.e.\ if $r\dashv r^\circ$.} (and hence in $\NREL$), one can express the axioms \eqref{AxiomsLaxAlg} as
\begin{align*}
e_X^\circ\sqsubseteq a &&\text{and}&& a\cdot Ua\cdot m_X^\circ\sqsubseteq a.
\end{align*}
In this context it is useful to think of a (numerical) relation $a:UX\relto X$ as an endomorphism $a:X\krelto X$, and, more general, of $r:UX\relto Y$ as an arrow $r:X\krelto Y$, called \emph{$\mU$-relation} in the sequel. Given also $s:Y\krelto Z$, one can compose $s$ and $r$ using (a variant of) \emph{Kleisli composition}:
\[
s\kleisli r:=s\cdot Ur\cdot m_X^\circ.
\]
The (numerical) relation $e_X^\circ:UX\relto X$ behaves almost as an identity arrow $X\krelto X$ since
\begin{align*}
r\kleisli e_X^\circ=r &&\text{and}&& e_Y^\circ \kleisli r\sqsupseteq r.
\end{align*}
We can now restate the second condition above as $a\kleisli a\sqsubseteq a$, or even as $a\kleisli a=a$ thanks to the first condition.

\begin{sideremark}
One calls a $\mU$-relation $r:X\krelto Y$ \emph{unitary} if $e_Y^\circ \kleisli r=r$, see \citep{Hof_ExpUnit}. These relations are not completely unfamiliar to topologists: a reflexive (numerical) relation $a:UX\relto X$ is a pretopology (preapproach structure) precisely if $a:X\krelto X$ is unitary.
\end{sideremark}

By restricting a convergence relation $a:UX\relto X$ to principal ultrafilters one obtains
\begin{itemize}
\item an order relation $a_0:=a\cdot e_X:X\relto X$ where $x\le y$ whenever $\doo{x}\to y$ (we write $\le$ for $a_0$ and $\to$ for a) if one starts with a topological space,
\item or a metric $a_0=a\cdot e_X:X\relto X$ where $a_0(x,y)=a(\doo{x},y)$ if one starts with an approach spaces.
\end{itemize}
Note that for us an order relation does not need to be anti-symmetric, hence, an ordered set $X=(X,\le)$ consists of a set $X$ and a relation $\le:X\times X\to\two$ satisfying
\begin{align*}
x\le x &&\text{and}&& (x\le y\,\&\,y\le x)\,\Rw\,x\le z.
\end{align*}
Similarly, a metric $d$ on set $X$ is only required to satisfy 
\begin{align*}
0\geqslant d(x,x) &&\text{and}&& d(x,y)+d(y,z)\geqslant d(x,z),
\end{align*}
a ``classical'' metric is then a \emph{separated} ($d(x,y)=0=d(y,x)$ implies $x=y$),  \emph{symmetric} ($d(x,y)=d(y,x)$) and \emph{finitary} ($d(x,y)<\infty$) metric. The construction $a\mapsto a\cdot e_X$ results in forgetful functors $\TOP\to\ORD$ and $\AP\to\MET$, both have a left adjoint defined by $(X,a_0)\mapsto(X,e_X^\circ\cdot U(a_0))$. Furthermore, one has a forgetful functor $\MET\to\ORD$ which can be seen as the ``point shadow'' of $\AP\to\TOP$: for a metric space $(X,d)$, define
\[
x\le y \text{ whenever } 0\geqslant d(x,y).
\]
As in the ``ultrafilter case'', $\MET\to\ORD$ has a left adjoint $\ORD\to\MET$ via interpreting an order relation as a numerical relation. 

\begin{sideremark}
The left adjoint $\ORD\to\MET$ has a further left adjoint which sends the metric $d$ on $X$ to the order
\[
x\le y \text{ whenever } d(x,y)<\infty
\]
on $X$.
\end{sideremark}

Putting everything together, we have the following commuting diagram of right adjoint functors:
\[
\xymatrix{\AP\ar[r]\ar[d] & \MET\ar[d]\\ \TOP\ar[r] & \ORD.}
\]
The pointwise ordering makes $\ORD$ an ordered category, and these forgetful functors reflect this property into $\TOP$, $\MET$ and $\AP$. Concretely, for morphisms $f,g:X\to Y$
\begin{align*}
\text{in $\TOP$:} &\hspace{1em} f\le g\text{ whenever } e_X(f(x))\to g(x)\\
\text{in $\MET$:} &\hspace{1em}  f\le g\text{ whenever } 0\geqslant d(f(x),g(x))\\
\text{in $\AP$:} &\hspace{1em}  f\le g\text{ whenever } 0\geqslant d(e_X(f(x)),g(x))
\end{align*}
for all $x\in X$. We emphasise that it is in general very useful to realise the ordered nature of ones category since it allows to speak about adjunction, a notion which will be very helpful in our study of injectivity in $\TOP$ and $\AP$. 

We have seen that both topological and approach spaces (and also metric spaces) can be described as sets equipped with a (convergence, numerical) relation satisfying two simple axioms which, moreover, remind us immediately to the reflexivity and the transitivity condition of an ordered set and, consequently, to the identity and the composition law of a category. In the next section we will have a closer look on the simplest of these kind of structures, namely ordered sets.

\section{Some facts about complete ordered sets}\label{SectComplOrdSet}

The transportation of order-theoretic concepts into the realm of spaces relies on their respective formulation in point-free style using the notions of \emph{module} (also called order-ideal or distributor) and \emph{adjunction}. In this section we give a quick overview, mainly to establish notation; and refer to \citep{Woo_OrdAdj} for a nice presentation of ``ordered sets via adjunction''.

We recall that an ordered set is complete if each down-closed subset (down-set for short) has a supremum, or, equivalently, each up-set has an infimum. Formulated more carefully, an ordered set $X$ is \emph{complete} if each up-set has an infimum, dually, it is \emph{cocomplete} if each down-set has a supremum. By definition, $X$ is complete if and only if $X^\op$ is cocomplete. The ``non-careful'' formulation above relies on the fact that, moreover, $X$ is complete if and only if $X$ is cocomplete. 

A subset $A\subseteq X$ of an ordered set $X$ is \emph{down-closed} if and only if its characteristic map is monotone of type $X^\op\to\two$; likewise, $A$ is \emph{up-closed} if and only if its characteristic map is monotone of type $X\to\two$. Both concepts can be brought under one roof by introducing the notion of \emph{module} $\varphi:X\modto Y$, which is defined as a relation $\varphi:X\relto Y$ compatible with the order relations on $X$ and $Y$ in the sense that $\varphi:X^\op\times Y\to\two$ is monotone. One quickly verifies that a relation $\varphi:X\relto Y$ is a module if and only if 
\[
(x\le x'\,\&\,x'\,\varphi\, y'\,\&\,y'\le y)\;\Rw\;x\,\varphi\, y,
\]
and the pointfree version of this formula reads as $(\le_Y\cdot\varphi\cdot\le_X)\subseteq\varphi$. Since order relations are reflexive one actually has equality, moreover, this condition can be split in two parts so that $\varphi:X\relto Y$ is a module if and only if
\begin{align*}
\varphi\cdot\le_X=\varphi &&\text{and}&& \le_Y\cdot\varphi=\varphi.
\end{align*}
Summing up, a module can be seen either as
\begin{enumerate}
\item a relation $\varphi:X\relto Y$ satisfying the two equations above, or
\item a monotone map $\varphi:X^\op\times Y\to\two$, or
\item a monotone map $\mate{\varphi}:Y\to\two^{X^\op}$.
\end{enumerate}
Note that the equivalence between (b) and (c) relies on the fact that $\ORD$ is Cartesian closed. In general, for ordered sets $X$ and $Y$, the function space $Y^X$ is given by the set of all monotone functions of type $X\to Y$ with the pointwise order: $h\le h'$ whenever $\forall x\in X\,.\,h(x)\le h'(x)$.

The order relation $\le$ on $X$ is an example of a module $\le:X\modto X$ since the transitivity axiom gives $\le\cdot\le=\le$. By definition it is the identity arrow on $X$ in the ordered category $\MOD$ of ordered sets and modules between them, where the compositional and order structure is inherited from $\REL$. Two further important examples of modules are induced by a monotone map $f:X\to Y$:
\begin{align*}
f_*:X\modto Y,\;x\,f_*\,y:\iff f(x)\le y &&\text{and}&& f^*:Y\modto X,\; y\,f^*\,x:\iff y\le f(x),
\end{align*}
and one has $f_*=b\cdot f$ and $f^*=f^\circ\cdot b$. One easily verifies the inequalities $\le_X\subseteq f^*\cdot f_*$ and $f_*\cdot f^*\subseteq\le_Y$ for a monotone map $f:X\to Y$, hence $f_*\dashv f^*$ in $\MOD$. If we think of $x\in X$ as $x:1\to X$, then $x^*$ is the  down-set $\downc x$ generated by $x$, and $x_*$ is the up-set $\upc x$ induced by $x$.  It is also worth noting that these constructions define functors
\begin{align*}
(-)_*:\ORD\to\MOD &&\text{and}&& (-)^*:\ORD^\op\to\MOD,
\end{align*}
in particular, the order relation $\le$ in $X$ is both $(1_X)_*$ and $1_X^*$. Furthermore, $f\le g$ if and only if $f^*\le g^*$ if and only if $g_*\le f_*$, hence $(-)_*$ is order reversing and $(-)^*$ is order preserving. By this observation, $f\dashv g$ in $\ORD$ if and only if $g^*\dashv f^*$ in $\MOD$, which in turn is equivalent to $f_*=g^*$. In pointwise notation, this reads as the familiar formula
\[
\forall\,x\in X,y\in Y\,.\,f(x)\le y\iff x\le g(y).
\]

Coming back to ``up's and down's'', we identify a down-set with a module of type $X\modto 1$, and an up-set with a module of type $1\modto X$. Hence, the ordered set of all down-sets of $X$ can be identified with both the exponential $\two^{X^\op}$ in $\ORD$ and the ``ordered hom-set'' $\MOD(X,1)$; and we write $PX$ to denote this object. With the latter interpretation, the mate $\mate{\varphi}:Y\to PX$ of a module $\varphi:X\modto Y$ sends $y\in Y$ to $y^*\cdot\varphi$.

\begin{sideremark}\label{AdjInMod}
The composite $\psi\cdot\varphi$ of a down-set $\psi:X\modto 1$ with an up-set $\varphi:1\modto X$ yields a module of type $1\modto 1$ which is either $\true$ or $\false$; it is true precisely if $\varphi$ and $\psi$ have a common element. On the other hand, $\varphi\cdot\psi:X\modto X$ relates $x$ and $y$ if and only if $x$ belongs to $\psi$ and $y$ belongs to $\varphi$; therefore $\varphi\cdot\psi\subseteq\le$ if and only if each element of $\psi$ is less or equal then each element of $\varphi$. From this we conclude that $\varphi\dashv\psi$ in $\MOD$ if and only if $\psi=x^*$ and $\varphi=x_*$ for some $x\in X$. Using the Axiom of Choice, we deduce that each adjunction $\varphi\dashv\psi$ in $\MOD$ with $\varphi:X\modto Y$ and $\psi:Y\modto X$ is of the form $f_*\dashv f^*$ for some $f:X\to Y$ in $\ORD$. In fact, this statement is equivalent to the Axiom of Choice as shown in \citep{BD_CauchyCompl}.
\end{sideremark}

The mate of the identity module $\le:X\modto X$ is the \emph{Yoneda embedding} $\yoneda_X:X\to PX$ sending $x\in X$ to its down closure $\downc x=x^*$, which is indeed fully faithful thanks to the well-known Yoneda lemma which states
\[\downc x\subseteq\varphi\iff x\in\varphi.\]
This is a rather trivial statement in the context of ordered sets; however, the reformulation of this result is the key in the translation process from $\ORD$ to $\TOP$ and $\AP$. Cocompleteness of an ordered set $X$ gives a map $\Sup_X:PX\to X$ which, when writing down the definition of ``Supremum'', turns out to be left adjoint to $\yoneda_X$. In fact, $X$ is cocomplete if and only if $\yoneda_X$ has a left adjoint. With the help of the Yoneda lemma one easily shows that any $\emph{monotone}$ map $L:PX\to X$ with $L\cdot\yoneda_X=1_X$ is actually left adjoint to $\yoneda_X$ (see also \ref{KZ}). Clearly, the ordered set $PX$ of down-sets is cocomplete where the supremum of a down-set of down-sets $\Psi\in PPX$ is given by union $\bigcup\Psi$, or, in the language of modules, by $\Psi\cdot(\yoneda_X)_*:X\modto 1$.

More generally, arbitrary union of modules $X\modto Y$ is again a module which tells us that each hom-set in $\MOD$ is actually a (co)complete ordered set, moreover, relational composition preserves suprema. Hence, for $\varphi:X\modto Y$, both ``composition with $\varphi$''-maps $-\cdot\varphi$ and $\varphi\cdot-$ have a right adjoint. Unwinding the definition, a right adjoint to $-\cdot\varphi$ must give, for each $\psi:X\modto Z$, the largest module of type $Y\modto Z$ whose composite with $\varphi$ is contained in $\psi$,
\[
\xymatrix{X\ar[r]|-{\object@{o}}^\psi\ar[d]|-{\object@{o}}_\varphi & Z\\ Y\ar@{..>}[ur]|-{\object@{o}}^\subseteq}
\]
and a right adjoint to $\varphi\cdot-$ must provide, for each $\psi:Z\modto Y$, the largest module of type $Z\modto X$ whose composite with $\varphi$ is contained in $\psi$.
\[
\xymatrix{Y & Z\ar[l]|-{\object@{o}}_\psi\ar@{..>}[dl]|-{\object@{o}}_\supseteq\\ X\ar[u]|-{\object@{o}}^\varphi}
\]
We denote the right adjoint of $-\cdot\varphi$ as $-\blackleft\varphi$, and call $\psi\blackleft\varphi$ the \emph{extension} of $\psi$ along $\varphi$. Similarly, $\varphi\blackright-$ denotes the right adjoint of $\varphi\cdot-$, and $\varphi\blackright\psi$ is called the \emph{lifting} of $\psi$ along $\varphi$. All what was just said about $\MOD$ could have been said earlier about $\REL$, indeed the operations $\blackleft$ and $\blackright$ are just restrictions to modules of these operations on $\REL$. It is worthwhile noting that, for instance, the extension $\psi\blackleft\varphi$ of $\psi$ along $\varphi$ is given by
\begin{equation}\label{formula_extension}
y\, (\psi\blackleft\varphi)\, z \iff \forall x\in X\,.\, (x\,\varphi\, y\,\Rw\, x\,\psi\, z) \iff \mate{\varphi}(y)\le\mate{\psi}(z).
\end{equation}
 
\begin{sideremark}\label{wcol}
A supremum of a down-set $\psi:X\modto 1$ is by definition a smallest upper bound. Now, as we observed in \ref{AdjInMod}, an up-set $\varphi:1\to X$ consists only of upper bounds of $\psi$ if and only if $\varphi\cdot\psi\subseteq\le$, and $\varphi$ is the up-set of all upper bounds precisely if $\varphi=(\le\blackleft\psi)$. Furthermore, $x\in X$ is a smallest upper bound of $\psi$ if and only if $x_*=(\le\blackleft\psi)$. We recall that $\le=(1_X)_*$, hence an ordered set $X$ is cocomplete if, for each down-set $\psi:X\modto 1$, the extension $(1_X)_*\blackleft\psi$ of $(1_X)_*$ along $\psi$ is equal to $x_*$ for some $x\in X$. It is useful to observe here that a cocomplete ordered set $X$ admits a formally more general kind of colimits, namely, for each monotone map $h:A\to X$ and each module $\psi:A\modto B$, there exists a monotone map $f:B\to X$ with $f_*=(h_*\blackleft\psi)$. A diagram of the form
\[
\xymatrix{A\ar[r]^h\ar[d]|-{\object@{o}}_\psi & X\\ B}
\]
is called \emph{weighted} (by $\psi$), such a monotone map $f$ with $f_*=(h_*\blackleft\psi)$ is a \emph{colimit} of this diagram. Furthermore, any sup-preserving map preserves also all colimits.
\end{sideremark}

A monotone map $f:X\to Y$ induces a string of adjunctions between the ``down-set-sets'': one has the inverse image function $PY\to PX,\,B\mapsto f^{-1}(B)$ which has a left adjoint $Pf:PX\to PY,\,A\mapsto\downc f(A)$ and a right adjoint $PX\to PY,\,A\mapsto\{y\in A\mid f^{-1}(\downc y)\subseteq A\}$. The ``module point of view'' allows for an elegant description of these maps using relational composition: the inverse image function is given by $\psi\mapsto\psi\cdot f_*$, its left adjoint by $\varphi\mapsto\varphi\cdot f^*$ and its right adjoint by $\varphi\mapsto\varphi\blackleft f_*$. 
\[
\xymatrix{PX\ar@/^1.5em/[rr]^{(-\cdot f^*)}_\bot\ar@/_1.5em/[rr]_{(-\blackleft f_*)}^\bot && PY.\ar[ll]|-{(-\cdot f_*)}}
\]
Note that $f_*\dashv f^*$ in $\ORD$ gives $-\cdot f^*\dashv -\cdot f_*$ in $\MOD$.
It is interesting to observe that $-\blackleft(\yoneda_X)_*$ is just the Yoneda embedding $\yoneda_{PX}$ of $PX$ (use \eqref{formula_extension}), and therefore $\Sup_{PX}=-\cdot(\yoneda_X)_*$.

More generally, for each module $\varphi:X\modto Y$ one has an adjunction $-\cdot\varphi\dashv -\blackleft\varphi$ in $\ORD$. Since $\MOD$ is an ordered category, both $-\cdot\varphi:PY\to PX$ and $-\blackleft\varphi:PX\to PY$ are by definition monotone maps, however, later on we wish to deduce that these maps are continuous respectively contractive which does \emph{not} follow from $\Mod{\mU}$ (the ultra-counterpart of $\MOD$) being ordered. Therefore we note here that $-\cdot\varphi$ is the mate of the module  $(\yoneda_Y)_*\cdot\varphi:X\modto PY$, and $-\blackleft\varphi$ is the mate of $(\mate{\varphi})_*:Y\modto PX$.

The Yoneda embedding $\yoneda_X:X\to PX$ has an important universal property: for any monotone map $f:X\to Y$ with cocomplete codomain $Y$, there exists a unique sup-preserving (=left adjoint) extension $g:PX\to Y$, i.e.\ $g\cdot\yoneda_X\cong f$. Here $g$ takes a down-set $\psi$ to a supremum of its image in $Y$. In ``modul\^es'': $\psi$ maps to the supremum of $\psi\cdot f^*$, that is, $g$ can be taken as the composite $\sup_Y\cdot(-\cdot f^*)$. The right adjoint of $g$ is even easier to describe: it is simply the mate $\mate{f_*}:Y\to PX$ of $f_*:X\modto Y$. As a consequence, the (non-full) subcategory $\SUP$ of $\ORD$ consisting of all sup-lattices (=cocomplete anti-symmetric ordered sets) and sup-preserving maps is reflective in $\ORD$, a left adjoint to the inclusion functor is given by the down-set functor $P:\ORD\to\SUP$ which sends $X$ to $PX$ and $f:X\to Y$ to the map $-\cdot f^*:PX\to PY$ (``direct image''). In fact, $\SUP$ is monadic over $\ORD$, and the induced monad is given by the down-set functor $P:\ORD\to\ORD$ with units  the Yoneda embeddings $\yoneda_X:X\to PX$ and multiplications $\yonmult_X:PPX\to PX,\,\Psi\mapsto\Psi\cdot(\yoneda_X)_*$ (``union''). Its restriction to discrete ordered sets gives the usual power-set monad on $\SET$ which has the category $\SUP$ as Eilenberg-Moore category too.

\begin{remark}\label{KZ}
The down-set monad $\mPsh$ on $\ORD$ has a very particular property: $P\yoneda_X\le \yoneda_{PX}$ for all ordered sets $X$. This seemingly harmless property turns out to be very powerful, it implies for instance that $h:PX\to X$ in $\ORD$ is the structure morphism of a $\mPsh$-algebra if and only if $h\cdot \yoneda_X=1_X$, moreover, such a map $h$ is necessarily left adjoint to $\yoneda_X$. These kinds of monads where independently introduced by \citeauthor{Koc_MonAd} (in his thesis, but see his \citeyear{Koc_MonAd} article) and \citep{Zob_Doct}, hence one refers to them as of \emph{Kock-Z\"oberlein type}. From their results one can extract the following 
\begin{theorem}
Let $\mT=\monad$ be a monad on a ordered category $\catfont{X}$ where $T$ is a 2-functor. Then the following assertions are equivalent.
\begin{eqcond}
\item $Te_X\le e_{TX}$ for all $X\in\catfont{X}$.
\item $Te_X\dashv m_X$ for all $X\in\catfont{X}$.
\item $m_X\dashv e_{TX}$ for all $X\in\catfont{X}$.
\item For all $X\in\catfont{X}$, a $\catfont{X}$-morphism $h:TX\to X$ is the structure morphism of a $\mT$-algebra if and only if $h\cdot e_X=1_X$ (and then $h\dashv e_X$).
\end{eqcond}
\end{theorem}
Actually, we should be more careful here. The result above is certainly true if the order on hom-sets of $\catfont{X}$ is separated as the argumentation relies on uniqueness of adjoints. Fortunately, in most of our cases $TX$ will be separated, hence every $\mT$-algebra is separated and everything works as well.
\end{remark}

It is also well-known that the category $\ORD_\sep$ of separated ordered sets and monotone maps is dually equivalent to the category $\TAL$ of \emph{totally algebraic lattices} and $\sup$- and $\inf$-preserving maps. We refer to \citep{RW_CCD4} for a nice presentation of this particular result, and to \citep{PT_Dual} for a nice presentation of duality theory in general. This duality can be obtained by first constructing a (dual) adjunction
\[
 D:\ORD\rightleftarrows\CCD^\op:S
\]
between $\ORD$ and the category $\CCD$ of \emph{(constructively) completely distributive lattices} and $\sup$- and $\inf$-preserving maps. We recall from \citep{FW_CCD} that a complete lattice $X$ is (ccd) if $\Sup_X:PX\to X$ has a left adjoint $t_X:X\to PX$. Note that $t_X$ corresponds to a module of type $X\modto X$, and this relation is precisely the totally-below relation $\lll$ studied first by \citep{Raney_CD}. Clearly, any lattice of the form $PX$ is (ccd) since one has the string of adjunctions 
\[
\yoneda_{PX}=-\blackleft(\yoneda_X)_*\vdash -\cdot(\yoneda_X)_*\vdash -\cdot(\yoneda_X)^*=P\!\yoneda_X.
\]
The functor $D:\ORD\to\CCD^\op$ sends an ordered set $X$ to $DX:=PX=\two^{X^\op}$ and a monotone map $f:X\to Y$ to $Df:=(-\cdot f_*):DY\to DX$ (inverse image function). For $L\in\CCD$ with $\yoneda_L\vdash\Sup_L\vdash t_L$, one defines $SL:=A$ where $A$ is the equaliser 
\[
 \xymatrix{A\ar[r]^i & L\ar@<0.7ex>[r]^{t_L}\ar@<-0.7ex>[r]_{\yoneda_L} & PL.}
\]
Hence, $A$ can be taken as $\{x\in L\mid x\lll x\}$, that is, $A$ consists precisely of the \emph{totally compact} elements of $L$. Given also $M\in\CCD$ with corresponding equaliser $SM:=B$ and a $\sup$- and $\inf$-preserving map $f:L\to M$, then its left adjoint $g:M\to L$ restricts to $g_0:B\to A$. With $Sf:=g_0$ one obtains a functor\footnote{Here we need anti-symmetry of our (ccd)-lattices. Otherwise $S$ is only a pseudo-functor.} $S:\CCD^\op\to\ORD$. By the Yoneda lemma, $\yoneda_X:X\to PX$ is fully faithful and its image is precisely the equaliser of $P\!\yoneda_X$ and $\yoneda_{PX}$. Hence, 
\[
 \xymatrix{X\ar[r]^{\yoneda_X} & PX\ar@<0.7ex>[r]^{P\!\yoneda_X}\ar@<-0.7ex>[r]_{\yoneda_{PX}} & PPX}
\]
is an equaliser diagram for each anti-symmetric ordered set $X$. From that we get a natural equivalence $\eta:1\to SD$ which is a natural isomorphism if we restrict our self to anti-symmetric ordered sets. For $L\in\CCD$, one defines $\eps_L:L\to DS(L)$ as the composite (of right adjoints) $L\xrightarrow{\yoneda_L} PL\xrightarrow{-\cdot i_*} PA$, where $i:A\hrw L$ is the inclusion map. Clearly, $\eps_L$ preserves infima, and it is not difficult to verify that $\eps_L$ preserves also suprema. Therefore $\eps_L:L\to DS(L)$ lives in $\CCD$ and is indeed the $L$-component of a natural transformation $\eps:1\to DS$. The necessary equations are now easily verified, therefore one obtains the desired dual adjunction. We will now determine the fixed subcategories. There is nothing left to do on the $\ORD$-side, we observed already that $\Fix(\eta)=\ORD_\sep$. Therefore we concentrate now on $L\in\CCD$. The left adjoint $c:PA\to L$ of $\eps_L:L\to PA$ (where $A=SL$) sends $\psi\in PA$ to $\Sup_L(\psi\cdot i^*)$ (where $i:A\hrw L$ is the inclusion map). In fact, one always has $\eps_L\cdot c=1$, hence $\eps_L$ is an equivalence if $c\cdot\eps_L\ge 1$, that is, every $x\in L$ is a supremum of the totally compact elements below $x$. A (ccd)-lattice with this property is called \emph{totally algebraic}, and we obtain $\ORD_\sep\cong\TAL^\op$ where $\TAL$ denotes the full subcategory of $\CCD$ defined by the totally algebraic lattices.

\begin{remark}\label{DualConstr}
Firstly, instead of $X\mapsto\two^{X^\op}$ one can also work with $X\mapsto\two^{X}$, and construct the dual adjunction above as
\[
 \xymatrix@C=5em{\ORD\ar@<1ex>[r]^{\hom(-,\two)} & \CCD^\op.\ar@<1ex>[l]^{\hom(-,\two)}\ar@{}[l]|\bot}
\]
In fact, one construction can be obtain from the other by composing it with the equivalence $(-)^\op:\ORD\to\ORD$.
\end{remark}

\begin{sideremark}
Secondly, as explained in \citep{RW_CCD4}, the duality $\ORD_\sep\cong\TAL^\op$ is the restriction of a ``big'' duality involving the category $\CCD_\mathrm{sup}$ of (ccd)-lattices and $\sup$-preserving maps on one side and the idempotent splitting completion $\kar(\REL)$ of $\REL$ on the other side. This result is then further generalised in \citep{RW_Split}.
\end{sideremark}

\section{A short visit to metric spaces}\label{SectMetSp}

The discussion of the previous section can be easily brought to metric spaces by considering numerical relations, which amounts to substituting $\two$ by $[0,\infty]$, $\&$ by $+$, $\true$ by $0$, $x\Rw y$ sometimes by $x\geqslant y$ and sometimes by $\max\{y-x,0\}$ (truncated minus)\footnote{\ldots because $\Rw$ sometimes denotes the right adjoint to $\&$ ($x\&-\dashv x\Rw-$), and sometimes is used to express the inclusion $r\subseteq r'$ of relations pointwise.}, $\exists$ by $\inf$, $\forall$ by $\sup$, and so on. Most notably, we will usually not consider the Cartesian structure (=$\max$-metric) on $X\times Y$ but rather the $+$-metric, and denote the resulting space as $X\otimes Y$. This comes with the advantage that, albeit $\MET$ is not Cartesian closed, it is \emph{monoidal closed} in the sense that $X\otimes-$ has a right adjoint $-^X$. Here $Y^X$ can be taken as the set of all contraction maps of type $X\to Y$ together with the sup-metric $d(h,k)=\sup_{x\in X}b(h(x),h'(x))$. We are especially interested in $PX:=[0,\infty]^{X^\op}$, where the distance on $[0,\infty]$ is given by $\delta(x,y)=y-x$, and consequently on $PX$ by $[\varphi,\psi]=\sup_{x\in X}(\psi(x)-\varphi(x))$. One should compare this with the order case where the truth value of $[\varphi\subseteq\psi]$ is given by $\forall x\in X\,.\,\varphi(x)\Rw\psi(x)$. A \emph{module} $\varphi:X\modto Y$ between metric spaces $X=(X,a)$ and $Y=(Y,b)$ can be seen as either
\begin{enumerate}
\item a numerical relation $\varphi:X\relto Y$ satisfying $\varphi\cdot a=\varphi$ and $b\cdot\varphi=\varphi$, or
\item a contraction map $\varphi:X^\op\otimes Y\to[0,\infty]$, or
\item a contraction map $\mate{\varphi}:Y\to PX$.
\end{enumerate}
As before,
\begin{itemize}
\item each contraction map $f:X\to Y$ induces a module $f_*:X\modto Y,\,f_*(x,y)=b(f(x),y)$ and a module $f^*:Y\modto X,\,f^*(y,x)=b(y,f(x))$,
\item the metric $a$ of $X=(X,a)$ is the identity module $X\modto X$ on $X$,
\item which induces the Yoneda embedding $\yoneda_X:X\to PX$ sending $x$ to $x^*$, 
\item the Yoneda lemma states now that $[\yoneda_X(x),\psi]=\psi(x)$, 
\item a metric space is cocomplete whenever $\yoneda_X$ has a left adjoint $\Sup_X:PX\to X$, 
\item the cocomplete metric spaces are precisely the injective ones,
\item the subcategory $\COCTS_\sep$ of cocomplete and separated metric spaces and $\sup$-preserving contraction maps is reflective (in fact, monadic) in $\MET$, and the Yoneda embedding $\yoneda_X:X\to PX$ serves as a reflection map,
\item and so on.
\end{itemize}
An immediate question is now how the important notion of \emph{Cauchy-completeness} fits into this framework. The answer can be found in \citeauthor{Law_MetLogClo}'s \citeyear{Law_MetLogClo} paper where he made the amazing discovery that equivalence classes of Cauchy sequences correspond precisely to right adjoint modules $\psi:X\modto 1$, and a Cauchy sequence converges to $x$ if and only if $x$ is a supremum of the corresponding module. Consequently, $X$ is Cauchy complete if and only if the restriction $\yoneda_X:X\to\tilde{X}$ of the Yoneda embedding to the subspace $\tilde{X}$ of $PX$ defined by all right adjoint modules has a left adjoint in $\MET$. Since $\yoneda_X:X\to\tilde{X}$ is dense (in the usual metric sense), this simply means that $\yoneda_X:X\to\tilde{X}$ is surjective. Furthermore, $\yoneda_X:X\to\tilde{X}$ is a Cauchy completion for any space $X$. It is also worth noting that $\tilde{X}\hrw PX$ is the equaliser of
\begin{align*}
 \xymatrix{PX\ar@<0.7ex>[r]^{P\!\yoneda_X}\ar@<-0.7ex>[r]_{\yoneda_{PX}} & PPX} &\text{ (see also Lemma \ref{LemAdjModEq}).}
\end{align*}
As for ordered sets, one can built a dual adjunction between $\MET$ and $\CDMET$, which restricts to a dual equivalence between the full subcategories of Cauchy complete metric spaces and algebraic metric spaces. The reader has certainly no difficulties in writing down the definitions of completely distributive metric space and consequently of the category $\CDMET$ as well as of algebraic metric space.

\begin{sideremark}\label{CocomplExpMet}
Since $\MET$ is not Cartesian closed one might wonder what the exponentiable objects are. They are characterised in \citep{CH_ExpVCat} as those spaces $X=(X,a)$ where, for all $x,y\in X$, $u+v=a(x,y)$ and $\eps>0$, there exists some $z\in X$ with $a(x,z)\le u+\eps$ and $a(z,y)\le v+\eps$. One easily sees that a cocomplete (=injective) metric space satisfies this property, just consider (with $w=a(x,y)$)
\[
\xymatrix{\{0\xrightarrow{\,w\,}2\}\ar[drr]_f\ar@{>->}[rr] && \{0\xrightarrow{\,u\,}1\xrightarrow{\,v\,}2\}\ar@{..>}[d]^{g}\\ && X}
\]
where $f(0)=x$, $f(2)=y$ and $g(1)$ gives the desired $z\in X$. Furthermore, with $Y$ also $Y^X$ is  cocomplete (=injective), just pass from
\begin{align*}
\xymatrix{A\ar@{>->}[r]\ar[dr] & B\\ & Y^X}
&&\text{to}&&
\xymatrix{X\times A\ar@{>->}[r]\ar[dr] & X\times B\\ & Y.}
\end{align*}
Since the product of cocomplete spaces is also cocomplete, we conclude that the \emph{full} subcategory of $\MET$ defined by all cocomplete spaces is Cartesian closed. This observation contradicts Theorem 2.2 of \citep{Wag_PhD}; however, I believe the proof given there is not correct. 

I do not know yet if the corresponding result for $\V$-categories is true, that is, if a cocomplete $\V$-category is exponentiable in $\Cat{\V}$. In fact, I do not know if the $\V$-category $\V$ is exponentiable in $\Cat{\V}$.
\end{sideremark}

\section{The dual space}\label{SectionDualSpace}

In the remaining sections we will go further and do ``exactly the same'' in $\TOP$ and $\AP$. The first obstacle waits right at the beginning as the fundamental notion of down-set $\psi:X^\op\to\two$ involves the dual ordered set, a concept which has no obvious counterpart in $\TOP$ and $\AP$.\footnote{At this point one might ask why we do not consider completeness and consequently up-sets $\varphi:X\to\two$. But this creates even bigger problems as we have to deal then with the exponential $\two^X$ which in general does not exists in $\TOP$ and $\AP$. Furthermore, we would then like to talk about weighted limits (dual of \ref{wcol}) which involves lifting of modules, another problematic operation in the the realm of topological and approach spaces.}

Clearly, one cannot directly dualise the convergence relation $\frx\to x$ of a topological space to ``$x\to\frx$'', it is necessary to move into a more symmetric environment. Our experience shows so far that a good candidate for this are \citeauthor{Nach_TopOrd}'s ordered compact Hausdorff spaces as well as its metric counterparts. Here an \emph{ordered compact Hausdorff space} is a triple $(X,\le,\alpha)$ where $(X,\le)$ is an ordered set and $\alpha$ is (the convergence relation of) a compact Hausdorff topology on $X$ so that $\{(x,y)\mid x\le y\}$ is closed in $X\times X$. We emphasise again that we do not assume the order relation to be anti-symmetric. A map $f:X\to Y$ between ordered compact Hausdorff spaces is a homomorphism if it is both monotone and continuous, and the resulting category we denote as $\ORDCH$. It is shown in \citep{Fla97a} that the full subcategory $\ORDCH_\sep$ of $\ORDCH$ defined by the objects with anti-symmetric order is the category of Eilenberg-Moore algebras for the prime filter monad (of up-sets) $\mPf$ on $\ORD$, and the ``non-separated'' version of this result can be found in \citep{Tho_OrderedTopStr} with the prime filter monad substituted by the ultrafilter monad. Based on its extension to $\REL$, the ultrafilter monad $\mU=\umonad$ on $\SET$ extends to a monad on $\ORD$ where 
$U:\ORD\to\ORD$ sends $(X,\le)$ to $(UX,U\!\!\le)$, and with this definition $e_X$ and $m_X$ are monotone maps. Then, by Remark \ref{RemarkXi}, $\{(x,y)\mid x\le y\}$ is closed in $X\times X$ if and only if $\alpha:U(X,\le)\to(X,\le)$ is monotone. Therefore the category $\ORDCH$ of ordered compact Hausdorff spaces and continuous monotone maps is precisely the Eilenberg-Moore category $\ORD^\mU$. For each ordered set $X$ there is a canonical map $\rho_X:UX\twoheadrightarrow BX,\,\frx\mapsto\{\upc A\mid A\in\frx\}$ which turns out to be the $X$-component of a monad morphism $\rho:\mU\to\mPf$. It is shown in \citep[Lemma 5]{Fla97a} that $\rho_X$ is even surjective, and one easily verifies that $\rho_X(\frx)\le\rho_X(\frx')\iff \frx\le\frx'$. Hence, $\rho_X:UX\twoheadrightarrow BX$ is the anti-symmetric reflection of $UX$, and  composition with $\rho$ induces the inclusion functor $\ORDCH_\sep\to\ORDCH$. As a byproduct of this discussion we obtain a notion of \emph{metric compact Hausdorff spaces} as the Eilenberg-Moore algebras for the extension of $\mU$ to $\MET$ based on its extension to numerical relations, that is, $\METCH=\MET^\mU$. However, in Section \ref{SectionExamples} we will see that the notion of primer filter has at least two metric counterparts.

\begin{sideremark}
This is the place where we have to take serious that the order on hom-sets of $\ORD$ and $\MET$ is not anti-symmetric. The functor $U$ does not restrict to an endofunctor on $\ORD_\sep$ respectively $\MET_\sep$. For instance, the order relation of $U\N$ is not anti-symmetric, where $\N$ has the natural order. To see this, just take $\frx\in UX$ such that each $A\in\frx$ contains arbitrary large odd numbers, and $\fry\in UX$ such that each $B\in\fry$ contains arbitrary large even numbers. Then $\frx\le\fry$ and $\fry\le\frx$, but $\frx$ can be chosen different from $\fry$. This begs the question if it would be more ``natural'' to consider pseudo-algebras instead. 
\end{sideremark}

One has canonical forgetful functors
\begin{align*}
K:\ORDCH\to\TOP &&\text{and}&& K:\METCH\to\AP,
\end{align*}
both send $(X,a_0,\alpha)$ to $(X,a_0\cdot\alpha)$ where $a_0$ is either an order relation or a metric. 

\begin{examples}\label{Sierpinski}
The ordered set $\two=\{0,1\}$ with the discrete (compact Hausdorff) topology lives in $\ORDCH$ and gives us the Sierpi\'nski space $\two$ where $\{1\}$ is closed and $\{0\}$ is open. The metric space $[0,\infty]$ with distance $\delta(x,y)=\max\{y-x,0\}$ equipped with the usual compact Hausdorff topology where $\frx$ converges to $\xi(\frx):=\sup_{A\in\frx}\inf A$ is a metric compact Hausdorff space which gives the usual approach structure $\lambda(\frx,x)=x-\xi(\frx)$ on $[0,\infty]$.
\end{examples}

Both forgetful functors have a left adjoint
\begin{align*}
M:\TOP\to\ORDCH &&\text{respectively}&& M:\AP\to\METCH
\end{align*}
which sends $X=(X,a)$ to $(UX,Ua\cdot m_X^\circ,m_X)$. For a topological space $X=(X,a)$, the order relation 
\[
UX\xrelto{\,m_X^\circ\,}UUX\xrelto{\,Ua\,}UX
\]
is described by
\[
 \frx\le\fry\hspace{2em}\text{whenever }\overline{A}\in\fry\text{ for every }A\in\frx.
\]
For an approach space $X=(X,a)$, the metric $UX\xrelto{\,m_X^\circ\,}UUX\xrelto{\,Ua\,}UX$ gives
\[
\inf\{\eps\mid \forall A\in\frx\,.\,\overline{A}^{(\eps)}\in\fry\}
\]
as distance from $\frx$ to $\fry$. We define now $(-)^\op:\TOP\to\TOP$ and $(-)^\op:\AP\to\AP$ by
\begin{align*}
\xymatrix{\TOP\ar[d]_M\ar@{..>}[r]^{(-)^\op} & \TOP\\ \ORDCH\ar[r]_{(-)^\op} &\ORDCH\ar[u]_K}
&&\text{and}&&
\xymatrix{\AP\ar[d]_M\ar@{..>}[r]^{(-)^\op} & \AP\\ \METCH\ar[r]_{(-)^\op} &\METCH\ar[u]_K}
\end{align*}
where in the lower row one dualises only the order respectively metric.

\begin{examples}
By  definition, an ultrafilter $\frX\in UUX$ of ultrafilters converges to $\frx\in UX$ in $X^\op$ whenever $\frx\le m_X(\frX)$, which is equivalent to $A^\#\in \frX$ for each closed set $A\in\frx$. From this one obtains that all sets $A^\#$ for $A\subseteq X$ closed form a basis for the topology on $X^\op$. In this sense, we dualise $X$ by making the closed subsets of $X$ open. A continuous map $\psi:X^\op\to\two$ can be identified with a closed subset $\calA\subseteq UX$, where $\calA\subseteq UX$ is closed if and only if $\calA$ is Zariski closed (i.e.\ closed for the compact Hausdorff topology $m_X$ on $UX$) and down-closed (with respect to the order $\le$ on $UX$).
\end{examples}

As it is well-known, both $\TOP$ and $\AP$ are not Cartesian closed. However, the topological space $X^\op$ turns out to be exponentiable in $\TOP$ and it does not matter that $X^\op$ is in general not exponentiable in $\AP$ since what we need is a right adjoint of $X^\op\otimes-$ which does exist. As in the metric case, we consider here the $+$-approach structure rather then the $\max$-structure on the product space. We recall from \citep{Pis_ExpTop}/\citep{Hof_TopTh} that a topological/approach space $X=(X,a)$  is exponentiable/$+$-exponentiable if and only if the diagram
\[
\xymatrix{UUX\ar[r]^{m_X}\ar[d]|-{\object@{|}}_{Ua} & UX\ar[d]|-{\object@{|}}^a\\
 UX\ar[r]|-{\object@{|}}_a & X}
\]
commutes. 
\begin{proposition}\label{PropTensExp}
For each ordered compact Hausdorff space $X$, $KX$ is exponentiable in $\TOP$. Likewise, for each metric compact Hausdorff space, $KX$ is $+$-exponentiable in $\AP$.
\end{proposition}
\begin{proof}
Let $X=(X,a_0,\alpha)$ be in $\ORDCH$ or $\METCH$. We have to show that $a:=a_0\cdot\alpha$ satisfies $a\cdot Ua\sqsupseteq a\cdot m_X$ (since the other inequality holds anyway), where $\sqsubseteq$ stands either for $\subseteq$ or $\geqslant$. But this follows easily:
\[
a\cdot Ua=a_0\cdot\alpha\cdot U(a_0)\cdot U\alpha\sqsupseteq a_0\cdot\alpha\cdot U\alpha=a_0\cdot\alpha\cdot m_X=a\cdot m_X.\qedhere
\]
\end{proof}

\begin{corollary}\label{Xopexp}
For each topological (approach) space $X$, $X^\op$ is ($+$-)exponentiable.
\end{corollary}

\begin{sideremark}
Clearly, both $\ORD^\mU$ and $\MET^\mU$ inherit products from $\ORD$ and $\MET$ respectively. However, more important to us is the monoidal structure on $\MET$ defined by the plus-metric, and therefore we are interested in transporting this structure to $\MET^\mU$. This problem is addressed in general in \citep{Moe_Tensor} where the author introduces the notion of a \emph{Hopf monad} on a monoidal category $\catfont{C}$, which captures exactly what is needed to transport the monoidal structure on $\catfont{C}$ to the category of Eilenberg--Moore algebras. By space reasons we must refer to \citep{Moe_Tensor} for the definition of Hopf monad, and simply state here that the monad $\mU=\umonad$ on $\MET$ is an example of a monad with a Hopf structure since
\begin{align*}
\tau_{X,Y}:U(X\otimes Y) & \to UX\otimes UY,\,\frw \mapsto(T\pi_1(\frw),T\pi_2(\frw))
& !:U1 &\to 1
\end{align*}
are contraction maps. This is clear for the second map, and for the first one it follows using Remark \ref{RemarkXi}. Consequently, $\MET^\mU$ inherits the monoidal structure from $\MET$: for $X=(X,a,\alpha)$ and $Y=(Y,b,\beta)$, $X\otimes Y$ becomes equipped with the plus-metric $a\otimes b$ and the product topology $U(X\times Y)\xrightarrow{\tau_{X,Y}}UX\times UY\xrightarrow{\alpha\times\beta}X\times Y$. Recall from Example \ref{Sierpinski} that $[0,\infty]$ lives in $\MET^\mU$, and it is now clear that $+:[0,\infty]\otimes[0,\infty]\to[0,\infty]$ is a $\mU$-homomorphism. We also remark that $K:\MET^\mU\to\AP$ is a strict monoidal functor.
\end{sideremark}

\begin{sideremark}
In \citep{Sim82a,Wyl84a} it is shown that $\ORDCH_\sep$ is also monadic over $\TOP$ where the monad is the prime filter (of opens) monad. Similarly, the adjunction $M\dashv K$ induces a monad on $\TOP$ respectively $\AP$, in fact, it extends the ultrafilter monad $\mU=\umonad$ to these categories. Moreover, the monad $\mU$ on $\TOP$ as well on $\AP$ is of Kock-Z\"oberlein type, which tells us that a topological/approach space is an Eilenberg-Moore algebra precisely if $e_X:X\to UX$ admits a retract $l_X:UX\to X$ (i.e.\ $l_X\cdot e_X=1_X$) in $\TOP$/$\AP$, and then $l_X$ is left adjoint to $e_X$ (see Remark \ref{KZ}). For $X=(X,a_0,\alpha)$ in $\ORDCH$ or $\METCH$, $\alpha:UX\to X$ turns out to be cont(inuous/ractive), hence our functors $K:\ORD^\mU\to\TOP$ and $\MET^\mU\to\AP$ can be seen as functors $\ORD^\mU\to\TOP^\mU$ and $\MET^\mU\to\AP^\mU$ respectively. On the other hand, for $X=(X,a)$ in $\TOP^\mU$ or $\AP^\mU$, the underlying ordered set $(X,a_0)$ together with the left adjoint $l_X$ of $e_X$ lives in $\ORD^\mU$/$\MET^\mU$. By definition, $l_X\dashv e_X$ in $\TOP$ respectively $\AP$ and consequently in $\ORD$ respectively in $\MET$, and one observes that the underlying order/metric of $UX$ is given by $r=Ua\cdot m_X^\circ$. From
\[
a_0(l_X(\frx),x)=r(\frx,e_X(x))=a(\frx,x)
\]
one reaches eventually at the conclusion that $\TOP^\mU\cong\ORD^\mU$ and $\AP^\mU\cong\MET^\mU$. In particular, it is a property of an approach space to come from a metric compact Hausdorff space (the corresponding result for topological spaces is well-known). Finally, one easily verifies that the ultrafilter monad $\mU$ on $\AP$ is a Hopf monad witnessed by the maps $\tau_{X,Y}$ and $!$ described above.\\

\end{sideremark}

\section{Cocomplete spaces}\label{SectCocomplSp}

With the notion of dual space at our disposal, one can now introduce \emph{$\mU$-modules} between topological spaces and approach spaces and develop their basic properties. We emphasise that everything goes exactly as for ordered sets, only the Yoneda lemma is technically more demanding. For topological spaces $X=(X,a)$ and $Y=(Y,b)$, a $\mU$-module $\varphi:X\kmodto Y$ is a $\mU$-relation $\varphi:X\krelto Y$ so that $X^\op\times Y\to\two$ is continuous; and for approach spaces $X=(X,a)$ and $Y=(Y,b)$, a $\mU$-module $\varphi:X\kmodto Y$ is a $\mU$-relation $\varphi:X\krelto Y$ so that $X^\op\otimes Y\to[0,\infty]$ is contractive. By Corollary \ref{Xopexp}, $\mU$-modules correspond to cont(inuous/ractive) maps $\mate{\varphi}:Y\to PX$, where $PX:=\two^{X^\op}$ in the topological case and $PX:=[0,\infty]^{X^\op}$ in the approach case. It is not completely trivial that the module-property can be also expressed with the help of Kleisli composition, but it is indeed true that a $\mU$-relation $\varphi:X\krelto Y$ is a $\mU$-module if and only if $b\kleisli\varphi=\varphi$ and $\varphi\kleisli a=\varphi$ (see \citep{CH_Compl}). This correspondence will be particularly useful when establishing cont(inuity/ractivity) of a map of type $Y\to PX$ as it is occasionally easier to verify these two equalities.

\begin{sideremark}
It should be noted that dual space considered in this notes is different from what was considered in \citep{CH_Compl,HT_LCls,Hof_Cocompl,CH_CocomplII}, the two ingredients of an ordered/metric compact Hausdorff space were considered separately there. Since the presheaf space $PX$ there is defined as a subspace of the exponential with respect to the compact Hausdorff topology only, it is not automatically clear that this gives the same presheaf space. The following observation tells us that there is no problem:

\noindent
\textsc{Fact:} For any $(X,a_0,\alpha)$ in $\ORDCH$ or $\METCH$ and any $Y$ in $\TOP$ respectively $\AP$, the exponential $Y^{(X,a_0\cdot\alpha)}\to Y^{(X,\alpha)}$ of $(X,\alpha)\to(X,a_0\cdot\alpha)$ is an embedding.

\noindent To prove this, we recall that the function space structure on $Y^X$ (with $Y=(Y,b)$ and $X=(X,a)$) is defined as the largest one making the evaluation map $\ev:Y^X\times X\to Y$ (respectively $\ev:Y^X\otimes X\to Y$ in the approach case) cont(inuous/ractive). Explicitly, for $\frp\in U(Y^X)$ and $h\in Y^X$, one has
\[
\frp\to h\iff \text{for all } \frw\in U(Y^X\times X)\text{ with }\frw\mapsto\frp\text{ and all }x\in X,\,(\frx\to x\,\Rw\, U\!\ev(\frw)\to h(x))
\hspace{1em}(\text{where }\frw\mapsto\frx)
\]
in the topological case and
\[
d(\frp,h)=\sup\{b(U\!\ev(\frw),h(x))-a(\frx,x)\mid \frw\in U(Y^X\otimes X)\text{ with }\frw\mapsto\frp, x\in X, (\frw\mapsto\frx)\}
\]
in the approach case. Now, in $Y^{(X,\alpha)}$ one has
\[
d_2(\frp,h)=\sup\{b(U\!\ev(\frw),h(\alpha(\frx)))\mid \frw\in U(Y^X\otimes X)\text{ with }\frw\mapsto\frp, (\frw\mapsto\frx)\},
\]
and in $Y^{(X,a_0\cdot\alpha)}$
\[
d_1(\frp,h)=\sup\{b(U\!\ev(\frw),h(x))-a_0(\alpha(\frx),x)\mid \frw\in U(Y^X\otimes X)\text{ with }\frw\mapsto\frp, x\in X, (\frw\mapsto\frx)\}.
\]
To conclude $d_1(\frp,h)\leqslant d_2(\frp,h)$, we show that
\[
b(U\!\ev(\frw),h(\alpha(\frx)))\geqslant b(U\!\ev(\frw),h(x))-a_0(\alpha(\frx),x)
\]
for any $x\in X$. In fact, the inequality above is equivalent to
\[
b(U\!\ev(\frw),h(\alpha(\frx)))+a_0(\alpha(\frx),x)\geqslant b(U\!\ev(\frw),h(x)),
\]
which follows from
\[
b(U\!\ev(\frw),h(\alpha(\frx)))+a_0(\alpha(\frx),x)\geqslant b(U\!\ev(\frw),h(\alpha(\frx)))+b_0(h(\alpha(\frx)),h(x))\geqslant b(U\!\ev(\frw),h(x)).
\]
Here $b_0$ denotes the underlying metric of the approach structure $b$ on $Y$. For topological spaces one can argue in a similar way.

Consequently, the function space $PX$ is essentially the exponential of a compact Hausdorff space, therefore its topology is the compact-open topology. An approach variant of this topology was introduced by \citeauthor{LS_MonClosed} in \citeyear{LS_MonClosed}.
\end{sideremark}

\begin{example}\label{Psh_vs_Filter}
In \citep{HT_LCls} it is shown that the topological space $PX$ is homeomorphic to the space $F_0(X)$ of all filters (including the improper one) on the lattice $\tau$ of open sets of X, where the topology on $F_0(X)$ has 
\begin{align*}
\{\frf\in F_0(X)\mid A\in\frf\}&& \text{($A\subseteq X$ open)}
\end{align*}
as basic open sets (see \citep{Esc_InjSp}). Here we can identify an element $\psi\in PX=\two^{X^\op}$ with a closed (=Zariski and down-closed) subset $\calA$ of $UX$. With this identification, the maps
\begin{align*}
PX\xrightarrow{\;\Phi\;} F_0(X),\,\calA\mapsto \bigcap\calA\cap\tau&&\text{and}&&
F_0(X)\xrightarrow{\;\Pi\;} PX,\,\frf\mapsto\{\frx\in UX\mid \frf\subseteq\frx\}
\end{align*}
are indeed continuous and inverse to each other.
\end{example}

Consequently, the structure $a$ of a space $X=(X,a)$ is a $\mU$-module $X\kmodto X$ and indeed the identity arrow on $X$ in the ordered category $\Mod{\mU}$ of topological/approach spaces and $\mU$-modules between them, composition is given by Kleisli-composition and the order structure is inherited from $\REL$ respectively $\NREL$. Each cont(inuous/ractive) map $f:X\to Y$ gives rise to $\mU$-modules
\begin{align*}
f_*:X\kmodto Y,\,f_*(\frx,y)=b(Uf(\frx),y)
&&\text{and}&&&
f^*:Y\kmodto X,\,f^*(\fry,x)=b(\fry,f(x))\\
f_*=b\cdot Uf &&&&& f^*=f^\circ\cdot b
\end{align*}
which form an adjunction $f_*\dashv f^*$ in $\Mod{\mU}$, and these constructions define functors $(-)_*:\TOP\to\Mod{\mU}$ and $(-)^*:\TOP^\op\to\Mod{\mU}$ respectively  $(-)_*:\AP\to\Mod{\mU}$ and $(-)^*:\AP^\op\to\Mod{\mU}$. The ``order on hom-sets'' in $\TOP$ and $\AP$ are reflections from their respective module categories as
\[
f\le h\iff f^*\sqsubseteq h^*\iff h_*\sqsubseteq f_*.
\]
From this follows that $f\dashv g$ in $\TOP$/$\AP$ if and only if $g^*\dashv f^*$ in $\Mod{\mU}$ if and only if $g^*=f_*$, which in pointwise notation reads as
\[
b(Uf(\frx),y)=a(\frx,g(y)),
\]
or, in the particular case of topological spaces, as
\[
Uf(\frx)\to y\iff \frx\to g(y).
\]

The ordered category $\Mod{\mU}$ has (co)complete hom-sets, and Kleisli-composition with a $\mU$-module $\varphi:X\kmodto Y$ from the right preserves suprema. As in the case of ordered sets, a right adjoint to $-\kleisli\varphi$ gives, for each $\psi:X\kmodto Z$, the largest $\mU$-module of type $Y\kmodto Z$ which composite with $\varphi$ is less or equal then $\psi$:
\begin{equation}\label{ExtUMod}
\xymatrix{X\ar@{-^>}[r]|-{\object@{o}}^\psi\ar@{-^>}[d]|-{\object@{o}}_\varphi & Z\\ Y\ar@{..^>}[ur]|-{\object@{o}}^\sqsubseteq}
\end{equation}
This $\mU$-module is called \emph{extension of $\psi$ along $\varphi$}, and we write $\psi\whiteleft\varphi$. It can be calculated in $\REL$ respectively $\NREL$ as $\psi\blackleft (U\varphi\cdot m_X^\circ)$. However, in the sequel it will not be necessary to remember how  $\psi\whiteleft\varphi$ is computed neither one needs to recall the structure $\fspstr{-}{-}$ on $PX$, as long as one believes in

\begin{theorem}[\citep{Hof_Cocompl}]\label{TheoremGenYoneda}
$\psi\whiteleft\varphi(\fry,z)=\fspstr{U\mate{\varphi}(\fry)}{\mate{\psi}(z)}$.
\end{theorem}

Since the structure $a$ of $X=(X,a)$ is a $\mU$-module $X\kmodto X$, we obtain as its mate the \emph{Yoneda embedding} $\yoneda_X=\mate{a}:X\to PX$ which sends $x$ to $x^*=a(-,x)$. Choosing in \eqref{ExtUMod} $\varphi$ as the identity module and $\psi:X\kmodto 1$, the theorem above specialises to the Yoneda

\begin{lemma}\label{LemmaYoneda}
$\fspstr{U\!\yoneda_X(\frx)}{\psi}=\psi(\frx)$.
\end{lemma}

As usual, the lemma above tells us that the Yoneda embedding is fully faithful (=initial). For a topological space $X$, the Yoneda lemma says that, when identifying $\psi\in PX$ with a filter $\frf\in F_0(X)$,
\[
U\!\yoneda_X(\frx)\to\frf\iff\frx\supseteq\frf,
\]
which follows also easily from the definition of the topology on $F_0(X)$ (see Example \ref{Psh_vs_Filter}).

Each module $\varphi:X\kmodto Y$ induces maps $-\kleisli\varphi:PY\to PX$ and $-\whiteleft\varphi:PX\to PY$ which are both cont(inuous/ractive) as $-\kleisli\varphi$ is the mate of the module $(\yoneda_Y)_*\kleisli\varphi:X\kmodto PY$, and $-\whiteleft\varphi$ is the mate of $(\mate{\varphi})_*:Y\kmodto PX$, and therefore form an adjunction $-\kleisli\varphi\dashv -\whiteleft\varphi$ in $\TOP$/$\AP$. Hence, for $f:X\to Y$ in $\TOP$/$\AP$, one has
\[
\xymatrix{PX\ar@/^1.5em/[rr]^{(-\kleisli f^*)}_\bot\ar@/_1.5em/[rr]_{(-\whiteleft f_*)}^\bot && PY.\ar[ll]|-{(-\kleisli f_*)}}
\]
In the sequel we write $Pf$ for $-\kleisli f^*$. Note that $\psi\whiteleft(\yoneda_X)_*=\fspstr{-}{\psi}=\psi^*$, hence $-\whiteleft(\yoneda_X)_*=\yoneda_{PX}$.

Following the order-path, one calls a topological/approach space \emph{cocomplete} if the Yoneda embedding $\yoneda_X:X\to PX$ has a left adjoint $\Sup_X:PX\to X$ in $\TOP$/$\AP$. If, for a topological space $X$, we think of $PX$ as $F_0(X)$, then $\Sup_X$ produces for each filter $\frf\in F_0(X)$ a smallest convergence point. In \citep{Hof_Cocompl} it is shown that cocomplete spaces behave pretty much as cocomplete ordered sets:
\begin{itemize}
\item cocomplete=injective,
\item $PX$ is cocomplete\footnote{Of course, this follows also from the fact that any power of $\two$ respectively $[0,\infty]$ is injective in $\TOP$ respectively $\AP$.} where a supremum $\Sup_X:PPX\to PX$ is given by $-\kleisli(\yoneda_X)_*$,
\item the subcategory $\COCTS_\sep$ of $\TOP$/$\AP$ consisting of cocomplete T$_0$ spaces and left adjoint morphisms is reflective, and the Yoneda embedding provides a universal arrow,
\item even better, $\COCTS_\sep$ is monadic over $\TOP$/$\AP$ where the induced monad $\mPsh$ is of Kock-Z\"oberlein type and has $P$ as functor, the Yoneda embeddings $\yoneda_X:X\to PX$ as units and $\yonmult_X:=-\kleisli(\yoneda_X)_*:PPX\to PX$ as multiplications (providing us with the filter monad in the topological case and with what one might call now \emph{approach filter monad} in the approach case),
\item even even better, $\COCTS_\sep$ is also monadic over $\SET$ and $\ORD$/$\MET$.
\end{itemize}

\section{A seemingly unnatural dual adjunction}\label{SectionDuality}

At the end of Section \ref{SectComplOrdSet} we briefly discussed the dual adjunction between $\ORD$ and $\CCD$. The proof sketched there is (can be) entirely formulated in ``modul\^es'', hence it goes through without big problems for $\TOP$/$\AP$. It is interesting to observe that this only applies to $X\mapsto\two^{X^\op}$, the construction $X\mapsto\two^{X}$ (see Remark \ref{DualConstr}) is a completely different story and studied in general in \citep{HS_Morita}. Note that $(-)^\op:\TOP\to\TOP$ is no longer an equivalence, and also that $\two^{X^\op}$ is a (very particular) topological space but $\two^{X}$ in general not since $\TOP$ is not Cartesian closed. Of course, $X\mapsto\two^{X}$ leads to the well-known dual adjunction between $\TOP$ and $\FRM$, so lets look now at $X\mapsto\two^{X^\op}$.

In analogy to the $\ORD$-case, a cocomplete topological/approach space $X$ is called \emph{completely distributive} if $\SUP_X:PX\to X$ has a left adjoint in $\TOP$/$\AP$. This is not an empty concept since any space of type $PX$ is (cd), witnessed by the string of adjunctions
\[
\yoneda_{PX}=-\whiteleft(\yoneda_X)_*\vdash -\kleisli(\yoneda_X)_*\vdash -\kleisli(\yoneda_X)^*=P\!\yoneda_X.
\]
We let $\CDTOP$ ($\CDAP$) denote the category of completely distributive topological (approach) T$_0$-spaces and left-and-right adjoint cont(inuous/ractive) maps. The presheaf construction defines functors
\begin{align*}
 D:\TOP^\op\to\CDTOP &&\text{respectively}&& D:\AP^\op\to\CDAP
\end{align*}
sending $f:X\to Y$ to $-\kleisli f_*:PY\to PX$, that is, $DX=PX$ and $Pf\dashv Df$. A completely distributive space $L$ comes together with $\yoneda_L:L\to PL$ and $t_L:L\to PL$ where $t_L\dashv \Sup_L$. As before, we consider now the equaliser
\begin{align}\label{EqDiagS}
 \xymatrix{A\ar[r]^i & L\ar@<0.7ex>[r]^{t_L}\ar@<-0.7ex>[r]_{\yoneda_{L}} & PL.}
\end{align}
in $\TOP$/$\AP$. Let also $M$ be a completely distributive space with corresponding equaliser $j:B\hrw M$ and $f:L\to M$ in $\CDTOP$/$\CDAP$, hence $f$ preserves suprema and has a left adjoint $g:M\to L$. Therefore the diagrams
\begin{align*}
\xymatrix{M\ar[r]^{\yoneda_M}\ar[d]_g & PM\ar[d]^{Pg}\\ L \ar[r]_{\yoneda_L} & PL}
&&\text{and}&&
\xymatrix{PL\ar[r]^{\Sup_L}\ar[d]_{Pf} & L\ar[d]^f\\ PM\ar[r]_{\Sup_M} & M}
\end{align*}
commute (up to equivalence), and from the latter follows that also
\[
 \xymatrix{M\ar[r]^{t_M}\ar[d]_g & PM\ar[d]^{Pg}\\ L \ar[r]_{t_L} & PL}
\]
commutes (up to equivalence, but $PL$ is separated, so it really commutes). We conclude that $g:M\to L$ restricts to a cont(inuous/ractive) map $g_0:B\to A$. Summing up, we obtain functors
\begin{align*}
 S:\CDTOP\to\TOP^\op &&\text{respectively}&& S:\CDAP\to\AP^\op 
\end{align*}
where $SL:=A$ and $Sf=g_0$. 

To construct a natural transformation $\eta:1\to SD$, we start by observing that $P\!\yoneda_X\cdot\yoneda_X=\yoneda_{PX}\cdot\yoneda_X$ for any $X$ in $\TOP$/$\AP$; however, $\yoneda_X$ is in general not the equaliser of $P\!\yoneda_X$ and $\yoneda_{PX}$. Nevertheless, the universal property of the equaliser gives a cont(inuous/ractive) map $\eta_X:X\to SD(X)$ which is just the corestriction of the Yoneda embedding, and $\eta=(\eta_X)_X$ is indeed a natural transformation. Let now $L$ in $\CDTOP$/$\CDAP$ with equaliser diagram \eqref{EqDiagS}, we put
\[
\xymatrix@C=5em{L\ar[r]_{\yoneda_L}\ar@/^{2em}/[rr]^{\eps_L} & PL\ar[r]_-{-\kleisli i_*} & PA=DS(L).}
\]
Then $\eps_L$ is as right adjoint since both $\yoneda_L$ and $-\kleisli i_*$ are. To see that $\eps_L$ is also left adjoint, we show that 
\[
\xymatrix{PL\ar[r]^{P\eps_L}\ar[d]_{\Sup_L} & PPA\ar[d]^{\sup_{PA}=-\kleisli(\yoneda_A)_*}\\
L\ar[r]_{\eps_L} & PA}
\]
commutes. Let $\psi\in PL$ and $\fra\in UA$. Then (with $L=(L,a)$)
\begin{align*}
\eps_L\cdot\Sup_L(\psi)(\fra)&=a(Ui(\fra),\Sup_L(\psi))\\
&= \fspstr{U(t_L\cdot i)(\fra)}{\psi} && (t_L\dashv\Sup_L)\\
&=\fspstr{U\!\yoneda_L(Ui(\fra))}{\psi} && (t_L\cdot i=\yoneda_L\cdot i)\\
&=\psi(Ui(\fra))=\psi\kleisli i_*(\fra) &&(\text{Yoneda lemma})
\intertext{and}
\Sup_{PA}\cdot P\eps_L(\psi) &= \psi\kleisli\eps_L^*\kleisli(\yoneda_A)_*\\
&=\psi\kleisli\yoneda_L^*\kleisli(-\kleisli i_*)^*\kleisli(\yoneda_A)_*
&&(\eps_L=(-\kleisli i_*)\cdot\yoneda_L)\\
&=\psi\kleisli\yoneda_L^*\kleisli(Pi)_*\kleisli(\yoneda_A)_*
&&(Pi\dashv(-\kleisli i_*),\text{ hence }(Pi)_*=(-\kleisli i_*)^*)\\
&=\psi\kleisli\yoneda_L^*\kleisli(\yoneda_L)_*\kleisli i_*=\psi\kleisli i_* .
\end{align*}
Next we show that $\eps=(\eps_L)_L$ is a natural transformation $\eps:1\to DS$. To this end, let $f:L\to M$ in $\CDTOP$/$\CDAP$ with left adjoint $g:M\to L$. We have to convince our self that
\[
\xymatrix{L\ar[d]_f\ar[r]^{\eps_L} & PA\ar[d]^{-\kleisli(g_0)_*}\\ M\ar[r]_{\eps_M} & PB}
\]
commutes (we use here the notation introduced above), which we do by pasting the commutative diagrams
\begin{align*}
\xymatrix{L\ar[d]_f\ar[r]^{\yoneda_L} & PL\ar[d]^{Pf}\\ M\ar[r]_{\yoneda_M} &PM}
&&\text{and}&&
\xymatrix{PL\ar[d]_{-\kleisli g_*}\ar[r]^{-\kleisli i_*} & PA\ar[d]^{-\kleisli(g_0)_*}\\
PM\ar[r]_{-\kleisli j_*} & PB}
\end{align*}
together. This is indeed possible since from $Pg\dashv Pf$ and $Pg\dashv(-\kleisli g_*)$ follows $Pf=-\kleisli g_*$. Finally, the composites
\[
\begin{matrix}
SL & \xrightarrow{\,\eta_{SL}\,} & SDS(L) &\xrightarrow{\,S(\eps_L)\,} & SL\hspace{2em}\\
x  & \longmapsto & x^* &\longmapsto & \Sup_L(x^*)&=x
\end{matrix}
\]
and
\[
\begin{matrix}
DX & \xrightarrow{\,\eps_{DX}\,} & DSD(X) & \xrightarrow{\,D(\eta_X)\,} & DX\hspace{3em}\\
\psi &\longmapsto & \psi^*\kleisli i_* & \longmapsto & \psi^*\kleisli i_*\kleisli(\eta_X)_*
&=\psi^*\kleisli(\yoneda_X)_*=\psi
\end{matrix}
\]
are both equal to the identity, where $i:SDX\hrw DX$ denotes the inclusion map.

\begin{theorem}
$(D,S,\eta,\eps)$ define a (dual) adjunction $\TOP^\op\leftrightarrows\CDTOP$ resp.\ $\AP^\op\leftrightarrows\CDAP$.
\end{theorem}

\begin{sideremark}\label{no_schiz_obj}
The dual adjunction above does not seem to be induced by a schizophrenic object. Certainly, $S\cong\hom(-,\two)$ respectively $S\cong\hom(-,[0,\infty])$, but there is no space $X$ with $D\cong\hom(-,X)$. This indicates that the ``obvious'' forgetful functor $\CDTOP\to\SET$ respectively $\CDAP\to\AP$ is a ``bad'' choice, in fact, we will later on (Remark \ref{yes_schiz_obj}) see that there is a better candidate.
\end{sideremark}

As for any dual adjunction, one obtains a dual equivalence between the fixed full subcategories
\begin{align*}
\Fix(\eta):=\{X\mid \eta_X\text{ is an isomorphism}\}
&&\text{and}&&
\Fix(\eps):=\{L\mid \eps_L\text{ is an isomorphism}\}
\end{align*}
which we determine now.
\begin{lemma}\label{LemAdjModEq}
For each topological/approach space $X$ and $\psi\in PX$, 
\[
 P\!\yoneda_X(\psi)=\yoneda_{PX}(\psi)\iff\text{$\psi$ is right adjoint.}
\]
\end{lemma}
\begin{proof}
Our proof uses the fact obtained by \citep{HT_LCls} that
\[
\tilde{X}:=\{\psi\in PX\mid \psi\text{ is right adjoint}\}                                                  
\]
is the Lawvere closure of $\yoneda_X(X)$ in $PX$. Clearly, the equaliser of $\yoneda_{PX}$ and $P\!\yoneda_X$ is Lawvere closed and contains $\yoneda_X(X)$, and the implication ``$\Lw$'' follows. To see ``$\Rw$'', note that from $P\!\yoneda_X(\psi)=\yoneda_{PX}(\psi)$ follows $\psi^*=\psi\kleisli\yoneda_X^*$, hence $\psi\kleisli\yoneda_X^*(\doo{\psi})$ is $\true$ respectively $0$. Since $Ue_Y\cdot e_Y= m_Y^\circ\cdot  e_Y$\footnote{The same holds for any monad where $T1=1$.} for any $Y$,
\[
\psi\kleisli\yoneda_X^*(\doo{\psi})
=\psi\cdot U\yoneda_X^*(e_{UPX}\cdot e_{PX}(\psi))
=\bigvee_{\frx\in UX}\psi(x)\otimes U\hat{a}(e_{UPX}\cdot e_{PX}(\psi),T\yoneda_X(\frx))
\]
where $\hat{a}$ denotes the structure on $PX$, $\otimes$ is either $\&$ or $+$, and $\bigvee$ is either $\exists$ or $\inf$. The result follows now from Proposition 4.16 (3.16 in the arXiv-version) of \citep{HT_LCls}.
\end{proof}
Hence, $X$ belongs to $\Fix(\eta)$ precisely if each right adjoint module $\psi$ is representable as $\psi=x^*$ for a unique $x\in X$. But this is precisely the definition of a \emph{Lawvere complete}\footnote{also called Cauchy complete} separated space as introduced in \citep{CH_Compl}. In both the topological and the approach case, Lawvere completeness together with separateness means soberness, so that $\Fix(\eta)$ is precisely the category $\SOB$/$\ASOB$ of sober topological/approach spaces and continuous/contraction maps.

\begin{example}\label{ExampleRighAdjModTop}
For a topological space $X$, a $\mU$-module $\varphi:1\kmodto X$ corresponds to a closed subset $A\subseteq X$, and $\psi:X\kmodto 1$ to a closed subset $\calA\subseteq UX$. With this identification, $\varphi\dashv\psi$ means that (see \citep{CH_Compl})
\begin{itemize}
\item $\calA=\{\frx\in UX\mid \forall\,x\in A\,.\,\frx\to x\}$,
\item there exists an ultrafilter $\frx_0\in\calA$ with $A\in\frx_0$.
\end{itemize}
Hence, for any $\frx\in\calA$ and any $B\in\frx$, $A\subseteq\overline{B}$ and therefore $B\in\frx_0$. We conclude that $\frx\le\frx_0$, hence $\calA=\downc\frx_0$.
\end{example}

For $L$ in $\CDTOP$/$\CDAP$, $\eps_L:L\to PA$ has a left adjoint $c:PA\to L$ which sends $\psi\in PA$ to $\Sup_L(\psi\kleisli i^*)$. Since $\eps_L$ preserves suprema and $\eps\cdot i=\yoneda_A$, we see that even $\eps_L\cdot c=1$ since
\begin{multline*}
\eps_L\cdot c(\psi)=\eps_L(\Sup_L(\psi\kleisli i^*))=\Sup_{PA}(P\eps_L(\psi\kleisli i^*))
=\Sup_{PA}(\psi\kleisli i^*\kleisli\eps_L^*)\\
=\Sup_{PA}(\psi\kleisli\yoneda_A^*)
=\yonmult_A\cdot P\yoneda_A(\psi)=\psi.
\end{multline*}
We call a completely distributive topological/approach space $L$ \emph{totally algebraic} if also $c\cdot\eps_L\cong1$, which amounts to the condition
\[
 \Sup_L(x^*\kleisli i_*\kleisli i^*)\cong x
\]
for each $x\in X$. Clearly, $\Fix(\eps)$ is the full subcategory of $\CDTOP$/$\CDAP$ consisting of all totally algebraic $T_0$-spaces; we denote this category as $\TATOP$ respectively as $\TAAP$. In conclusion,
\begin{theorem}\label{TheoremAbsDuality}
$\SOB^\op\cong\TATOP$ and $\ASOB^\op\cong\TAAP$.
\end{theorem}

\begin{example}
By definition, a topological space $X$ is totally algebraic if each element $x\in X$ is a supremum of the distributor $x^*\kleisli i_*\kleisli i^*:X\kmodto 1$. Intuitively, $x^*\kleisli i_*\kleisli i^*$ is the down-set of all totally algebraic elements below $x$, and in fact, $\frx\in UX$ belongs to $x^*\kleisli i_*\kleisli i^*$ if and only if there is some $\fra\in UA$ with $\frx\le\fra$ and $\fra\to x$.
\end{example}

\begin{sideremark}
It is well-known (see, for instance,Theorem 2.0 of \citep{LR_SGDuality}) that these fixed subcategories are reflective if and only if $\eta_{SL}$ respectively $\eps_{DX}$ are isomorphisms, that is, $SL$ is sober respectively $DX$ is totally algebraic. Now, any completely distributive space is cocomplete, hence Lawvere complete (=sober), and $SL$ is L-closed (see \citep{HT_LCls}) in $L$ since it is the equaliser of $\yoneda_L$ and $t_L$. Therefore $SL$ is sober. Certainly, $DX=PX$ is totally algebraic for each sober space $X$. For an arbitrary space $X$, the induced $\mU$-module $i_*$ of the sobrification $i:X\to\tilde{X}$ satisfies $i^*\kleisli i_*=1$ and $i_*\kleisli i^*=1$, therefore $PX\cong P\tilde{X}$ and the assertion follows.
\end{sideremark}

\section{Frames vs.\ complete distributivity}\label{SectionFramesDist}

In the previous section we have studied the dual adjunctions
\begin{align*}
\TOP^\op\leftrightarrows\CDTOP&&\text{and}&&\AP^\op\leftrightarrows\CDAP
\end{align*}
which (I believe) are quite different from the ``traditional ones with frames (see \citep{Isb_Frm}) respectively approach frames (see \citep{BLO_AFrm}). Nevertheless, these adjunctions restrict to dual equivalences involving (approach) sober spaces; therefore one might ask now about the relationship between frames and completely distributivity spaces. In this section we will consider only the topological case since I do not know the answer for approach spaces. 

Recall from Example \ref{Psh_vs_Filter} that $PX$ is homeomorphic to the filter space $F\calO X$, where $\calO X$ denotes as usual the frame of open subsets of a topological space $X$. Therefore we can hope that there is a commutative diagram
\[
\xymatrix{ & \TOP^\op\ar[dl]_\calO\ar[dr]^D\\ \FRM\ar[rr]_F && \CDTOP}
\]
of functors, where $FL$ denotes the usual filter space of a frame. More general, for a meet semi-lattice $L$ one puts
\[
 FL:=\{\frf\subseteq L\mid\text{ $\frf$ is a (possibly improper) filter}\}
\]
which is a topological space with
\[
 x^\#=\{\frf\in FL\mid x\in\frf\}\hspace{2em}(x\in L)
\]
as basic open set. Note that $1^\#=FL$ and $(x\wedge y)^\#=x^\#\cap y^\#$. Furthermore, the underlying order on $FL$ is given by
\[
\frf\le\frg\iff \doo{\frf}\to\frg\iff \forall x\in \frg\,.\,\frf\in x^\#\iff \frg\subseteq\frf,
\]
which also tells us that $FL$ is separated (=T$_0$). For a meet semi-lattice homomorphism $f:L\to M$, the mapping
\[
Ff:FL\to FM,\,\frf\mapsto\upc\{f(x)\mid x\in\frf\}
\]
is continuous since
\[
 Ff^{-1}(y^\#)=\{\frf\in FL\mid\exists x\in\frf\,.\,f(x)\le y\}=\bigcup_{x:f(x)\le y}x^\#,
\]
and so is
\[
 f_!:FM\to FL,\,\frg\mapsto f^{-1}(\frg).
\]
since
\begin{equation}\label{eqforrho}
 f_!^{-1}(x^\#)=\{\frg\mid f_!(\frg)\in x^\#\}=\{\frg\mid f(x)\in\frg\}=f(x)^\#.
\end{equation}
Furthermore, one easily verifies that $f_!\dashv Ff$ in $\TOP$. Given also $g:L\to M$ with $f\le g$ and $\frf\in FL$, then
\[
 \{g(x)\mid x\in\frf\}\subseteq\upc\{f(x)\mid x\in\frf\}=Ff(\frf)
\]
and therefore $Ff(\frf)\le Fg(\frf)$. We write $\TOP_{\inf}$ for the 2-category of T$_0$-spaces and right adjoint continuous maps with the pointwise order on hom-sets, and $\SLAT$ denotes the 2-category of meet semi-lattices and meet semi-lattice homomorphisms with the pointwise order on hom-sets.
\begin{proposition}
$F:\SLAT\to\TOP_{\inf}$ is a 2-functor.
\end{proposition}

Given a meet semi-lattice $L$, one has the mapping
\[
\alpha_L:L\to\calO(FL),\,x\mapsto x^\#
\]
which is an order-embedding since $x^\#\subseteq y^\#\iff\upc x\in y^\#\iff x\le y$. Furthermore, $\alpha_L$ preserves all existing infima in $L$. To see this, observe first that
\[
 \interior(\calA)=\{\frf\in FL\mid \exists x\in\frf\,.\,x^\#\subseteq\calA\}
\]
Let now $(x_i)_{i\in I}$ be a family of elements of $L$ with infimum $x\in L$. Then
\begin{multline*}
\bigwedge_{i\in I} x_i^\#=\interior(\bigcap_{i\in I} x_i^\#)
=\{\frf\in FL\mid \exists z\in\frf\,\forall i\in I\,.\,z^\#\subseteq x_i^\#\}\\
=\{\frf\in FL\mid \exists z\in\frf\,\forall i\in I\,.\,z\le x_i\}
=\{\frf\in FL\mid x\in\frf\}=x^\#.
\end{multline*}
If $L$ is complete, then $\alpha_L:L\to\calO(FL)$ has a left adjoint $\beta_L:\calO(FL)\to L$ which is necessarily given by
\[
 \beta_L(\calA)=\bigwedge\{x\in L\mid \calA\subseteq x^\#\}.
\]

\begin{lemma}
Assume that $L$ is complete. For any open subset $\calA\subseteq FL$,
\[
 \bigwedge\{x\in L\mid \calA\subseteq x^\#\}=\bigvee\{y\in L\mid y^\#\subseteq\calA\}.
\]
\end{lemma}
\begin{proof}
We only need to show ``$\le$''. We put $z=\bigvee\{y\in L\mid y^\#\subseteq\calA\}$ and show $\calA\subseteq z^\#$. To this end, let $\frf\in\calA$. Since $\calA$ is open, there is some $u\in\frf$ with $u^\#\subseteq\calA$. Hence $u\le z$ and therefore $\frf\in z^\#$.
\end{proof}

\begin{proposition}
For every frame $L$, $\beta_L:\calO(FL)\to L$ is a frame homomorphism.
\end{proposition}
\begin{proof}
Clearly, $\beta_L(FL)=\top$. Let now $\calA,\calB\in\calO(FL)$. Then
\begin{multline*}
\beta_L(\calA)\wedge\beta_L(\calB)
=\bigvee\{y\in L\mid y^\#\subseteq\calA\}\wedge\bigvee\{z\in L\mid z^\#\subseteq\calB\}\\
=\bigvee\{y\wedge z\mid y^\#\subseteq\calA,z^\#\subseteq\calB\}
=\bigvee\{x\in L\mid x^\#\subseteq\calA\cap\calB\}=\beta_L(\calA\cap\calB).\qedhere
\end{multline*}
\end{proof}
\noindent Hence, for any frame $L$, one has
\[
\xymatrix{FL\ar@/^1.5em/[rr]^{F\alpha_L}_\top\ar@/_1.5em/[rr]_{(\beta_L)_!}^\top && F\calO F(L)\ar[ll]|-{F\beta_L}}
\]
Since $P(FL)\cong F\calO F(L)$ and
\[
 F\alpha_L(\frf)=\langle\{x^\#\mid x\in\frf\}\rangle=\yoneda_{FL}(\frf),
\]
we conclude that $FL$ is a completely distributive T$_0$-space.

\begin{proposition}
$F:\SLAT\to\TOP_{\inf}$ restricts to a 2-functor $F:\FRM_\wedge\to\CDTOP_{\inf}$ where $\FRM_\wedge$ denotes the full subcategory of $\SLAT$ defined by those meet-semilattices which are frames, and $\CDTOP_{\inf}$ denotes the 2-category of completely distributive T$_0$-spaces and right adjoint continuous maps.
\end{proposition}

To show that $F:\FRM_\wedge\to\CDTOP_{\inf}$ is an equivalence of categories, we will now describe its inverse $\Pt:\CDTOP_{\inf}\to\FRM_\wedge$. To motivate our construction, note that this functor should send a completely distributive space $Y$ of the form $Y\cong PX$ for $X\in\TOP$ to the frame $\calO X\cong\TOP(X,\two)^\op$ of opens of $X$. By the universal property of the Yoneda embedding, 
\[
 \mathrm{LeftAdjoint}(PX,2)\to\TOP(X,\two),\,g\mapsto g\cdot\yoneda_X
\]
is an order isomorphism. Its inverse sends $\varphi:X\to\two$ to the left adjoint
\begin{equation}\label{LeftKan}
 \varphi_L:=\Sup_\two\cdot P\varphi:PX\to\two.
\end{equation}
Therefore we consider, for any topological space $X$,
\[
 \Lambda(X):=\{\varphi:X\to\two\mid \text{$f$ is continuous and left adjoint}\}
\]
which becomes an ordered set with the pointwise order. In the sequel we will write $\calC(X)$ for the coframe of all continuous maps of type $X\to\two$. Note that $\varphi:X\to\two$ is left adjoint in $\TOP$ if and only if it is continuous and left adjoint in $\ORD$ (with respect to the underlying orders). The first hint that we are on the right track is
\begin{lemma}\label{CompNatTransFrm}
For each frame $L$, the map $\rho_L:L\to\Lambda(FL)^\op$ sending $x\in L$ to
\[
\varphi_x:FL\to\two,\,\frf\mapsto
\begin{cases}
 1 & x\notin\frf\\
 0 & x\in\frf
\end{cases}
\]
is an order-isomorphism.
\end{lemma}
\begin{proof}
First note that $\varphi_x$ is the characteristic map of the complement of $x^\#$, hence it is continuous. Furthermore, $\varphi_x$ preserves suprema (=intersection), hence it is left adjoint. From
\[
 x\le y\iff\forall\frf\in FL\,.\,(x\in\frf\Rw y\in\frf)\iff \varphi_y\le\varphi_x
\]
we deduce that $L\to\Lambda(FL)^\op$ is an order-embedding. Let now $\varphi:FL\to\two$ be continuous and left adjoint. Put $\calB=\varphi^{-1}(0)$ and $\frf=\bigvee\calB$. Since $\varphi$ preserves suprema, $\varphi(\frf)=0$ and therefore $\frf\in\calB$. Since $\calB$ is open, there is some $x\in\frf$ with $x^\#\subseteq\calB$. Hence $\upc x\le\frf$, that is, $\frf\subseteq\upc x$, and therefore $\frf=\upc x$. We conclude that $\varphi=\varphi_x$.
\end{proof}

\begin{proposition}
Let $X$ be a completely distributive spaces with $t_X\dashv\Sup_X\dashv\yoneda_X$. Then the inclusion map $i:\Lambda(X)\to\calC(X)$ has a right adjoint $r:\calC(X)\to\Lambda(X)$ given by $r(\varphi)=\varphi_L\cdot t_X$ (see \eqref{LeftKan}). Moreover, $r$ preserves finite suprema.
\end{proposition}
\begin{proof}
First note that $r(\varphi)$ is left adjoint since it is a composite of left adjoint. Furthermore, $i\cdot r\le 1$ since $\varphi=\varphi_L\cdot\yoneda_X\ge \varphi_L\cdot t_X$ for any $\varphi\in\calC(X)$, and $r\cdot i=1$ since $\varphi=\varphi\cdot\Sup_X\cdot t_X=\Sup_\two\cdot P\varphi\cdot t_X=\varphi_L\cdot t_X$ for each left adjoint $\varphi:X\to\two$. Finally, $r:\calC(X)\to\Lambda(X)$ is the corestriction of
\[
 \calC(X)\xrightarrow{\;\cong\;}\Lambda(PX)\xrightarrow{\;\text{left adjoint}\;}\calC(PX)
\xrightarrow{\;\text{coframe homom.\ induced by }t_X\;}\calC(X),
\]
therefore $r$ preserves finite suprema.
\end{proof}
\begin{corollary}
For each completely distributive spaces $X$, $\Lambda(X)$ is a coframe.
\end{corollary}

For any left adjoint $g:Y\to X$ in $\TOP$, composition with $g$ defines a monotone map
\[
\Lambda(g):\Lambda(X)\to\Lambda(Y),\,\varphi\mapsto\varphi\cdot g.
\]
Furthermore, since
\[
\xymatrix{\Lambda(X)\ar[r]^{\Lambda(g)}\ar[d] & \Lambda(Y)\ar[d]\\
\calC(X)\ar[r]_{\calC(g)} & \calC(Y)}
\]
commutes, $\Lambda(g)$ preserves finite suprema. For $X$ in $\CDTOP_{\inf}$ we put $\Pt(X):=\Lambda(X)^\op$, and for $f:X\to Y$ in $\CDTOP_{\inf}$ with left adjoint $g:Y\to X$ we define $\Pt(f)=\Lambda(g)^\op$. Then

\begin{proposition}
$\Pt:\CDTOP_{\inf}\to\FRM_\wedge$ is a 2-functor.
\end{proposition}
Furthermore, we revise Lemma \ref{CompNatTransFrm}:
\begin{lemma}
$\rho_L$ is the $L$-component of a natural isomorphism $\rho:1_{\FRM_\wedge}\to\Pt F$.
\end{lemma}
\begin{proof}
Use \eqref{eqforrho} to conclude naturality.
\end{proof}
For a space $X$ in $\CDTOP_{\inf}$, we put
\[
 \sigma_X:X\to F\Pt(X),\,x\mapsto\{\varphi\in\Lambda(X)\mid\varphi(x)=0\}.
\]
\begin{lemma}
$\sigma_X$ is surjective.
\end{lemma}
\begin{proof}
Let $\frj\subseteq\Lambda(X)$ be an ideal. For any $\varphi\in\frj$, put $A_\varphi:=\{x\in X\mid\varphi(x)=0\}$ and $x_\varphi:=\bigvee A_\varphi$. Since $x_\psi\le x_\varphi$ for $\varphi\le\psi\in\frj$, the association $\varphi\mapsto x_\varphi$ defines a codirected diagram $D:\frj^\op\to X$. Let $x=\bigwedge_{\varphi\in\frj}x_\varphi$. By continuity, $\varphi(x)=0$ for every $\varphi\in\frj$. Let now $\varphi_0\in\Lambda(X)$ with $\varphi_0\notin\frj$. For any $\varphi\in\frj$, $\varphi_0\not\le\varphi$ and therefore there is some $x\in A_\varphi$ with $\varphi_0(x)=1$, hence $\varphi_0(x_\varphi)=1$. Consequently, $\varphi_0(x)=1$.
\end{proof}

By definition, any space $X=FL$ for some frame $L$ has a basis for the closed sets formed by the complements of the opens $x^\#$ ($x\in L$). The characteristic map of such a basic closed set is left adjoint (see Lemma \ref{CompNatTransFrm}), hence any $\varphi\in\calC(X)$ is the infimum of elements of $\Lambda(X)$. Via the adjunction $t_X\dashv\Sup_X$ one can transport this property to any completely distributive space $X$ as follows. For any $\varphi\in\calC(X)$, $\varphi\cdot\Sup_X\in\calC(PX)$, hence $\varphi\cdot\Sup_X\cong\bigwedge_i\varphi_i$ in $\calC(PX)$ with all $\varphi_i:PX\to\two$ left adjoint, and therefore $\varphi\cong\varphi\cdot\Sup_X\cdot t_X\cong(\bigwedge_i\varphi_i)\cdot t_X\cong\bigwedge_i(\varphi_i\cdot t_X)$.
\begin{lemma}\label{LemmaSigmaInjective}
For each completely distributive space $X$ and $x,y\in X$ with $x\not\cong y$, $\sigma_X(x)\neq\sigma(y)$.
\end{lemma}
\begin{proof}
If, for instance, $y\notin\cl\{x\}$, then there exists some ``left adjoint closed subset'' $B\subseteq X$ with $y\notin B$ and $x\in B$.
\end{proof}
\begin{proposition}
For any $X\in\CDTOP_{\inf}$, $\sigma_X:X\to F\Pt(X)$ is an isomorphism.
\end{proposition}
\begin{proof}
We know alreay that $\sigma_X:X\to F\Pt(X)$ is bijective. To see continuity, notice that 
\[
 \sigma_X^{-1}(\varphi^\#)=\{x\in X\mid\varphi(x)=0\}
\]
for any $\varphi\in\Lambda(X)$. Let now $B\subseteq X$ be closed with left adjoint characteristic map $\varphi:X\to\two$. Then
\[
 \sigma_X(B)=\{\sigma_X(x)\mid x\in B\}=F\Pt(X)\setminus(\varphi^\#).
\]
Clearly, $\varphi\notin\sigma_X(x)$ for any $x\in B$. Let now $\frj\subseteq\Lambda(X)$ be an ideal with $\varphi\notin\frj$. One has $\frj=\sigma_X(x)$ for some $x\in X$ and, since $\varphi\notin\sigma_X(x)$, $x\in B$.
\end{proof}
\begin{lemma}
$\sigma=(\sigma_X)_X$ is a natural isomorphism $\sigma:1_{\CDTOP_{\inf}}\to F\Pt$.
\end{lemma}
\begin{proof}
We have to show the naturality condition. To this end, let $f:X\to Y$ in $\CDTOP_{\inf}$ with left adjoint $g:Y\to X$. We identify $\Lambda(X)$ with the set of all ``left adjoint closed subsets'' of $X$, and $\sigma_X(x)=\{A\in\Lambda(X)\mid x\notin A\}$. Then
\[
\downc\{g^{-1}(A)\mid x\notin A\}=\{B\in\Lambda(Y)\mid x\notin f^{-1}(B)\}
=\{B\in\Lambda(Y)\mid f(x)\notin B\}.\qedhere
\]
\end{proof}
\begin{theorem}
$F:\FRM_\wedge\to\CDTOP_{\inf}$ and $\Pt:\CDTOP_{\inf}\to\FRM_{\wedge}$ define an equivalence of categories.
\end{theorem}

\begin{corollary}
A topological space is equivalent to the filter space of some frame if and only if it is completely distributive.
\end{corollary}

Throughout we have emphasised that both $F$ and $\Pt$ are 2-functors, hence the subcategories of $\FRM_\wedge$ and $\CDTOP_{\inf}$ defined by the left adjoint morphisms are equivalent as well. Therefore

\begin{theorem}
$\FRM$ is equivalent to $\CDTOP$.
\end{theorem}

\begin{remark}\label{yes_schiz_obj}
The results of this section tell us that $\CDTOP$ is actually a very nice category: it is monadic over $\SET$. However, we have to take here the ``right'' forgetful functor $\CDTOP\to\SET$ (see also Remark \ref{no_schiz_obj}); namely the one which sends $X\in\CDTOP$ to the set of all right adjoint continuous maps of type $\two\to X$. Any such map sends necessarily $1$ to the top element of $X$, hence it is completely determined by the image of $0$. But note that, unlike in ordered sets, not every $x\in X$ defines a right adjoint via $0\mapsto x$. Therefore our result really extends the well-known fact that the canonical forgetful functor $\CCD\to\SET$ is monadic. I do not know yet if the corresponding functor $\CDAP\to\SET,\,X\mapsto\mathrm{LeftAdjoint}(X,[0,\infty])$ is monadic.
\end{remark}

\section{Continuous metric spaces}\label{SectionContMetSp}

Motivated by the well-known fact that the continuous lattices are precisely the injective topological spaces under the Scott topology, we call a metric space \emph{continuous} if it underlies an injective approach space. Our first goal is to show that this is indeed a property rather then an additional structure in the sense that there is at most one such approach space. More precise, we show that each injective approach space is a metric compact Hausdorff space where the compact Hausdorff topology is the Lawson topology of the underlying order of the metric. Certainly, one could argue that each separated injective approach space is a split subobject of a power of $[0,\infty]$, and use that $[0,\infty]$ is a metric compact Hausdorff space. Eventually, one obtains a concrete functor $\AP^{\mPsh}\to\AP^{\mU}$ which must be induced by a monad morphism $\mU\to\mPsh$. However, this argument uses the fact that $[0,\infty]$ is an initial cogenerator in $\AP$, but we do not know yet if the corresponding fact is true for $(\mT,\V)$-categories in general. Therefore we give here a different argument which does not rely on this property of $[0,\infty]$. To do so we start at the other end and present the monad morphism $\mU\to\mPsh$ right away. Recall that an approach space $X=(X,a)$ induces a metric $r:=Ua\cdot m_X^\circ$ on $UX$, and $r:UX\relto UX$ can be viewed as a $\mU$-relation $r:X\krelto UX$. This relation is actually a $\mU$-module $r:X\kmodto UX$ as one easily verifies:
\begin{align*}
r\kleisli a &= Ua\cdot m_X^\circ\cdot Ua\cdot m_X^\circ=r\cdot r=r,\text{ and}\\
(Ua\cdot m_X^\circ\cdot m_X)\kleisli r &=
Ua\cdot m_X^\circ\cdot m_X\cdot UUa\cdot Um_X^\circ\cdot m_X^\circ\\
&=Ua\cdot m_X^\circ\cdot Ua\cdot m_{UX}\cdot m_{UX}^\circ\cdot m_X^\circ=Ua\cdot m_X^\circ\cdot Ua\cdot m_X^\circ=r\cdot r=r.
\end{align*}
From that one obtains a contraction map $\yonedaT_X:UX\to PX$, which turns out to be the $X$-component of a natural transformation $U\to P$. To check naturality, let also $Y=(Y,b)$ be an approach space and $f:X\to Y$ be a contraction map. Furthermore, let $s:=Ub\cdot m_Y^\circ$ be the induced metric on $UY$ and not that
\[
U(f^*)\cdot m_X^\circ=Uf^\circ\cdot Ub\cdot m_X^\circ=Uf^\circ\cdot s=(Uf)^*,
\]
where $(Uf)^*$ is the module induced by the contraction map $Uf:UX\to UY$ between metric spaces. With this in mind, the left-lower path in
\[
\xymatrix{UX\ar[r]^{\yonedaT_X}\ar[d]_{Uf} & PX\ar[d]^{Pf}\\ UY\ar[r]_{\yonedaT_Y} & PY}
\]
sends $\frx$ to $s(-,Uf(\frx))=Uf^*(-,\frx)$, and the the upper-right path sends $\frx$ to
\[
\yonedaT_X(\frx)\kleisli f^*=r(-,\frx)\cdot Uf^*=Uf^*(-,\frx).
\]
Since also the triangle
\[
\xymatrix{UX\ar[rr]^{\yonedaT_X} && PX\\ & X\ar[ul]^{e_X}\ar[ur]_{\yoneda_X} &}
\]
commutes for each approach space $X$, we conclude that ``composition with $\yonedaT_X$'' induces a functor $\AP^{\mPsh}\to\AP^\mU$ and, consequently, $(\yonedaT_X)_X$ is a monad morphism. Here we use the following well-known fact.
\begin{proposition}\label{PropositionMonadMorphism}
Let $\mT=\monad$ and $\mT'=\monadb$ be monads on a category $\catfont{C}$, and let $d:T\to T'$ be a natural transformation. Then the following assertions are equivalent.
\begin{eqcond}
\item $d$ is a monad morphism from $\mT$ to $\mT'$.
\item For every $\mT'$-algebra $(X,T'X\xrightarrow{\alpha}X)$, $(X,TX\xrightarrow{d_X}T'X\xrightarrow{\alpha}X)$ is a $\mT$-algebra.
\item For every object $X$ in $\catfont{C}$, $(T'X,TT'X\xrightarrow{d_{T'X}}T'T'X\xrightarrow{m'_X}T'X)$ is a $\mT$-algebra.
\end{eqcond}
\end{proposition}

\begin{example}\label{ExamplePXisMetCompHaus}
Since $PX$ is cocomplete it also a metric compact Hausdorff space where the convergence $UPX\to PX$ sends $\frp\in UPX$ to $\yonedaT_{PX}(\frp)\kleisli(\yoneda_X)_*$ in $PX$. Recall from Lemma \ref{LemmaYoneda} that $(\yoneda_X)_*:X\kmodto PX$ is given by the evaluation relation $\ev:UX\relto PX,\,\ev(\frx,\psi)=\psi(\frx)$. Therefore, for any $\frx\in UX$, one has
\[
(\yonedaT_{PX}(\frp)\kleisli(\yoneda_X)_*)(\frx)
= U(\fspstr{-}{-})\cdot m_{PX}^\circ\cdot U\yoneda_X(\frx,\frp)\\
= U(\fspstr{-}{-}\cdot U\yoneda_X)\cdot m_{X}^\circ (\frx,\frp)=U\!\ev\cdot m_{X}^\circ (\frx,\frp).
\]
\end{example}

\begin{sideremark}\label{RemarksecondYoneda}
The contraction map $\yonedaT_X:UX\to PX$ can be seen as a ``second'' Yoneda embedding, in fact, as a function it is the co-restriction of the Yoneda embedding of the \emph{metric} space $UX$. Therefore the metric Yoneda lemma applies, but for this co-restriction an even stronger result holds: for $\frX\in UUX$ and $\psi\in PX$, $\fspstr{U\!\yonedaT_X(\frX)}{\psi}=\psi(m_X(\frX))$.
\end{sideremark}

Of course, all what was said so far applies \emph{mutatis mutandis} to topological spaces. Hence, for a (separated) injective space $X$ one gets a compact Hausdorff topology
\begin{equation}\label{Lawson}
\xymatrix@C=5em{UX\ar[r]^-{\yonedaT_X}\ar@/^2.3em/[rr]^{l_X} &PX\cong F_0(X)\ar[r]^-{\Sup_X} & X}
\end{equation}
which is known as the \emph{Lawson topology}. Furthermore, $l_X:UX\to X$ is characterised as being left adjoint to $e_X:X\to UX$ in $\TOP$ and sends each ultrafilter $\frx\in UX$ to its smallest convergence point which can be calculated as
\[l_X(\frx)=\bigvee_{A\in\frx}\bigwedge_{x\in A}x=\bigwedge_{A\in\frx}\bigvee_{x\in A}x.\]
From this formula one concludes that this convergence is already encoded in the underlying order, therefore the topology of $X$ can be recovered from the order structure alone. It also follows that, for injective space $X$ and $Y$, a monotone map $f:X\to Y$ (between the underlying ordered sets) is continuous provided that it preserves co-directed infima\footnote{Recall that we consider the dual of the specialisation order. We should also mention that continuity is even equivalent to preservation of these infima.}.

For an (separated) approach space $X=(X,a)$, we define $l_X$ as in \eqref{Lawson} and, with $a_0$ denoting the underlying metric of $X$, $a(\frx,x)=a_0(l_X(\frx),x)$. We show that $l_X$ is indeed the Lawson topology of the underlying topological space of $X$. It is tempting to argue here that, since $l_X\dashv e_X$ in $\AP$, one also has $l_X\dashv e_X$ in $\TOP$ and we are done. Unfortunately, we are not done since the underlying topological space of the approach space $UX$ is \emph{not} the topological space which comes from applying $U$ to the underlying topological space $X_t$ of $X$, in fact, the latter one has a coarser convergence (see Example \ref{DifferentOrder} below). At least we know that $l_X:U(X_t)\to X_t$ is continuous and, since
\[
a_0(l_X(\frx),x)=r(\frx,e_X(x))=a(\frx,x),
\]
one also has
\[
l_X(\frx)\le x\iff a_0(l_X(\frx),x)=0\iff a(\frx,x)=0\iff\frx\to x\iff\frx\le e_X(x),
\]
and the assertion follows. Here we use the fact that the underlying order of the underlying topology of $X$ coincides with the underlying order of the underlying metric of $X$. In conclusion, the approach structure of an injective approach space can be recovered form its underlying metric; and a contraction map between continuous metric spaces is a contraction map between the corresponding approach spaces if it preserves co-directed infima (i.e.\ if it is continuous with respect to the Scott-topologies of the underlying lattice). The full subcategory of $\AP$ consisting of all injective approach spaces we denote as $\ContMet$, it can be also viewed as a (non-full) subcategory of $\MET$.

\begin{example}\label{DifferentOrder}
We consider the approach space $[0,\infty]$ with $\lambda(\frx,x)=x-\xi(\frx)$ (see \ref{Sierpinski}). In the underlying topology,
\[
\frx\to x\iff 0=x-\xi(\frx)\iff \xi(\frx)\geqslant x.
\]
In particular, any interval $[0,u]$ is closed. Take now the filter base $\frg:=\{(1,1+\eps)\mid 0<\eps\}$ and let $\fry\in U[0,\infty]$ be with $\frg\subseteq\fry$. Then $\doo{1}\not\le\fry$ (since $[0,1]\notin\fry$) but $\delta(\doo{1},\fry)=0$ (since every $B\in\fry$ contains elements arbitrary close to $1$ from the right).
\end{example}

\begin{sideremark}
The metric space $[0,\infty]$ is continuous since it underlies the injective approach space $[0,\infty]$. Certainly, every continuous metric space is also a continuous lattice via its underlying order; however, it should be noted a continuous lattice (via its free metric) is in general not a continuous metric space. For instance, the Sierpi\'nski space $\two$ is not injective in $\AP$. To see this, just consider the embedding $\{0,\infty\}\hrw[0,\infty]$ and $f:\{0,\infty\}\to\two$ with $f(0)=\true$ and $f(\infty)=\false$, and observe that there is no contraction map $g:[0,\infty]\to\two$ extending $f$ since there exists $\frx\in U[0,\infty)$ with $\lambda(\frx,\infty)=0$.
\end{sideremark}

\begin{remark}
If $X$ is an injective approach space, then both its underlying metric and topological space are injective. Therefore $X$ is a metric compact Hausdorff space whose metric space is cocomplete and has a continuous underlying lattice; moreover, the compact Hausdorff topology is the Lawson topology of this lattice. We are wondering how far is this from a characterisation of a continuous metric space.
\end{remark}

We observed already that the approach space $[0,\infty]$ is actually a monoid in the monoidal category $\AP$ since addition $+$ is a contraction map $+:[0,\infty]\otimes[0,\infty]\to[0,\infty]$. Hence it induces a monad $\mM=(M,0,+)$ on $\AP$ where $M=-\otimes[0,\infty]$. For each approach space $X$, 
\[
\trans_X:X\otimes [0,\infty]\to PX, (u,x)\mapsto a(-,x)+u
\]
is a contraction map since it is the mate of the composite
\[
X^\op\otimes X\otimes [0,\infty]\xrightarrow{\,a\otimes 1\,} [0,\infty]\otimes [0,\infty]\xrightarrow{\,+\,}[0,\infty]
\]
of contraction maps. Thinking of $u\in[0,\infty]$ as a $\mU$-module $u:1\kmodto 1$, then $\trans_X(x,u)$ is the $\mU$-module $u\kleisli  x^*$. One easily confirms that the family $\trans=(\trans_X)_X$ is a monad morphism $\mM\to\mPsh$. Therefore each injective approach space admits an action
\[+:=\Sup_X\cdot\trans_X:X\otimes [0,\infty]\to X,\]
which satisfies
\[
a_0(x+u,y)=a_0(\Sup_X(u\kleisli x^*),y)=[u\kleisli x^*,y^*]=a(x,y)-u.
\]

Fixing $u\in[0,\infty]$, one obtains $t_u:X\to X,\,x\mapsto x+u$ in $\AP$. Recall that $a(\frx,x)=a_0(l_X(\frx),x)$, where $l_X\dashv e_X$ in $\AP$. Moreover, from 
\[
a_0(x,y)\geqslant a_0(x+u,y+u)=a_0(x,y+u)-u
\]
follows $a_0(x,y)+u\geqslant a_0(x,y+u)$, and hence also
\begin{align}\label{formula_a_plus_u}
a(\frx,y)+u=a_0(l_X(\frx),y)+u\geqslant a_0(l_X(\frx),y+u)=a(\frx,y+u).
\end{align}
For a numerical relation $\varphi:X\relto Y$ and $u\in[0,\infty]$, we write $\varphi\nplus u$ for the relation defined by $\varphi\nplus u(x,y):=\varphi(x,y)+u$. Note that $U(\varphi \nplus u)=U\varphi\nplus u$, and, given also $\psi:Y\relto Z$ and $v\in[0,\infty]$, $(\psi\nplus v)\cdot (\varphi\nplus u)=(\psi\cdot\varphi)\nplus(v+u)$. With this notation, the formula \eqref{formula_a_plus_u} reads as $a\nplus u\geqslant t_u^\circ\cdot a$, which allows us to conclude
\[
(Ua)\nplus u=U(a\nplus u)\geqslant Ut_u^\circ\cdot Ua,
\]
that is, $Ua(\frX,\frx)+u\geqslant Ua(\frX,Ut_u(\frx))$. Since $t_u$ is a contraction map one has $l_X\cdot Ut_u\le t_u\cdot l_X$ in the underlying order of $X$, and therefore
\[
a(Ut_u(\frx),x)=a_0(l_X\cdot Ut_u(\frx),x)\leqslant a_0(l_X(\frx)+u,x)=a(\frx,x)-u.
\]
We are now in position to prove

\begin{theorem}\label{TheoremInjApExp}
Each injective approach space is exponentiable in $\AP$.
\end{theorem}
\begin{proof}
Recall from \citep{Hof_ExpUnit} that an approach space $X=(X,a)$ is exponentiable if, for all $\frX\in UUX$ and $x\in X$ with $a(m_X(\frX),x)<\infty$, all $v,u\in[0,\infty)$ with $v+u=a(m_X(\frX),x)$ and all $\eps>0$ there exists an ultrafilter $\frx\in UX$ such that
\begin{align*}
Ua(\frX,\frx)\leqslant v+\eps &&\text{and}&& a(\frx,x)\leqslant u+\eps.
\end{align*}
Assume now that $X=(X,a)$ is injective in $\AP$, and let $\frX\in UUX$, $x\in X$ with $w:=a(m_X(\frX),x)<\infty$ and $u,v\in[0,\infty]$ with $u+v=w$. Put $\fry:=Ul_X(\frX)$ and $\frx:=Ut_u(\fry)$. Then
\begin{align*}
Ua(\frX,\frx) &\leqslant Ua(\frX,\fry)+u= u,\text{ and}\\
a(\frx,x) &\leqslant a(\fry,x)-u\\
&= a_0(l_X\cdot Ul_X(\fry),x)-u\\
&= a_0(l_X\cdot m_X(\fry),x)-u\\
&=w-u= v,
\end{align*}
and the assertion follows.
\end{proof}

\begin{sideremark}
In the proof above we do not need $X$ to be cocomplete, it is enough if $X$ admits suprema of $\mU$-modules of the form $\yonedaT_X(\frx)$ and $t(x,u)$. We will come back to this in Section \ref{SectionExamples}.
\end{sideremark}

With the same argument as in Remark \ref{CocomplExpMet} one can show that with $Y$ and $X$ also $Y^X$ and $Y\times X$ are injective approach space, hence
\begin{theorem}\label{ContMetCartClosed}
$\ContMet$ is Cartesian closed.
\end{theorem}
\begin{remark}
I do not know if in general a cocomplete $(\mT,\V)$-category is exponentiable.
\end{remark}
\begin{example}\label{ExampleActionPX}
The $[0,\infty]$-action on $PX$ sends $(\psi,u)$ to $u\kleisli\psi^*\kleisli(\yoneda_X)_*=u\kleisli\psi=\psi\nplus u$, therefore we write in the sequel $\nplus:PX\otimes[0,\infty]\to PX$. For later use we record already that it is not only a contraction map but even an $\mU$-algebra homomorphism, that is, the diagram
\[
\xymatrix{U(PX\otimes[0,\infty])\ar[r]^-{U\nplus}\ar[d]_\tau & UPX\ar[dd]^\alpha\\
UPX\otimes U[0,\infty]\ar[d]_{\alpha\otimes\xi}\\
PX\otimes[0,\infty]\ar[r]_-{\nplus} & PX}
\]
commutes (for $\alpha:UPX\to PX$ see Example \ref{ExamplePXisMetCompHaus}). To see this, let $\frq\in U(PX\otimes[0,\infty])$ with $\frp:=U\pi_1(\frq)\in UPX$ and $\fru:=U\pi_2(\frp)\in U[0,\infty]$ and let $\frx\in UX$. Since the diagram
\[
\xymatrix{UX\times PX\times [0,\infty]\ar[r]^-{1\times\nplus}\ar[d]_{\ev\times 1} &
UX\times PX\ar[d]^\ev\\ [0,\infty]\times[0,\infty]\ar[r]_+ & [0,\infty]}
\]
commutes, one obtains
\begin{align*}
\alpha\cdot U\nplus(\frq) &= \inf_{\frX,\,m_X(\frX)=\frx}U\ev(\frX,U\nplus(\frq))\\
&=\inf_{\frX,\,m_X(\frX)=\frx}\inf_{\stackrel{\frW\in U(UX\times PX\times[0,\infty])}{U\pi(\frW)=\frX,U\pi_{23}(\frW)=\frq}}\xi\cdot U\!\ev\cdot U(1\times\nplus)(\frW)\\
&=\inf_{\frX,\,m_X(\frX)=\frx}
\inf_{\stackrel{\frW\in U(UX\times PX\times[0,\infty])}{U\pi_1(\frW)=\frX,U\pi_{23}(\frW)=\frq}}
\xi\cdot U\!+(U(\ev\times 1)(\frW))\\
&=\inf_{\frX,\,m_X(\frX)=\frx}
\inf_{\stackrel{\frW\in U(UX\times PX\times[0,\infty])}{U\pi_1(\frW)=\frX,U\pi_{23}(\frW)=\frq}}
\xi\cdot U\!\ev(\tilde{\frW})+\xi(\fru)
&&(\tilde{\frW}=U\pi_{12}(\frW)\in U(UX\times PX))\\
&=\inf_{\frX,\,m_X(\frX)=\frx}
\inf_{\stackrel{\tilde{\frW}\in U(UX\times PX)}{U\pi_1(\tilde{\frW})=\frX,U\pi_2(\tilde{\frW})=\frp}}
\xi\cdot U\!\ev(\tilde{\frW})+\xi(\fru)\\
&=\alpha(\frp)(\frx)+\xi(\fru).
\end{align*}
\end{example}

\section{Everything is relative}\label{SectionRelative}

So far we have studied spaces which admit \emph{all} suprema; however, it is often desirable to limit the discussion to certain chosen ones. This is, for instance, the case in domain theory where one typically considers directed cocomplete ordered sets, and the ``directed version'' of complete distributivity is called continuity. The main point for us is here that many results are valid for both cases, one just has to write $JX$ (the ordered set of all directed down-sets) instead of $PX$ everywhere.

This suggests to start with a specification of certain $\mU$-modules, and to study spaces which admit all suprema of $\mU$-modules belonging to this specified class. This is indeed a well-known procedure in the context of enriched category theory, we refer to \citep{Kel_EnrCat,AK_Colim,KS_Colim,KL_MonCol}. A similar investigation of relative cocompleteness for $(\mT,\V)$-categories (hence for topological and approach spaces) was done in \citep{CH_CocomplII}. There seems to be no equal treatment of relative distributivity (or continuity) in the literature, but some initial steps are done in \citep{HW_AppVCat}. We also wish to point the reader to \citep{Stu_Dynamics} where an extensive study of complete distributivity in the context of quantaloid enriched categories can be found.

Following \citep{KS_Colim}, one might want to start with a collection $\Phi[X]$ of $\mU$-modules of type $X\kmodto 1$, for each space $X$, where $\Phi[X]$ contains all representable modules $x^*:X\kmodto 1$ ($x\in X$). Then a $\Phi$-weighted diagram in a space $X$ is given by a cont(inuous/ractive) map $d:D\to X$ and a $\mU$-module $\psi:D\kmodto 1$ in $\Phi[D]$. A colimit of such a diagram is an element $x\in X$ which represents $d_*\whiteleft\psi$, that is, $x_*=d_*\whiteleft\psi$. One calls $x$ a $\psi$-weighted colimit of $d$ and writes $x\simeq \colim (d,\varphi)$. One would then call a space $X$ $\Phi$-cocomplete if $X$ admits all $\Phi$-weighted colimits. Furthermore, a cont(inuous/ractive) map $f:X\to Y$ preserves the $\psi$-weighted colimit of $d$ if $f(\colim(\psi,d))\cong\colim(\psi,f\cdot d)$. If the family $\Phi[X]$ is functorial in the sense that, for all $f:X\to Y$ in $\TOP$/$\AP$ and all $\psi\in\Phi[X]$, $\psi\kleisli f^*\in\Phi[Y]$, then it is enough to consider weighted diagrams where $d$ is the identity $1_X:X\to X$ since the diagrams $(d:D\to X, \psi:D\kmodto 1)$ and $(1_X:X\to X, \psi\kleisli d^*:X\kmodto 1)$ share the same colimit. Finally, it is often convenient to assume that the family $\Phi[X]$ is \emph{saturated}, meaning that the inclusion map $i:\Phi[X]\to PX$ preserves $\Phi$-weighted colimits, for each space $X$. As we will see below, saturated implies functorial.

One would then call a space $X$ $\Phi$-cocomplete if $X$ admits all colimits weighted by some $\psi:X\kmodto 1$ in $\Phi[X]$. However, the situation for spaces is a bit more complicated then the one for enriched categories as it can be seen already in the case $\Phi[X]=PX$ all $\mU$-modules of type $X\kmodto 1$. If $X$ is cocomplete, then $\Sup_X:PX\to X$ calculates for each weighted diagram $1_X:X\to X$, $\psi:X\kmodto 1$ in $X$ a colimit $\Sup_X(\psi)$, however, already for topological spaces the existence of all weighted colimits does not guarantee cocompleteness of $X$. In fact, \citep{HW_AppVCat} presents an example of a topological spaces $X$ which admits all suprema of $\mU$-modules of type $X\kmodto 1$ but $X$ is not cocomplete. The problem here is that the induced map $PX\to X,\psi\mapsto x$ does not need to be cont(inuous/ractive), and therefore is in general only a right adjoint to $\yoneda_X:X\to PX$ in $\ORD$. The situation changes if we allow $\mU$-modules $\psi:D\kmodto A$ in the definition of weighted colimits, where $A$ might be different from the one-point space $1$. A \emph{colimit} of such a diagram is now a cont(inuous/ractive) map $g:A\to X$ which represents $d_*\whiteleft\psi$, that is, $g_*=d_*\whiteleft\psi$. With this modification it is indeed true that $X$ is cocomplete if and only if $X$ admits all colimits. In other words, $X$ admits ``continuously'' suprema of all $\mU$-modules $\psi:X\kmodto 1$ if and only if $X$ admits colimits of all $\mU$-modules $\psi:X\kmodto A$. 

\begin{example}[Composition as a colimit]\label{comp_as_colim}
Let $\varphi:X\kmodto Y$ and $\psi:Y\kmodto Z$ be $\mU$-modules, and consider the diagram
\[
\xymatrix{Y\ar@{-^{>}}|-{\object@{o}}[d]_\psi\ar[r]^{\mate{\varphi}} & PX.\\ Z}
\]
Then $\colim(\mate{\varphi},\psi)=\mate{\psi\kleisli\varphi}$. I learnt this fact from \citep{Stu_Hausdorff}.
\end{example}

Therefore what we need is not just a choice of $\mU$-modules of type $X\kmodto 1$, but rather a class $\Mod{\Phi}$ of $\mU$-modules $\varphi:X\kmodto Y$. One possibility is to extend the given family $\Phi[X]$ to such a class by defining, for $\varphi:X\kmodto Y$ in $\Mod{\mU}$,
\[
\varphi:X\kmodto Y\text{ in }\Mod{\Phi}\;\text{ if }\;\forall y\in Y\,.\,y^*\kleisli \varphi\in\Phi[X].
\]
Note that, for any cont(inuous/ractive) map $f:Z\to Y$, the $\mU$-module $f^*$ belongs to $\Mod{\Phi}$, and $f^*\kleisli\varphi$ is in $\Mod{\Phi}$ whenever $\varphi:X\kmodto Y$ is in $\Mod{\Phi}$. In \citep{Stu_Hausdorff} it is shown (in the context of quantaloid-enriched categories, but the argument is based on Example \ref{comp_as_colim} and therefore adapts easily to our case) that the family $\Phi[X]$ is saturated if and only if $\Mod{\Phi}$ is actually a subcategory of $\Mod{\mU}$. In \citep{CH_CocomplII} we went the other way around and started with a class $\Mod{\Phi}$ of $\mU$-modules containing all $\mU$-modules of the form $f^*$, closed under certain compositions (see below), and such that
\begin{equation}\label{AxPhi}
(\forall y\in Y\,.\,y^*\kleisli \varphi\in\Mod{\Phi})\iff \varphi\in\Mod{\Phi}
\end{equation}
for all $\varphi: X\kmodto Y\in\Mod{\mU}$. Note that \eqref{AxPhi} guarantees already that $\Mod{\Phi}$ is closed under compositions of the form $f^*\kleisli\varphi$. Combining \citep{Stu_Hausdorff} with \citep{CH_CocomplII} gives
\begin{theorem}\label{TheoremSatComp}
Assume that a family $\Phi[X]$ of $\mU$-modules of type $X\kmodto 1$ ($X$ in $\TOP$ or $\AP$) is given, and define $\Mod{\Phi}$ as above. Then the following assertions are equivalent.
\begin{eqcond}
\item The family $\Phi[X]$ is saturated.
\item $\Mod{\Phi}$ is a subcategory of $\Mod{\mU}$.
\item For all $\psi:X\kmodto 1$ in $\Phi[X]$ and all cont(inuous/ractive) maps $f:X\to Y$ and $g:Y\to X$ where $g_*\in\Mod{\Phi}$,
\begin{align*}
\psi\kleisli f^*\in\Phi[Y] &&\text{and}&& \psi\circ g_*\in\Phi[Y].
\end{align*}
\end{eqcond}
\end{theorem}
\begin{proof}
By definition, $\varphi:X\kmodto Y$ belongs to $\Mod{\Phi}$ if and only if $\mate{\varphi}:Y\to PX$ factors through $\Phi[X]\hrw PX$. Assume (i) and let $\varphi:X\kmodto Y$ and $\psi:Y\kmodto Z$ be in $\Mod{\Phi}$. Then $\mate{z^*\kleisli\psi\kleisli\varphi}:1\to PX$ factors through $\Phi[X]\hrw PX$, for each $z\in Z$, hence $\psi\kleisli\varphi$ belongs to $\Mod{\Phi}$. The implication (ii)$\Rw$(iii) is clear, so assume now (iii). Since $\Mod{\Phi}$ is closed under compositions of the form $\varphi\kleisli f^*$, it is enough to show that $i:\Phi[X]\to PX$ preserves suprema of $\mU$-modules of type $\Phi[X]\kmodto 1$ in $\Mod{\Phi}$. Let $\psi:\Phi[X]\kmodto 1$ be in $\Mod{\Phi}$. Then the colimit of $i$ and $\psi$ in $PX$ is given by $\psi\kleisli i^*\kleisli(\yoneda_X)_*\in\Phi[X]$.
\end{proof}

Due to the considerations above, throughout this section we assume that a subcategory $\Mod{\Phi}$ of $\Mod{\mU}$ is given which satisfies \eqref{AxPhi} and contains $f^*$ for every cont(inuous/ractive) $f:X\to Y$. Following the nomenclature of \citep{CH_CocomplII}, a cont(inuous/ractive) map $f:X\to Y$ is called \emph{$\Phi$-dense} if $f_*\in\Mod{\Phi}$, and a topological/approach space $X$ is called \emph{$\Phi$-injective} if it is injective w.r.t.\ $\Phi$-dense embeddings. Furthermore, we define 
\[
\Phi X=\{\psi\in PX\mid \psi\in\Mod{\Phi}\}
\]
as a subspace of $PX$. One verifies easily that the Yoneda embedding $\yoneda:X\to PX$ corestricts to a $\Phi$-dense mapping $\yoneda_X^\Phi:X\to\Phi X$. For each $\mU$-module $\varphi:X\kmodto Y$, $\varphi\in\Mod{\Phi}$ if and only if its mate $\mate{\varphi}:Y\to PX$ factors through the embedding $\Phi X\hrw PX$.

For a $\mU$-module $\varphi:X\kmodto Y$ in $\Mod{\Phi}$, the cont(inuous/ractive) map $-\kleisli\varphi:PX\to PY$ sends $\psi\in\Phi X$ to $\psi\kleisli\varphi\in\Phi X$ and therefore restricts to $-\kleisli\varphi:\Phi X\to\Phi Y$. In particular, $Pf:PX\to PY$ restricts to $\Phi f:\Phi X\to \Phi Y$ since $f^*\in\Mod{\Phi}$. The right adjoint $-\kleisli f_*$ of $Pf$ restricts to a right adjoint of $\Phi f$ if $f$ is $\Phi$-dense. In fact, it is shown in \citep{CH_CocomplII} that $f$ is $\Phi$-dense if and only if $\Phi f$ has a right adjoint.

A topological/approach space $X$ is \emph{$\Phi$-cocomplete} if and only if $\yoneda_X^\Phi:X\to\Phi X$ has a left adjoint $\Sup_X^\Phi:\Phi X\to X$, or, equivalently, if $X$ has all weighted colimits where the weight $\psi:D\kmodto A$ belongs to $\Mod{\Phi}$. One obtains at once that a $\Phi$-cocomplete space $X$ is $\Phi$-injective, an extension of $f:A\to X$ along the $\Phi$-dense embedding $i:A\to B$ is given by $\colim (f,i_*)$. In turn, $\Phi$-injectivity of $X$ gives a cont(inuous/ractive) map $\Sup_X^\Phi:\Phi X\to X$ as an extension of $1_X:X\to X$ along $\yoneda_X^\Phi:X\to\Phi X$ which turns out to be left adjoint to $\yoneda_X^\Phi$ in $\TOP$/$\AP$.

A cont(inuous/ractive) map $f:X\to Y$ is called \emph{$\Phi$-cocontinuous} if it preserves all $\Phi$-weighted colimits which exist in $X$. The following results can be found in \citep{CH_CocomplII}.

\begin{proposition}
Let $f:X\to Y$ a cont(inuous/ractive) maps between $\Phi$-cocomplete spaces.
\begin{enumerate}
\item $f$ is $\Phi$-cocontinuous if and only if $f\cdot \Sup^\Phi_X\cong \Sup^\Phi_Y\cdot\Phi f$.
\item $f$ is $\Phi$-cocontinuous and $\Phi$-dense if and only if $f$ is left adjoint.
\end{enumerate}
\end{proposition}

\begin{corollary}\label{PhiXcocompl}
For each space $X$, $\Phi X$ is $\Phi$-cocomplete where $\Sup_{\Phi X}^\Phi=-\kleisli(\yoneda_X^\Phi)_*$. Furthermore, the inclusion map $\Phi X\hrw PX$ is $\Phi$-cocontinuous.
\end{corollary}
As in the absolute case, the subcategory $\Cocts{\Phi}_\sep$ of $\TOP$/$\AP$ consisting of $\Phi$-cocomplete T$_0$-spaces and $\Phi$-cocontinuous morphisms is reflective with the Yoneda embedding as universal arrow. Furthermore, the inclusion functor $\Cocts{\Phi}_\sep\to\TOP$/$\AP$ is monadic. The induced monad $\mPhi=\Phimonad$ is of Kock-Z\"oberlein type and has $\Phi$ as functor, the Yoneda embeddings $\yoneda_X^\Phi:X\to\Phi X$ as units and $\yonmult_X^\Phi:=-\kleisli(\yoneda_X^\Phi)_*:\Phi\Phi X\to \Phi X$ as multiplications.

\begin{theorem}\label{TheoremPhiKleisli}
The category $\Mod{\Phi}$ is dually equivalent to the Kleisli category $\TOP_{\mPhi}$/$\AP_{\mPhi}$ of $\mPhi=\Phimonad$ on $\TOP$/$\AP$.
\end{theorem}
\begin{proof}
We have seen already that $\mU$-modules $X\kmodto Y$ in $\Mod{\Phi}$ are in bijection with cont(inuous/ractive) maps of type $Y\to\Phi X$, where the identity distributor $a:X\kmodto X$ corresponds to the Yoneda embedding $\yoneda_X^\Phi:X\to\Phi X$. Let now $\varphi:X\kmodto Y$ and $\psi:Y\kmodto Z$ be $\mU$-modules in $\Phi$. By Example \ref{comp_as_colim},
\[
 \mate{\psi\kleisli\varphi}=\colim(\mate{\varphi},\psi)
=\Sup_{\Phi X}^\Phi\cdot\Phi\mate{\varphi}\cdot\mate{\psi}
=\yonmult_X^\Phi\cdot\Phi\mate{\varphi}\cdot\mate{\psi}.\qedhere
\]
\end{proof}

Both $\Mod{\Phi}$ and $\TOP_{\mPhi}$/$\AP_{\mPhi}$ are actually ordered categories, and the equivalence above is indeed a 2-equivalence.

The notion of complete distributivity generalises in an obvious to this relative case, and was studied in this context under the name ``continuity'' in \citep{HW_AppVCat}. One naturally expects that the proofs of Section \ref{SectionDuality} can be adapted to this case leading to a duality theorem for ``$\Phi$-algebraic spaces''. It is the aim of this section to show that this is indeed the case.

More general, R.~Rosebrugh and R.J. Wood showed in \citep{RW_CCD4} that the category $\CCD_\textrm{sup}$ of constructive complete distributive lattices and suprema preserving maps is equivalent to the idempotent splitting completion $\kar(\REL)$ of the category $\REL$ of sets and relations, as well as to the idempotent splitting completion $\kar(\MOD)$ of the category $\MOD$ of ordered sets and modules. Later on, in \citep{RW_Split} they observed that this theorem is ``not really about lattices'' but rather a special case of a much more general result about ``a mere monad $\monadfont{D}$ on a mere category $\catfont{C}$''.
\begin{theorem}[\cite{RW_Split}]\label{TheoremSplit}
Let $\monadfont{D}$ be a monad on a category $\catfont{C}$ where idempotents split. Then
\[
 \kar(\catfont{C}_\monadfont{D})\cong\Spl(\catfont{C}^\monadfont{D}).
\]
Here $\catfont{C}_\monadfont{D}$ denotes the Kleisli and $\catfont{C}^\monadfont{D}$ the Eilenberg--Moore category of $\monadfont{D}$.
\end{theorem}
We recall that an idempotent morphism $e:X\to X$ in a category $\catfont{A}$ splits if $e=s\cdot r$ for $r:X\to Y$ and $s:Y\to X$ in $\catfont{A}$ with $r\cdot s=1_Y$. One says that idempotents split in $\catfont{A}$ if every idempotent is of this form. Most ``everyday'' categories have this property since $s$ can be taken as the equaliser of $e$ and $1_X$ and necessarily $r$ as the induced morphism, or $r$ as the coequaliser of $e$ and $1_X$ and $s$ as the induced morphism; supposing here that these (co)limits exist. The arguably most prominent example of a (highly) non-complete category is $\REL$, and for instance the idempotent relation $<:\R\relto\R$ does not split in $\REL$. In any case, the idempotent splitting completion $\kar(\catfont{A})$ of $\catfont{A}$ has as objects pairs $(X,e)$ where $e$ is idempotent, and a morphism $f:(X,e)\to (X',e')$ in $\kar(\catfont{A})$ is an $\catfont{A}$-morphism $f:X\to X'$ such that $e'\cdot f=f=f\cdot e$. The category $\catfont{A}$ is fully embedded into $\kar(\catfont{A})$ via $X\mapsto(X,1_X)$, all idempotents split in $\kar(\catfont{A})$ and it is indeed the free idempotent splitting completion of $\catfont{A}$. To explain the latter, let $F:\catfont{A}\to\catfont{B}$ be a functor where idempotents split in $\catfont{B}$. One can construct now the (essentially unique) extension $\tilde{F}:\kar(\catfont{A})\to\catfont{B}$ as follows. For any object $(X,e)$ in $\kar(\catfont{A})$, define $\tilde{F}(X,e)$ as the idempotent splitting $FX\xrightarrow{r}\tilde{F}(X,e)\xrightarrow{s}FX$ of the idempotent $Fe$ in $\catfont{B}$; and for a morphism $f:(X,e)\to (X',e')$ in $\kar(\catfont{A})$ put $\tilde{F}f=r'\cdot Ff\cdot s$ where $r'$ and $s'$ split $Fe'$.

Since idempotents split in $\catfont{C}$, idempotents also split in $\catfont{C}^\monadfont{D}$. The objects of $\Spl(\catfont{C}^\monadfont{D})$ are triples $(X,\alpha,t)$ where $(X,\alpha)$ is an Eilenberg--Moore algebra for $\monadfont{D}$ and $t:X\to DX$ is an algebra homomorphism into the free algebra with $\alpha\cdot t=1_X$. The morphisms of $\Spl(\catfont{C}^\monadfont{D})$ are just the algebra homomorphism between the (underlying) algebras. Consequently, if an algebra $(X,\alpha)$ admits splittings $t,t':X\to DX$ then the identity map is an isomorphism between $(X,\alpha,t)$ and $(X,\alpha,t')$. Hence we might as well think of $\Spl(\catfont{C}^\monadfont{D})$ as the full subcategory of $\catfont{C}^\monadfont{D}$ defined by those algebras $(X,\alpha)$ which admit a splitting $t:X\to DX$ in $\catfont{C}^\monadfont{D}$. Note that, as shown in \citep{RW_Split}, if $\monadfont{D}$ is of Kock-Z\"oberlein type, then $(X,\alpha)$ admits at most one splitting which is necessarily left adjoint to $\alpha$.

A category where idempotents split is sometimes also called Cauchy complete, due to the fact that in the language of modules both properties (for categories and metric spaces respectively) are instances of the same definition. Therefore many properties we know about Cauchy completion of metric spaces are shared by $\kar(\catfont{A})$, for instance:

\begin{lemma}
Let $\catfont{A}$ be a full subcategory of $\catfont{B}$ and assume that idempotents split in $\catfont{B}$. Let $\overline{\catfont{A}}$ be the full subcategory of $\catfont{B}$ defined by the retracts of the objects in $\catfont{A}$. Then idempotents split in $\overline{\catfont{A}}$ and $\catfont{A}\to\overline{\catfont{A}}$ is the free idempotent splitting completion of $\catfont{A}$.  
\end{lemma}
\begin{proof}
Every idempotent in $\overline{\catfont{A}}$ splits in $\catfont{B}$ and the splitting belongs to $\overline{\catfont{A}}$. By definition, every $B$ in $\overline{\catfont{A}}$ splits some idempotent $e:A\to A$ in $\catfont{A}$. If $B$ splits also $e':A'\to A'$ in $\catfont{A}$, so that $A\xrightarrow{r}B\xrightarrow{s}A$ and  $A'\xrightarrow{r'}B\xrightarrow{s}A'$ with $e=sr$, $rs=1_B$ and $e'=s'r'$, $r's'=1_B$, then $s'r:(A,e)\to(A',e')$  and $sr':(A',e')\to(A,e)$ are inverse to each other in $\kar(\catfont{A})$. Choosing for every $B$ in $\overline{\catfont{A}}$ such an idempotent $e:A\to A$ in $\catfont{A}$ defines the object part of a functor $G:\overline{\catfont{A}}\to\kar(\catfont{A})$, which sends a morphism $f:B\to B'$ in $\overline{\catfont{A}}$ to $s' f r:(A,e)\to(A',e')$ in $\kar(\catfont{A})$. With $F:\kar(\catfont{A})\to\overline{\catfont{A}}$ denoting a functor induced by the universal property, one verifies that both $GF$ and $FG$ are naturally isomorphic to the identity.
\end{proof}

Clearly, every algebra $(X,\alpha)$ which admit a splitting $t:X\to DX$ is a retract of the free algebra $DX$. Vice versa, if $(X,\alpha)$ is a retract of a free algebra, then $(X,\alpha)$ is projective with respect to those morphisms in $\catfont{C}^\monadfont{D}$ which are split epimorphisms in $\catfont{C}$, hence $\alpha:DX\to X$ admit a splitting $t:X\to DX$. Consequently, $\Spl(\catfont{C}^\monadfont{D})$ is the free idempotent splitting completion of full subcategory of $\catfont{C}^\monadfont{D}$ defined by the free algebras which is known to be equivalent to $\catfont{C}_\monadfont{D}$, and Theorem \ref{TheoremSplit} follows. 

Our principal object of interest here is the monad $\mPhi=\Phimonad$ on $\TOP$/$\AP$. We know already that the category of Eilenberg--Moore algebras of $\mPhi$ has $\Phi$-cocomplete T$_0$-spaces as objects, and $\Phi$-cocontinuous cont(inuous/ractive) maps as morphisms. The objects of $\Spl(\TOP^{\mPhi})$ respectively $\Spl(\AP^{\mPhi})$ are those $\Phi$-cocomplete T$_0$-spaces $X$ where $\Sup_X^\Phi:\Phi X\to X$ has a left adjoint adjoint. In the sequel we call such a space \emph{$\Phi$-distributive}, and denote the category of $\Phi$-distributive T$_0$-spaces and $\Phi$-cocontinuous cont(inuous/ractive) maps as $\DTop{\Phi}_{\sup}$/$\DApp{\Phi}_{\sup}$.

Combining Theorem \ref{TheoremSplit} with Theorem \ref{TheoremPhiKleisli} yields

\begin{theorem}\label{TheoremKarDistrib}
$\kar(\Mod{\Phi})^\op\cong\DTop{\Phi}_{\sup}$/$\DApp{\Phi}_{\sup}$.
\end{theorem}

Of course, the equivalence above is induced by the equivalence $\varphi:X\kmodto X'\mapsto(-\kleisli\varphi):\Phi X'\to\Phi X$ between $\Mod{\Phi}^\op$ and the full subcategory of $\Cocts{\Phi}_\sep$ defined by the free algebras. Accordingly, the corresponding functors
\begin{align*}
 S:\kar(\Mod{\Phi})^\op\to\DTop{\Phi}_{\sup}/\DApp{\Phi}_{\sup} &&\text{and}&&
 I:\DTop{\Phi}_{\sup}/\DApp{\Phi}_{\sup}\to \kar(\Mod{\Phi})^\op
\end{align*}
can be constructed as follows. For $(X,\theta)$ in $\kar(\Mod{\Phi})$, let $\Phi X\xrightarrow{r}S(X,\theta)\xrightarrow{s}\Phi X$ be a splitting of the idempotent $-\kleisli\theta:\Phi X\to\Phi X$; to have something concrete,
\[
 S(X,\theta)=\{\psi\in\Phi\mid\psi\kleisli\theta=\psi\},
\]
$r:\Phi X\to S(X,\theta),\,\psi\mapsto\psi\kleisli\theta$ and $s:S(X,\theta)\to\Phi X$ is the inclusion functor. If $\varphi:(X,\theta)\to (X',\theta')$, then $S\varphi:S(X,\theta)\to S(X',\theta')$ sends $\psi\in S(X,\theta)$ to $\psi\kleisli\theta$. Let now $X$ be a $\Phi$-distributive T$_0$-space with $\yoneda_X^\Phi\vdash\Sup_X^\Phi\vdash t$. Then $t:X\to\Phi X$ corresponds to a module $\theta:X\kmodto X$ in $\Mod{\Phi}$ which is necessarily idempotent. Furthermore, $\Phi X\xrightarrow{\Sup_X^\Phi} X\xrightarrow{\;t\;}\Phi X$ splits the idempotent $-\kleisli\theta:\Phi X\to \Phi X$, and therefore $I(X)$ can be taken as $(X,\theta)$. Accordingly, for $f:X\to X'$ one calculates now $I(f)=\theta'\kleisli f^*\kleisli\theta$, in the sequel we denote $\theta'\kleisli f^*\kleisli\theta$ also as $f^\#$. Note that both functors $S$ and $I$ are actually 2-functors.

For a $\Phi$-distributive T$_0$-space $X$, the natural isomorphism $X\cong SI(X)$ stems from the fact that both $X$ and $S(X,\theta)$ split the idempotent $-\kleisli\theta:\Phi X\to\Phi X$. Hence,
\begin{align*}
 X\to S(X,\theta),\,x\mapsto x^*\kleisli\theta &&\text{and}&&
S(X,\theta)\to X,\,\psi\mapsto\Sup_X^\Phi(\psi)
\end{align*}
are inverse to each other. Certainly, also $(X,\theta)\cong IS(X,\theta)$ for every $(X,\theta)$ in $\kar(\Mod{\Phi})$, but to describe the natural isomorphism $(X,\theta)\kmodto IS(X,\theta)$ we need some notation.

For $(X,\theta)$ in $\kar(\Mod{\Phi})$ we define $\widehat{\theta}=(-\kleisli\theta)\cdot\mate{\theta}:X\to S(X,\theta)$, which is indeed just the corestriction of $\mate{\theta}:X\to\Phi X$ to $S(X,\theta)$. Furthermore, we put ${\widehat{\theta}}_+=\widehat{\theta}_*\kleisli\theta$ and $\widehat{\theta}^+=\theta\kleisli\widehat{\theta}^*$. Note that $\widehat{\theta}^+\kleisli\widehat{\theta}_+=\theta$ since $\widehat{\theta}^*\kleisli \widehat{\theta}_*={\mate{\theta}}^*\kleisli{\mate{\theta}}_*=\fspstr{U\mate{\theta}(-)}{\mate{\theta}(-)}=\theta\whiteleft\theta$ by Theorem \ref{TheoremGenYoneda}, idempotency of $\theta$ gives $\theta\le\theta\whiteleft\theta$, and therefore $\theta=\theta\kleisli\theta\kleisli\theta\le\theta\kleisli(\theta\whiteleft\theta)\kleisli\theta\le\theta\kleisli\theta=\theta$. One easily verifies that the suprema in $S(X,\theta)$ are given by
\[
 \Sup_{S(X,\theta)}^\Phi:\Phi S(X,\theta)\to S(X,\theta),\,\Psi\mapsto \Psi\kleisli\widehat{\theta}_+,
\]
and the left adjoint of $\Sup_{S(X,\theta)}^\Phi$ by
\[
 t:S(X,\theta)\to \Phi S(X,\theta),\,\psi\mapsto\psi\kleisli\widehat{\theta}^+.
\]
Therefore $t\cdot \Sup_{S(X,\theta)}^\Phi$ sends $\psi$ to $\psi\kleisli\widehat{\theta}_+\kleisli\widehat{\theta}^+$, hence $t=\mate{\omega}$ for $\omega=\widehat{\theta}_+\kleisli\widehat{\theta}^+$. Since $S(X,\theta)$ splits both $-\kleisli\theta$ and $-\kleisli\omega$,
\[
\xymatrix@C=1em{\Phi X\ar[rr]^{-\kleisli\theta}\ar[dr] && \Phi X\\
 & S(X,\theta)\ar[ur]\ar[dr]^t\\
\Phi S(X,\theta)\ar[ur]^{\Sup_{S(X,\theta)}^\Phi}\ar[rr]_ {-\kleisli\omega}&& \Phi S(X,\theta)}
\]
$(X,\theta)$ and $(S(X,\theta),\omega)$ are naturally isomorphic in $\kar(\Mod{\Phi})$ via
\begin{align*}
\widehat{\theta}_+:(X,\theta)\kmodto (S(X,\theta),\omega)
&&\text{and}&&
\widehat{\theta}^+:(S(X,\theta),\omega)\kmodto (X,\theta).
\end{align*}
Finally, we note that the diagrams
\begin{align*}
\xymatrix@C=4em{(X,\theta)\ar@{-^>}[r]|-{\object@{o}}^\varphi
\ar@{-^>}[d]|-{\object@{o}}_{\widehat{\theta}_+}
& (X',\theta')\ar@{-^>}[d]|-{\object@{o}}^{\widehat{\theta'}_+}\\
(S(X,\theta),\omega)\ar@{-^>}[r]|-{\object@{o}}_{(-\kleisli\varphi)^\#}
& (S(X',\theta'),\omega')}
&&
\xymatrix@C=3em{X\ar[r]^f\ar[d]_{\widehat{\theta}} & X'\ar[d]^{\widehat{\theta'}}\\
S(X,\theta)\ar[r]_{-\kleisli f^\#} & S(X',\theta')}
\end{align*}
commute, for $\varphi:(X,\theta)\to(X',\theta')$ in $\kar(\Mod{\Phi})$ and $f:X\to Y$ in $\DTop{\Phi}_{\sup}$/$\DApp{\Phi}_{\sup}$.

Following \citep{RW_CCD4} (and motivated by \citep{Vic_Infosys}) we consider the category $\Inf{\Phi}$ whose objects are pairs $(X,\theta)$ where $X$ is a topological/approach space and $\theta:X\kmodto X$ is an idempotent relation in $\Mod{\Phi}$, and whose morphism $f:(X,\theta)\to(X',\theta')$ are $\Phi$-dense cont(inuous/ractive) maps $f:X\to X'$ satisfying $\theta(\frx,x)\le\theta'(Uf(\frx),f(x))$, for each $\frx\in UX$ and $x\in X$, that is, $\theta\le f^*\kleisli\theta'\kleisli f_*$ or, equivalently, $f_*\kleisli\theta\le\theta'\kleisli f_*$.
\begin{example}
For each $(X,\theta)$ in $\kar(\Mod{\Phi})$, $\widehat{\theta}:X\to S(X,\theta)$ is a morphism in $\Inf{\Phi}$.
\end{example}
To each $f:(X,\theta)\to(X',\theta')$ in $\Inf{\Phi}$ we associate modules
\begin{align*}
f_+=\theta'\kleisli f_*\kleisli\theta &&\text{and}&& f^+=\theta\kleisli f_*\kleisli\theta'
\end{align*}
in $\Mod{\Phi}$, and then $f_+:(X,\theta)\kmodto(X',\theta')$ and $f^+:(X',\theta')\kmodto(X,\theta)$ are morphisms in $\kar(\Mod{\Phi})$ and $f_+\dashv f^+$. Furthermore, these constructions define functors
\begin{align*}
(-)_+:\Inf{\Phi}\to \kar(\Mod{\Phi}) &&\text{and}&& (-)^+:\Inf{\Phi}^\op\to\kar(\Mod{\Phi})
\end{align*}
with $X_+=X=X^+$. For a $\Phi$-distributive space $X=(X,a)$ we consider the equaliser $i:A\to X$ of $\yoneda_X^\Phi,\mate{\theta}:X\to\Phi X$ and observe that, for $\frx\in UA$ and $x\in A$,
\[
a(\frx,x)=\fspstr{U\yoneda_X^\Phi(\frx)}{\yoneda_X^\Phi(x)}
=\fspstr{U\yoneda_X^\Phi(\frx)}{\mate{\theta}(x)}=\theta(\frx,x).
\]
Hence $i:(A,a)\to(X,\theta)$ lives in $\Inf{\Phi}$, and we have $i_+\dashv i^+$ in $\kar(\Mod{\Phi})$, but also $i_+:A\kmodto X\dashv i^+:X\kmodto A$ in $\Mod{\mU}$ since $\theta\le a$. Furthermore,
\[
 i^+=i^*\kleisli\theta=\theta(-,i(-))=\fspstr{-}{i(-)}=i^*,
\]
and hence also $i_+=i_*$.

We call a $\Phi$-distributive T$_0$-space $X$ \emph{$\Phi$-algebraic} if $X$ is isomorphic to a space of form $\Phi Y$. Moving to the other side of the equivalence, $X$ is $\Phi$-algebraic if and only if $(X,\theta)$ is isomorphic to some $(Y,(1_Y)_*)$ in $\kar(\Mod{\Phi})$.  Let $X$ be $\Phi$-algebraic, and assume that $\alpha:(Y,(1_Y)_*)\kmodto(X,\theta)$ and $\beta:(X,\theta)\kmodto(Y,(1_Y)_*)$ are inverse to each other in $\kar(\Mod{\Phi})$. As above one verifies that $\alpha:Y\kmodto X$ is left adjoint to $\beta:X\kmodto Y$ in $\Mod{\Phi}$, and, since $X$ is $\Phi$-cocomplete, $\alpha=f_*$ and $\beta=f^*$ for some cont(inuous/ractive) map $f:Y\to X$. Furthermore, $f$ equalises $\yoneda_X^\Phi,\mate{\theta}:X\to\Phi X$ since $f^*\kleisli\theta=f^*$. We write $i:A\to X$ for the equaliser of $\yoneda_X^\Phi,\mate{\theta}:X\to\Phi X$, and $h:Y\to A$ for the map induced by $f$. Then $f^*=h^*\kleisli i^*$, hence $h^*=f^*\kleisli i_*$ and
\begin{align*}
i^*\kleisli f_*\kleisli h^*&= i^*\kleisli f_*\kleisli f^*\kleisli i_*
=i^*\kleisli\theta\kleisli i_*=i^*\kleisli i_*=(1_A)_*,\\
f_*\kleisli h^*\kleisli i^*&=f_*\kleisli f^*=\theta\le (1_X)_*.
\end{align*}
Therefore $f_*\kleisli h^*\dashv i^*$ in $\Mod{\mU}$, which implies $i_*=f_*\kleisli h^*\in\Mod{\Phi}$. Clearly, $i^*\kleisli i_*=(1_A)_*$, but also $i_*\kleisli i^*=\theta$ since
\[
\theta=f_*\kleisli f^*=i_*\kleisli h_*\kleisli h^*\kleisli i^*\le i_*\kleisli i^*=i_+\kleisli i^+\le\theta.
\]
\begin{proposition}
Let $X$ a $\Phi$-distributive T$_0$ space, and $i:A\to X$ be the equaliser of $\yoneda_X^\Phi,\mate{\theta}:X\to\Phi X$. Then $X$ is $\Phi$-algebraic if and only if $i$ is $\Phi$-dense and $i_*\kleisli i^*=\theta$.
\end{proposition}
The full subcategory of $\DTop{\Phi}$ respectively $\DApp{\Phi}$ determined by the $\Phi$-algebraic spaces we denotes as  $\ATop{\Phi}$ and $\AApp{\Phi}$ respectively.
\begin{theorem}
$\Mod{\Phi}$ is dually equivalent to $\ATop{\Phi}_{\sup}/\AApp{\Phi}_{\sup}$.
\end{theorem}
The functor $S:\Mod{\Phi}^\op\to\ATop{\Phi}/\AApp{\Phi}$ is just the restriction of $S:\kar(\Mod{\Phi})^\op\to\DTop{\Phi}/\DApp{\Phi}$, its inverse $C:\ATop{\Phi}/\AApp{\Phi}\to\Mod{\Phi}^\op$ substitutes $(X,\theta)$ by the isomorphic $(A,(1_A)_*)$ where $i:A\to X$ denotes the equaliser of $\yoneda_X^\Phi,\mate{\theta}:X\to\Phi X$, and accordingly sends $f:X\to X'$ to the restriction of $f^*$ to $A$ and $A'$, that is, to $i^*\kleisli f^*\kleisli {i'}_*$.
\begin{lemma}
For $X$ in $\Mod{\Phi}$, the equaliser of $\Phi(\yoneda_X^\Phi),\yoneda_{\Phi X}^\Phi:\Phi X\to\Phi\Phi X$ is given by
\[
 \tilde{X}_\Phi:=\{\psi\in\Phi X\mid \psi:X\kmodto 1 \text{ is right adjoint in }\Mod{\mU}\}\hrw\Phi X
\]
\end{lemma}
We write $\eta_X^\Phi:X\to CS(X)$ for the restriction of the Yoneda embedding $\yoneda_X^\Phi$ to $\tilde{X}_\Phi$, then the isomorphism $X\kmodto CS(X)$ is given by $(\eta_X^\Phi)_*$. For a $\Phi$-algebraic space $X$, the isomorphism $SC(X)\to X$ is the restriction of $\Sup_X^\Phi$ to $\Phi A$.

Since both $S$ and $C$ are indeed 2-functors, we obtain immediately that the category $\map(\Mod{\Phi})$ of left adjoint modules in $\Mod{\Phi}$ is dually equivalent to the category $\ATop{\Phi}/\AApp{\Phi}$ of $\Phi$-algebraic spaces and right adjoint $\Phi$-cocontinuous cont(inuous/ractive) maps between them. We call a T$_0$-space $X$ \emph{$\Phi$-sober} if each $\varphi:Y\kmodto X$ in $\map(\Mod{\Phi})$ is of the form $\varphi=f_*$ for some (unique) $f:Y\to X$. Note that each space of the form $\Phi X$ is $\Phi$-sober. More importantly for us, also $\tilde{X}_\Phi$ is $\Phi$-sober which can be seen as follows. For any $\Psi:\tilde{X}_\Phi\kmodto 1$ in $\Mod{\Phi}$ which is right adoint in $\Mod{\mU}$ put $\psi=\Psi\kleisli(\eta_X^\Phi)_*$, then $\psi\in\tilde{X}_\Phi$ and  $\Psi=\psi\kleisli(\eta_X^\Phi)^*=\psi^*\kleisli(\eta_X^\Phi)_*\kleisli(\eta_X^\Phi)^*=\psi^*$. We write $\Sob{\Phi}$ for the category of $\Phi$-sober spaces and $\Phi$-dense maps, the considerations above imply that $(-)_*:\Sob{\Phi}\to\map(\Mod{\Phi})$ is an equivalence of categories. Therefore
\begin{theorem}\label{TheoremDualSobAlg}
$\Sob{\Phi}$ is dually equivalent to $\ATop{\Phi}/\AApp{\Phi}$.
\end{theorem}
It is high time to present examples.

\section{Examples}\label{SectionExamples}

The main purpose of this section is to describe some possible choices of $\Mod{\Phi}$, to explain why they (might) lead to interesting classes of spaces, and in some of these case to spell out the meaning of the duality theorems of the previous sections. We have to admit right at the beginning that, unfortunately, we do not have yet intrinsic topological discription of $\Phi$-distributivity or $\Phi$-algebraicity other then the relationship of distributivity with frames exhibited in Section \ref{SectionFramesDist}. Nevertheless, we hope to be able to convince the reader that it is at least desireable to have such descriptions.

In the topological case, we know that $\mPsh$ is isomorphic to the filter monad on $\TOP$. Consequently, the monad $\mPhi$ corresponding to $\Mod{\Phi}$ is isomorphic to a submonad of the filter monad, which puts us in the context of \citep{EF_SemDom} where many semantic domains are identified as the algebras for certain submonads of the filter monad. In \citep{CH_CocomplII} we showed already how the defining properties of these submonads translate into the language of modules. As one of the virtues of this ``module approach'' we see the fact it automatically provides us with metric and other variants of these monads and therefore of these topological domains. It was also observed there that many of these examples can be described in a uniform manner as follows: take $\Mod{\Phi}$ as the category all those modules $\varphi:X\kmodto Y$ where ``$\varphi$-colimits commute with certain limits'' \citep{KS_Colim}, that is, where the monotone/contractive map
\[
 \varphi\kleisli-:\Mod{\mU}(1,X)\to \Mod{\mU}(1,Y)
\]
preserves chosen limits.

\subsection{The absolute case}
Certainly we can choose no limits at all, and hence $\Mod{\Phi}=\Mod{\mU}$ is the category of all $\mU$-modules. The results of the previous section restate Theorem \ref{TheoremAbsDuality} and, more general, tell us that the category $\CDTOP_{\sup}$ respectively $\CDAP_{\sup}$ of completely distributive T$_0$-spaces and left adjoint cont(inuous/ractive) maps is dually equivalent to the idempotent splitting completion $\kar(\Mod{\mU})$ of $\Mod{\mU}$, and that the category $\TATOP_{\sup}$ respectively $\TAAP_{\sup}$ of totally algebraic T$_0$-spaces and left adjoint cont(inuous/ractive) maps is dually equivalent to $\Mod{\mU}$.

\subsection{The ``inhabited'' case}

Our next example is $\Mod{\Phi}$ being the category of all $\mU$-modules $\varphi:X\to Y$ where $\varphi\kleisli-$ preserves the top element, we call such a $\mU$-modules \emph{inhabited}. Explicitly, $\varphi:X\kmodto Y$ is inhabited if and only if
\begin{align*}
 \forall y\in Y\,\exists\frx\in UX\,.\,\frx \varphi y
&&\text{resp.}&&
 0\geqslant\sup_{y\in Y}\inf_{\frx\in UX}\varphi(\frx,y).
\end{align*}
A continuous map $f$ between topological spaces is $\Phi$-dense if and only if $f$ is dense in the usual topological sense, and a topological space $X$ is $\Phi$-cocomplete if and only if $X$ is densly injective, that is, a Scott domain \citep{Book_ContLat}. Correspondingly, we call a contraction map $f:X\to Y$ between approach spaces $X=(X,a)$ and $Y=(Y,b)$ \emph{dense} if $f$ is $\Phi$-dense, that is, if
\[
 0\geqslant\inf_{\frx\in UX} b(Uf(\frx),y)
\]
for all $y\in Y$. Every right adjoint $\mU$-module is inhabited, hence a topological/approach space is $\Phi$-sober precisely if it is sober. The results of the previous section tells us now that the category $\SOB_\textrm{dense}$ respectively  $\ASOB_\textrm{dense}$ of sober spaces and dense maps is dually equivalent to the category of ``inhabited algebraic'' spaces and right adjoint cont(inuous/ractive) maps which preserve inhabited suprema.

\subsection{The prime case}

One can go further and consider $\Mod{\Phi}$ being the category of all $\mU$-modules $\varphi:X\kmodto Y$ where $\varphi\kleisli-$ preserves finite or countable suprema, or even all weighted limits. The latter case is not so interesting here since for this choice a $\mU$-module $\varphi$ belongs to $\Mod{\Phi}$ if and only if $\varphi$ is right adjoint. Colimits weighted by right adjoints are absolute, that is, ever cont(inuous/ractive) map preserves them. Moreover, a T$_0$-space $X$ is $\Phi$-cocomplete if and only if $X$ is $\Phi$-distributive if and only if $X$ is $\Phi$-algebraic if and only if $X$ is sober. Consequently, Theorem \ref{TheoremDualSobAlg} gives us the flash news that the category of sober spaces and left adjoint cont(inuous/ractive) maps is dually equivalent to the category of sober spaces and right adjoint cont(inuous/ractive) maps.

The first case, on the other hand, seems to be more promising. First of all, we find it interesting that this definition, applied to metric spaces, yields forward Cauchy completeness as shown in \citep{Vic_LocComplI}: for a metric space $X$, the modules $\psi:X\modto 1$ where $\psi\cdot-$ preserves finite infima correspond precisely to forward Cauchy filters, and $x$ is a supremum of $\psi$ if and only if $x$ is a limit point of the corresponding filter. Turning now to the topological case, the induced monad $\mPhi$ on $\TOP$ is isomorphic to the prime filter (of opens) monad which we encountered already in Section \ref{SectionDualSpace}. Recall from Section \ref{SectionDualSpace} that $\TOP^{\mPhi}$ is equivalent to the category $\ORDCH_\sep$ of anti-symmetric ordered compact Hausdorff space and monotone continuous maps. These spaces are also known under the designation \emph{stably compact} as they are precisely those spaces which are sober, locally compact, and have the property that its compact down-sets\footnote{Recall that our underlying order is dual to the specialisation order} are closed under finite intersections. As usual, it is enough to require stability under empty and binary intersections, and stability under empty intersection translates to compactness of $X$. Note that a T$_0$-space is locally compact if and only if it is core-compact if and only if it is exponentiable. With an eye on the approach case, we remark that it follows from ``general abstract non-sense'' that every $\Phi$-cocomplete T$_0$-space $X$ has these properties. In fact, $X$ is sober since $\Mod{\Phi}$ contains all right adjoint modules, and $X$ is exponentiable respectively $+$-exponentiable by Proposition \ref{PropTensExp}. For a stably compact space $X$ and $A\subseteq X$, $A$ is a compact down-set if and only if the characteristic map $\varphi:X\to\two$ of its complement is a morphism in $\ORDCH_\sep$, that is, $\varphi$ is monotone and preserves smallest convergence points of ultrafilters (or, equivalently, of prime filters of opens). Since both maps
\begin{align*}
\bigvee:\two^n\to\two &&\text{and}&& \inf:[0,\infty]^n\to [0,\infty]
\end{align*}
are left adjoints in $\TOP$ and $\AP$ respectively, and since both inclusion functors $\ORDCH_\sep\hrw\TOP$ and $\METCH_\sep\hrw\AP$ preserve products, we conclude that $\bigvee:\two^n\to\two$ and $\inf:[0,\infty]^n\to [0,\infty]$ are morphisms in $\ORDCH_\sep$ and $\METCH_\sep$ respectively. Therefore the supremum of maps $\varphi_i:X\to\two$ ($1\le i\le n$) in $\ORDCH_\sep$ is again in $\ORDCH_\sep$, and likewise for the metric case.

For a stably compact topological space $X$ one can easily show that every prime filter $\psi$ has a smallest convergence point $x$, and the map $\psi\mapsto x$ from $\Phi X$ to $X$ is indeed continuous. In order to explain this better we make us of a slightly different but equivalent description of stably compact topological space used in \citep{Sim82a}. There a space $X$ is called stable if, for open subsets $U_1,\ldots,U_n$ and $V_1,\ldots,V_n$ ($n\in\N$) of $X$ with $U_i\ll\V_i$ for each $1\le i\le n$, also $\bigcap_i U_i\ll\bigcap_i V_i$. Of course it is enough to consider only $n=0$ and $n=2$, and for $n=0$ the condition reduces to $X\ll X$, that is, $X$ is compact. Then $X$ is called well-compact if it is sober, core-compact and stable. It is shown \citep[Lemma 3.7]{Sim82a} that, for a core-compact and stable space $X$, the set of limit points of a prime filter is irreducible. Hence, if $X$ is in addition sober, every prime filter has a smallest convergence point. Furthermore, \citep[Lemma 3.9]{Sim82a} states that the induced map $\Phi X\to X$ is continuous\footnote{Actually, it is even shown there that this map is well-compact, which in the language of the this paper means that it is $\Phi$-cocontinuous. But since the monad $\mPhi$ is of Kock-Z\"oberlein type, we know that, once it is continuous, it is even left adjoint.}. At the end of the next subsection we provide a different argument for this fact which also works for approach spaces. If $X$ is only weakly sober, then this map is only defined up to equivalence, but in fact any chosen map $\Phi X\to X$ is continuous. It is now clear that, without assuming the T$_0$-axiom, a topological space $X$ is $\Phi$-cocomplete if and only if $X$ is weakly sober, exponentiable and stable.

A stably compact space is called \emph{spectral} (or \emph{coherent}) if the compact down-sets form a basis for the topology of $X$. One easily verifies that each space of the form $\Phi X$ is spectral, and with an argument similar to the one used before Lemma \ref{LemmaSigmaInjective} one shows that every $\Phi$-distributive space is spectral. Unfortunately, I do not know yet if the converse is also true.

A continuous map $f:(X,a)\to(Y,b)$ between topological spaces is $\Phi$-dense if it is dense in a very strong sense: for each $y\in Y$, there must exist a largest ultrafilter $\frx\in UX$ with $Uf(\frx)\to y$. For lack of a better name we call these maps ultra-dense. The general results of \citep{CH_CocomplII} tell us that a topological T$_0$-space is stably compact if and only if it is injective with respect to ultra-dense embeddings. 
Furthermore, by Theorem \ref{TheoremDualSobAlg}, the category $\SOB_{\textrm{ultra-dense}}$ of sober spaces and ultra-dense maps is dually equivalent to the category of $\Phi$-algebraic spaces (which are very special spectral spaces) and right adjoint continuous maps which preserve smallest convergence points of ultrafilters.

\subsection{The ultrafilter case}

A closely related example one obtains using $\yonedaT_X:UX\to PX$ of Section \ref{SectionContMetSp}: for a space $X$, let $\Phi[X]$ be the image of $\yonedaT_X$. Of course, for topological spaces one gets the prime filter monad discussed above, but the situation is different for approach spaces. For every $\mU$-module of the form $\yonedaT_X(\frx_0):X\kmodto 1$ one has
\[
 \yonedaT_X(\frx_0)\kleisli\varphi=\xi\cdot U\varphi(\frx_0)
\]
for all $\varphi:1\kmodto X$ and therefore $\yonedaT_X(\frx_0)\kleisli-:\Mod{\mU}(1,X)\to[0,\infty]$ preserves finite sup's. Furthermore, using Remark \ref{RemarkXi} one shows that 
\[
\yonedaT_X(\frx_0)\kleisli(\hom(u,\varphi))\geqslant
\hom(u,\yonedaT_X(\frx_0)\kleisli\varphi) 
\]
for every $\varphi:1\kmodto X$ and $u\in[0,\infty]$. Since for every contraction map $\Mod{\mU}(1,X)\to[0,\infty]$ one has the reverse inequality, we conclude that $\yonedaT_X(\frx_0)\kleisli-$ preserves even the operation $\hom(u,-)$ on $\Mod{\mU}(1,X)$. This begs the question if every module $\psi:X\kmodto 1$ where $\psi\kleisli-$ preserves all finite sup's and ``homing'' with all $u\in[0,\infty]$ is of the form $\psi=\yonedaT_X(\frx)$ for some $\frx\in UX$. If this is the case it follows that the corresponding class $\Mod{\Phi}$ of $\mU$-modules is a subcategory of $\Mod{\mU}$ (see Theorem \ref{TheoremSatComp}); however, since we do not know this yet we present a different argument.

Recall that the functor $M_0:\AP\to\MET$ sends $X=(X,a)$ to $M_0(X)=(UX,\tilde{a})$ where $\tilde{a}=Ua\cdot m_X^\circ$. More general, for an arbitrary $\mU$-relation $\varphi:X\krelto Y$ we define $\tilde{\varphi}=U\varphi\cdot m_X^\circ:UX\relto UY$. Given also $\psi:Y\krelto Z$, then
\[
\widetilde{\psi\kleisli\varphi}=U\psi\cdot UU\varphi\cdot Um_X^\circ\cdot m_X^\circ
=U\psi\cdot UU\varphi\cdot m_{UX}^\circ\cdot m_X^\circ
=U\psi\cdot m_X^\circ\cdot U\varphi\cdot m_X^\circ=\tilde{\psi}\cdot\tilde{\varphi}.
\]
Consequently, if $\varphi:X\kmodto Y$ is $\mU$-module, then $\tilde{\varphi}:M_0(X)\modto M_0(Y)$ is a module between metric spaces. We also remark that $\varphi$ can be seen as a module $\varphi:M_0(X)\modto Y_0$. By definition, $\varphi:X\kmodto Y$ belongs to $\Mod{\Phi}$ if there is a function\footnote{guaranteed by the Axiom of Choice} $f:Y\to UX$ such that
\[
 \varphi=\tilde{a}(-,f(-))=f^\circ\cdot\tilde{a}=f^*.
\]
Note that $f:M_0(X)\to Y_0$ is necessarily contractive since $f^*=\varphi$ is a module between metric spaces. Let now $\varphi:(X,a)\kmodto(Y,b)$ and $\psi:(Y,b)\to (Z,c)$ be in $\Mod{\Phi}$ with $\psi=g^*$ and $\varphi=f^*$. Then
\begin{multline*}
 \psi\kleisli\varphi=g^\circ\cdot\tilde{b}\cdot U\varphi\cdot m_X^\circ
=g^\circ\cdot\tilde{b}\cdot\tilde{\varphi}
=g^\circ\cdot\widetilde{b\kleisli\varphi}
=g^\circ\cdot\tilde{\varphi}
=g^\circ\cdot U\varphi\cdot m_X^\circ\\
=g^\circ\cdot Uf^\circ\cdot UUa\cdot Um_X^\circ\cdot m_X^\circ
=g^\circ\cdot Uf^\circ\cdot m_X^\circ \cdot Ua\cdot m_X^\circ
=(m_X\cdot Uf\cdot g)\cdot\tilde{a}=(m_X\cdot Uf\cdot g)^*.
\end{multline*}

By definition, the corresponding monad $\mPhi$ appears in the (epi,mono) factorisation $U\twoheadrightarrow\mPhi\rightarrowtail P$ of the monad morphism $\yonedaT:U\to P$, and the monad morphism $U\twoheadrightarrow\mPhi$ induces full embeddings $\AP^{\mPhi}\to\METCH$. By the ``second Yoneda lemma'' (Remark \ref{RemarksecondYoneda}), $\yonedaT_X:UX\to PX$ is fully faithful, hence $UX\twoheadrightarrow\mPhi X$ is a quotient map, in fact, $UX\twoheadrightarrow\mPhi X$ gives the T$_0$-reflection of $UX$. If $X$ is a separated metric compact Hausdorff space, then the universal property of $UX\twoheadrightarrow\mPhi X$ provides us with  an inverse $\Sup_X^\Phi:\Phi X\to X$ of $\yoneda_X^\Phi:X\to\Phi X$. We conclude that $\AP^{\mPhi}$ is equivalent to the category of separated metric compact Hausdorff space.

Given an approach space $X=(X,a)$ which is a $\Phi$-algebra, then $X$ is $+$-exponentiable by Proposition \ref{PropTensExp}. Furthermore, the structure map $\alpha:UX\to X$ picks, for each ultrafilter $\frx$, a supremum of the $\mU$-module $\yonedaT_X(\frx):X\kmodto 1$, that is, a point $\alpha(\frx)\in X$ such that, for each $x\in X$, $a(\frx,x)=a_0(\alpha(\frx),x)$. Conversely, assume now that an approach space $X=(X,a)$ admits all suprema of $\mU$-module $\yonedaT_X(\frx):X\kmodto 1$ where $\frx\in UX$. Let $l:UX\to X$ be any map which chooses a supremum of $\yonedaT_X(\frx)$, for each $\frx\in UX$. Then $l:UX\to X$ is a morphism in $\MET$ but in general not in $\AP$. However, if $X$ is in addition $+$-exponentiable, then $l$ is indeed a morphism in $\AP$. To see this, recall from \citep{Hof_TopTh} that $+$-exponentiability of $X$ is equivalent to commutativity of 
\[
\xymatrix{UUX\ar[d]_{m_X}\ar[r]|-{\object@{|}}^{Ua} & UX\ar[d]|-{\object@{|}}^a\\
 UX\ar[r]|-{\object@{|}}_a & X}
\]
in $\NREL$. Then, with $a=a_0\cdot l$, one obtains
\begin{multline*}
l\cdot Ua\cdot m_X^\circ\cdot m_X
\le a_0\cdot l\cdot Ua\cdot m_X^\circ\cdot m_X
=a\cdot m_X
=a\cdot Ua\\
=a\cdot Ua_0\cdot Ul
=a\cdot Ua\cdot Ue_X\cdot Ul
\le a\cdot Ua\cdot m_X^\circ\cdot Ul
\le a\cdot Ul.
\end{multline*}
We conclude that an approach space $X$ is $\Phi$-cocomplete if and only if $X$ is $+$-exponentiable and, for each ultrafilter $\frx\in UX$, there exists a point $x_0\in X$ such that $a(\frx,x)=a_0(x_0,x)$, for all $x\in X$.

\subsection{The tensor case}
We discuss briefly a further example which is only relevant for the approach case. For any approach space $X=(X,a)$, we put $\Phi[X]$ to be the set of all $\mU$ modules $\psi:X\kmodto 1$ where $\psi=u\kleisli x^*$ where $x\in X$ and $u\in[0,\infty]$. Hence, for $\frx\in UX$, $\psi(\frx)=a(\frx,x)+u$. In order to see that $\Mod{\Phi}$ is closed under compositions in $\Mod{\mU}$, it seems to be more convenient to make use of Theorem \ref{TheoremSatComp} and prove that $\Mod{\Phi}$ is closed under the two types of compositions specified there. In fact, for a contractive map $f:X\to Y$ one has
\[
 \psi\kleisli f^*=u\kleisli x^*\kleisli f^*=u\kleisli f(x)^*,
\]
and for $g:Y\to X$ with $g_*$ in $\Mod{\Phi}$ one obtains
\[
 \psi\kleisli g_*=u\kleisli x^*\kleisli g_*=u\kleisli v\kleisli y^*=(v+u)\kleisli y^*,
\]
where $x^*\kleisli g_*=v\kleisli y^*$. The monad $\mPhi$ corresponding to this choice of modules is closely related to the monad $\mM=(M,0,+)$ induced by the monoid $[0,\infty]=([0,\infty],+,0)$ in the monoidal category $\AP$ since there is a monad morphism $\trans:\mM\to\mPhi$ described before Theorem \ref{TheoremInjApExp}.

\subsection{The ultra-and-tensor case}

The proof of Theorem \ref{TheoremInjApExp} suggests to consider a combination of the two previous examples. Given $X$ in $\AP$, we define $\Phi[X]$ as the set of all $\mU$-modules $\psi:X\kmodto 1$ of the form $\psi=\yonedaT_X(\frx)\nplus u$ for some $\frx\in UX$ and $u\in[0,\infty]$ (see Section \ref{SectionContMetSp}, before Theorem \ref{TheoremInjApExp}). Hence, a $\mU$-module $\varphi:X\kmodto Y$ between approach spaces $X=(X,a)$ and $Y=(Y,b)$ belongs to $\Mod{\Phi}$ precisely if there exist functions $h:Y\to UX$ and $\alpha:Y\to[0,\infty]$ with
\[
 \varphi(\frx,y)=\tilde{a}(\frx,h(y))+\alpha(y),
\]
for all $\frx\in UX$ and $y\in Y$. As above, we use Theorem \ref{TheoremSatComp} to show that $\Mod{\Phi}$ is closed under compositions in $\Mod{\mU}$. Let $X=(X,a)$ and $Y=(Y,b)$ be approach spaces and assume that $\psi:X\kmodto 1$ belongs to $\Phi[X]$ with corresponding $\frx_0\in UX$ and $u\in [0,\infty]$. For $f:X\to Y$ in $\AP$ and $\fry\in UY$ one has
\[
\psi\kleisli f^*(\fry)
=\psi\cdot(Uf^\circ\cdot\tilde{b})(\fry)
=\inf_{\frx\in UX}\tilde{b}(\fry,Uf(\frx))+\tilde{a}(\frx,\frx_0)+u
=\tilde{c}(\frz,Uf(\frx_0))+u.
\]
Therefore $\psi\kleisli f^*$ belong to $\Mod{\Phi}$. Let now that $g:Y\to X$ be in $\AP$ such that $g_*:Y\kmodto X$ is in $\Mod{\Phi}$, witnessed by $k:X\to UY$ and $\beta:X\to[0,\infty]$. Hence, for all $\fry\in UY$ and $x\in X$,
\[
 a\cdot Ug(\fry,x)=g_*(\fry,x)=\tilde{b}(\fry,k(x))+\beta(x).
\]
To see that $\psi\kleisli g_*$ belongs to $\Phi[Y]$, observe first that, for a numerical relation $r:X\relto Y$, a function $\gamma:Y\to[0,\infty]$ and for $s(x,y)=r(x,y)+\gamma(y)$, one has
\[
 Us(\frx,\fry)=Ur(\frx,\fry)+\xi\cdot U\gamma(\frx)
\]
for all $\frx\in UX$ and $\fry\in UY$, where $\xi:U[0,\infty]\to[0,\infty]$ is defined as $\xi(\fru)=\sup_{A\in\fru}\inf A$. From this one concludes
\begin{align*}
\tilde{a}(Ug(\fry),\frx) &= Ua\cdot m_X^\circ\cdot Ug(\fry,\frx)\\
&= U(a\cdot Ug)\cdot m_Y^\circ(\fry,\frx)\\
&= \inf_{\frY,\,m_Y(\frY)=\fry} U(a\cdot Ug)(\frY,\frx)\\
&=\inf_{\frY,\,m_Y(\frY)=\fry} U\tilde{b}(\frY,Uk(\frx))+\xi\cdot U\beta(\frx)\\
&=U\tilde{b}\cdot m_Y^\circ(\fry,Uk(\frx))+\xi\cdot U\beta(\frx)\\
&=\tilde{b}(\fry,m_Y\cdot Uk(\frx))+\xi\cdot U\beta(\frx),
\end{align*}
and finally obtains
\[
\psi\kleisli g_*(\fry)=\psi(Ug(\fry))=\tilde{a}(Ug(\fry),\frx_0)+u
=\tilde{b}(\fry,m_Y\cdot Uk(\frx_0))+\xi\cdot U\beta(\frx_0)+u,
\]
for all $\fry\in UY$.

Both contraction maps
\begin{align*}
\trans_X:X\otimes [0,\infty]\to PX,\, (u,x)\mapsto a(-,x)+u &&\text{and}&&
\yonedaT_X:UX\to PX,\, \frx\mapsto\tilde{a}(-,\frx)
\end{align*}
factor through $\Phi X\hrw PX$, and this is all one needs to make the proof of Theorem \ref{TheoremInjApExp} work. Therefore every $\Phi$-cocomplete approach space is exponentiable. This also raises the question if the monad $\mPhi$ is the image of the composite of the ultrafilter monad $\mU=\umonad$ on $\AP$ and the monad $\mM=(M,0,+)$. It is well-known that in general the composition of monads does not lead to a monad again; however, it does if we have a distributive law at hand. Recall that a distributive law \citep{Beck_DistLaws} of $\mU$ over $\mM$ is a natural transformation $d:UM\to MU$ such that the diagrams
\begin{equation}\label{DistLawDiag}
\xymatrix{
  & U\ar[dl]_{U0}\ar[dr]^{0_U} & & UMM\ar[r]^{d_M}\ar[d]_{U+} & MUM\ar[r]^{Md} & MMU\ar[d]^{+_U}\\
UM\ar[rr]^d & & MU & UM\ar[rr]^d && MU\\
& M\ar[ul]^{e_M}\ar[ur]_{Me} & & UUM\ar[u]^{m_M}\ar[r]_{Ud} & UMU\ar[r]_{d_U} & MUU\ar[u]_{Mm}
}
\end{equation}
commute. Each contractive map $\xi:U[0,\infty]\to[0,\infty]$ defines a natural transformation $d:UM\to MU$ where $d_X$ is the composite
\[
U(X\otimes[0,\infty])\xrightarrow{\can} UX\otimes U[0,\infty]
\xrightarrow{1_X\otimes\xi}UX\otimes[0,\infty],
\]
and vice versa, each natural transformation $d:UM\to MU$ comes from a unique contraction map $\xi:U[0,\infty]\to[0,\infty]$. Moreover, the natural transformation $d:UM\to MU$ associated to $\xi$ is a distributive law of $\mU$ over $\mM$ if and only if the diagrams
\begin{align}\label{StrictAxioms}
\xymatrix{[0,\infty]\ar[r]^{e_{[0,\infty]}}\ar[dr]_{1_{[0,\infty]}} & U[0,\infty]\ar[d]^\xi\\ & [0,\infty]}
&&
\xymatrix{UU[0,\infty]\ar[r]^{m_{[0,\infty]}}\ar[d]_{U\xi} & U[0,\infty]\ar[d]^\xi\\ U[0,\infty]\ar[r]_\xi & [0,\infty]}
\\
\xymatrix{U1\ar[r]^-{U0}\ar[d]_{!} & U[0,\infty]\ar[d]^\xi \\ 1\ar[r]_-{0} & [0,\infty]}
&&
\xymatrix{U([0,\infty]\times[0,\infty])\ar[r]^-{U+}\ar[d]_{\langle\xi\cdot U\pi_1,\xi\cdot U\pi_2\rangle} & U[0,\infty]\ar[d]^\xi\\ [0,\infty]\times[0,\infty]\ar[r]_-{+} & [0,\infty]}\notag
\end{align}
commute. In fact, one easily verifies that $d=(d_X)_X$ is a natural transformation $d:UM\to MU$, and that commutativity of the diagrams \eqref{DistLawDiag} correspond precisely to commutativity of the diagrams \eqref{StrictAxioms}. Let now $d=(d_X)_X:UM\to MU$ be a natural transformation. Then, thanks to naturality of $d$, for each set $X$ we have that $\pi_2\cdot d_X:\beta(X\times[0,\infty])\to[0,\infty]$ is equal to $\xi\cdot U(\pi_2)$ where $\xi=\pi_2\cdot d_1:U[0,\infty]\to[0,\infty]$. On the other hand, $\pi_2\cdot d$ induces a natural transformation $UM\to U$ between $\SET$-functors where $UM:\SET\to\SET$ preserves finite sums. Since there is exactly one natural transformation $UM\to U$ (see \citep{Bor_SumUlt}), we conclude $\pi_1\cdot d=\beta(\pi_1)$. In particular, distributive laws of $\mU$ over $\mM$ correspond to $\mU$-algebra structures on the approach space $[0,\infty]$. Since $\mU$ in $\AP$ is of Kock-Z\"oberlein type and $[0,\infty]$ is separated there is only one such structure, namely $\xi:U[0,\infty]\to[0,\infty],\,\fru\mapsto\sup_{A\in\fru}\inf A$, and $\xi$ makes indeed all diagrams \eqref{StrictAxioms} commutative.

The algebras for the composite monad $\mM\diamond\mU$ on $\AP$ can be described as pairs $(X,*)$ where $X$ is a metric compact Hausdorff space and $*:X\otimes[0,\infty]\to X$ is a $[0,\infty]$-action on $X$ in $\METCH$, and a homomorphism is a $\mU$-homomorphism which preserves the action. There is a canonical natural transformation $\yonedaT^{[0,\infty]}:\mM\diamond\mU\to\mPsh$ whose $X$-component is the composite
\[
UX\otimes[0,\infty]\xrightarrow{\yonedaT_X\otimes 1}
PX\otimes[0,\infty]\xrightarrow{\nplus}PX
\]
(see also Example \ref{ExampleActionPX}). It follows now from Proposition \ref{PropositionMonadMorphism} that $\yonedaT^{[0,\infty]}$ is a monad morphism. In fact, for an approach space $X$, composing $\yonmult_X:PPX\to PX$ with $\yonedaT^{[0,\infty]}_{PX}$ gives
\[
UPX\otimes[0,\infty]\xrightarrow{\alpha\otimes 1}
PX\otimes[0,\infty]\xrightarrow{\nplus}PX
\]
since both diagrams
\[
\xymatrix{
UPX\otimes[0,\infty]\ar[r]^{\yonedaT^{PX}\otimes 1}\ar[dr]_{\alpha\otimes 1} &
PPX\otimes[0,\infty]\ar[r]^-{\nplus}\ar[d]_{\yonmult_X\otimes 1} &
PPX\ar[d]^{\yonmult_X}\\
& PX\otimes[0,\infty]\ar[r]_-{\nplus} & PX}
\]
commute. Hence, the $[0,\infty]$-action on $PX$ is given by $\nplus:PX\otimes[0,\infty]\to PX$ which is indeed a morphism in $\METCH$ by Example \ref{ExampleActionPX}.

\subsection{Monads over $\SET$}

So far we have exploited the fact that the category $\Cocts{\Phi}$ is monadic over $\TOP$ respectively $\AP$. However, under further conditions on $\Mod{\Phi}$, $\Cocts{\Phi}$ is also monadic over $\SET$, and therefore Theorem \ref{TheoremSplit} applies to the induced monad on $\SET$. To finish this paper we briefly discuss this case.

Recall from \citep{CH_CocomplII} that $\Cocts{\Phi}$ is monadic over $\SET$ provided that, in addition to the condition imposed in Section \ref{SectionRelative}, $\Mod{\Phi}$ satisfies the following condition which we assume from now on: for each surjective cont(inuous/ractive) map $f$, $f_*$ belongs to $\Mod{\Phi}$. Hence, under these conditions, $\Cocts{\Phi}\cong\SET^{\mPhidiscrete}$ where $\mPhidiscrete$ is the restriction of the monad $\mPhi$ on $\TOP$ respectively $\AP$ to $\SET$. A morphism from $X$ to $Y$ in the Kleisli category $\SET_{\mPhidiscrete}$ is a cont(inuous/ractive) map $X\to \Phi Y$ where we consider $X=(X,e_X^\circ)$ and $Y=(Y,e_Y^\circ)$ with the discrete structure, and it corresponds to a $\mU$-module $X\kmodto Y$ in $\Mod{\Phi}$. We write $\Mat{\Phi}$ for the ordered category of all unitary $\mU$-relations $\varphi:X\krelto Y$ where $\varphi:(X,e_X^\circ)\kmodto(Y,e_Y^\circ)$ belongs to $\Mod{\Phi}$, the composition is Kleisli composition and the order on hom-sets is the pointwise one. Then the morhpisms $\varphi:X\krelto Y$ of $\Mat{\Phi}$ correspond precisely to the morphisms $\mate{\varphi}:Y\to\Phi X$ in $\SET_{\mPhidiscrete}$, and with the help of Example \ref{comp_as_colim} one concludes that the compositional structures match. In conclusion, $\Mat{\Phi}\cong\SET_{\mPhidiscrete}$, even as ordered categories. By definition, $\Mat{\Phi}$ embeds fully into $\Mod{\Phi}$ by considering a set as a discrete space. For a topological/approach space $X=(X,a)$, the convergence relation $a:X\krelto X$ is unitary and idempotent. Furthermore, $a=i^*\kleisli i_*$ where $i:(X,e_X^\circ)\to(X,a),\,x\mapsto x$, hence $a:(X,e_X^\circ)\kmodto(X,e_X^\circ)$ belongs to $\Mod{\Phi}$. From this one obtains a full embedding $\Mod{\Phi}\to\kar(\Mat{\Phi})$, and therefore $\kar(\Mod{\Phi})\cong\kar(\Mat{\Phi})$. From Theorem \ref{TheoremKarDistrib} we infer now that \[\kar(\Mat{\Phi})^\op\cong\DTop{\Phi}_{\sup}/\DApp{\Phi}_{\sup}.\]

For the choice of all $\mU$-modules on topological spaces, the result above tells us that $\CDTOP_{\sup}$ is dually equivalent to $\kar(\UREL)$, where $\UREL$ denotes the ordered category of sets and unitary $\mU$-relations. Hence $\FRM\cong\CDTOP$ is dually equivalent to category $\map(\kar(\UREL))$ defined by the left adjoint morphisms in $\kar(\UREL)$.

\def\cprime{$'$}


\end{document}